\documentclass[10pt,onecolumn]{article}
\usepackage[left=1in,top=1in,right=1in,bottom=1in]{geometry}
\usepackage{graphicx,caption,hyperref,authblk,enumitem,grffile,parskip}
\usepackage{amssymb,amsmath,amsthm,xcolor,mathrsfs}
\usepackage[hang,flushmargin]{footmisc} 
\graphicspath{{./Figures/}}	
\usepackage{transparent}
\usepackage[normalem]{ulem}

\theoremstyle{plain}

\newtheorem{proposition}{Proposition}[section]

\newtheorem{lemma}{Lemma}[section]
\newtheorem{remark}{Remark}[section]

\DeclareMathOperator{\eps}{\varepsilon}
\title{\vspace{-4em}A symmetric mechanism for symmetry-breaking\\ in oscillator networks with strong nonlinear coupling}
\author[1]{Theodore Vo}
\author[2]{Yangyang Wang}

\affil[1]{\footnotesize School of Mathematics, Monash University, Clayton, Victoria 3800, Australia}
\affil[2]{\footnotesize Department of Mathematics, Brandeis University, Waltham, Massachusetts 02453, USA}

\begin{document}
\bibliographystyle{abbrv}
\maketitle

\begin{abstract} \noindent
In this article, we describe and analyse a novel mechanism for symmetry-breaking in minimal symmetrically coupled identical slow/fast oscillator networks with strong nonlinear mutually inhibitory coupling. We show that the symmetry-breaking, surprisingly, originates from the canard dynamics of a folded node that lies on the axis of symmetry. By applying geometric singular perturbation theory and the blow-up technique to a normal form, we determine the geometric mechanisms by which the {\em symmetric folded node} induces symmetry-breaking. More specifically, we show that (i) the fold curve of the coupled system is orthogonal to the axis of symmetry at the symmetric folded node; (ii) there is only one primary maximal canard (either strong or weak, depending on parameters), which always lies on the axis of symmetry and is the axis of rotation for the twisting of solutions; and (iii) the number of rotations is the key local diagnostic feature that breaks the symmetry. 
Our work is closely related to that of Kristiansen and Pedersen [SIAM J. Appl. Dyn. Syst., {\bf 22} (2023)] on symmetrically coupled FitzHugh-Nagumo oscillators with strong linear inhibitory gap junctional coupling, however, we consider nonlinear coupling and we identify and study multiple sub-types of their `cusped singularities'.
We demonstrate our theoretical results by applying them to a model of the eukaryotic cell cycle in which the symmetric folded node plays a key role in rhythmogenesis. More specifically, we study periodic and quasi-periodic symmetry-breaking mixed-mode oscillatory attractors of the cell cycle model. We show that the local twisting induced by the symmetric folded node is the local mechanism that both breaks the symmetry and generates the small-amplitude oscillations in the mixed-mode dynamics.
\end{abstract}

\vspace{1em}


\section{Introduction} \label{sec:intro}
Symmetry-breaking refers to a phenomenon that occurs in models or systems that have certain symmetries, but their solutions do not necessarily share these symmetries. It plays key roles in numerous branches of science, including elementary particle physics \cite{Gross1996}, pattern formation in chemical oscillator systems \cite{Mainzer1997}, transitions between gaits in animal locomotion \cite{Collins1993}, and chimera states in coupled oscillator networks \cite{Abrams2004,Kuramoto2002}.  

In classical deterministic symmetry-breaking, systems of coupled identical oscillators can exhibit asymmetric rhythms that are small perturbations of a symmetric state. Common examples include rhythms that are slightly out of phase or are close to anti-phase. 
Recently, a new type of symmetry-breaking, called {\em strong symmetry-breaking} (SSB), was discovered in minimal chimera systems, i.e., low-dimensional networks of coupled, identical slow/fast oscillators that exhibit chimera-like solutions \cite{Awal2023,Awal2024,Epstein2024}. In this context, SSB occurs when the amplitudes, frequencies, and/or other properties of the coupled, identical oscillators differ substantially, typically by an order of magnitude. Such SSB rhythms are far from any symmetric state. For example, one oscillator can exhibit small-amplitude oscillations (SAOs) whilst another oscillator can exhibit large-amplitude oscillations (LAOs) or mixed-mode oscillations (MMOs), despite the two oscillators being identical and the coupling being symmetric.  

Some of the first examples of SSB rhythms were the SAO-LAO and SAO-MMO rhythms found in coupled, identical planar slow/fast oscillators from chemistry (for coupled Lengyel-Epstein oscillators \cite{Awal2019,Awal2024} and coupled Koper models \cite{Epstein2024}) and engineering (for coupled van der Pol equations \cite{Awal2023}). In these systems, the coupling is diffusive and is modeled by linear terms in the vector field.  Moreover, the coupling strengths are typically small-amplitude so that, to leading order, the geometry of the coupled system can be regarded as a cross-product of the geometries of the individual oscillators; see {\em cross-product quilts} in \cite{Awal2023}.

For these model systems, it was shown that the primary mechanism responsible for the SSB was a non-symmetric (or asymmetric) folded singularity. Asymmetric folded singularities are points in the phase space that stay away from the axis of symmetry, and allow solutions to cross from an invariant slow manifold of one stability to an invariant slow manifold with a different stability. It was shown that the maximal canard solutions associated to the asymmetric folded singularity would determine key diagnostics of the SSB rhythms, such as the number of oscillations each oscillator would make and whether each oscillator would remain small-amplitude or would be directed to make a fast (large-amplitude) jump. 

In this article, we report on a robust, novel SSB mechanism for coupled, identical slow/fast oscillator networks with strong nonlinear coupling. We find that the singularity that underlies the new SSB is a {\em symmetric folded node}. 
Even though the symmetric folded node lies on the axis of symmetry, the canard solutions that it generates are (surprisingly) responsible for symmetry-breaking. 

There are two main contributions in this manuscript. First, we provide rigorous analysis of the novel symmetric folded node and the symmetry-breaking that it induces. 
More specifically, we show that for identical slow/fast oscillators coupled via strong nonlinear mutual inhibition:
\begin{enumerate}[label=(\roman*)]
\setlength{\itemsep}{-2pt}
\item there are two types of symmetric folded node: {\em strong} and {\em weak};
\item the fold curve is orthogonal to the axis of symmetry at the symmetric folded node;
\item every solution that passes near the symmetric folded node is guaranteed to exhibit at least one twist; 
\item the symmetric folded node only has one primary canard, which is aligned with axis of symmetry;
\item the one primary canard is also the axis of rotation for the twisting; 
\item The funnel of the symmetric folded node (strong or weak) corresponds to the entirety of the attracting sheet of the critical manifold.
\item the slow manifolds are partitioned into ribbons and the parity of the ribbons, i.e., the number of twists that the ribbon makes, determines the direction of any subsequent fast jump. 
\end{enumerate}
We also show numerically that the strong symmetric folded node has two types of sequential rotational behaviour: the primary rotation occurs around the axis of symmetry and the secondary rotation occurs around a rotation axis that seems to emerge from the weak eigendirection of the symmetric folded node. 
Our second main contribution in this article is the demonstration of our symmetric folded node theory in a model of the eukaryotic cell cycle \cite{Dragoi2024}, which consists of a pair of identical slow/fast oscillators with strong nonlinear coupling in the fast directions. We show that the symmetric folded node in this system plays a major role in the formation of the SSB MMO attractors, such as those shown in Fig.~\ref{fig:introsymmbreaking}.

\begin{figure}[h!]
  \centering
  \includegraphics[width=5in]{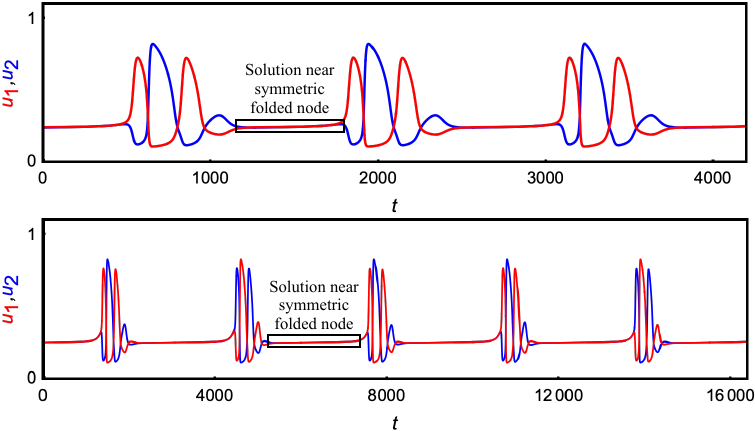}
  \caption{Representative SSB MMO attractors in a model for the eukaryotic cell cycle (see Section~\ref{sec:modelgspt}). (a) Periodic attractor. Oscillator 1 (red) exhibits two large-amplitude spikes per period, whereas oscillator 2 (blue) exhibits one spike per period, the timing of which occurs approximately in between the two spike events of oscillator 1. (b) Quasi-periodic attractor. The number of large-amplitude spikes varies from event to event. In both cases, the symmetry-breaking is determined by the slow local dynamics near a symmetric folded node (boxed regions).}
  \label{fig:introsymmbreaking}
\end{figure} 

To the best of our knowledge, there have been two other studies of canard dynamics in coupled oscillator systems with strong coupling \cite{Kristiansen2023,Pedersen2026}. 
In \cite{Kristiansen2023}, the authors study a symmetrically coupled pair of FitzHugh-Nagumo oscillators with inhibitory gap junctional coupling, i.e., the coupling terms are linear functions of the fast variables. Moreover, the coupling strength is $\mathcal O(1)$ with respect to the small timescale separation parameter $\eps$. They show that their system possesses a symmetric folded singularity, which they call a {\em cusped singularity} since it occurs at a cusp bifurcation of the layer problem. They then use geometric singular perturbation theory and the blow-up technique to study the emergence of small-amplitude oscillations from the neighbourhood of the cusped singularity, and they study the resulting mixed-mode oscillations. 

In \cite{Pedersen2026}, the author extends the results of \cite{Kristiansen2023} to symmetrically coupled oscillator systems with mutually inhibitory coupling of any form but with the added condition that a full system equilibrium is close to the cusped singularity. 
That is, the author considers parameter configurations close to a folded saddle-node bifurcation of type II \cite{Krupa2010}.
The author demonstrates that the proximity of the full system equilibrium to the cusped singularity is responsible for rhythmogenesis in two examples: a two-cell Wilson-Cowan-like neural network with slow adaptation, and a synaptically coupled Morris-Lecar network with fast synaptic inhibition. 

Our work here may be regarded as a direct continuation of \cite{Kristiansen2023} and \cite{Pedersen2026}. 
We concentrate specifically on coupled oscillators with strong nonlinear mutually inhibitory coupling. 
We make no assumptions about the existence of any full system equilibria near the cusped singularity. 
In this setting, we will show that the symmetric folded nodes (i.e., the cusped node singularities of \cite{Kristiansen2023}) can have two distinct types, strong and weak, each with distinctly different dynamics. 
A core focus of our work which differs from \cite{Kristiansen2023,Pedersen2026} is that we focus on how the local dynamics near symmetric folded node singularities induces symmetry breaking in these coupled oscillator networks. 

\medskip
\begin{remark}
A closely related case study is provided in \cite{MayoraCebollero2025}. There, the authors study small networks of Brusselator models symmetrically coupled with strong linear coupling. The main differences in \cite{MayoraCebollero2025} to our work and that of \cite{Kristiansen2023,Pedersen2026} is that (i) the coupled Brusselator system is a non-standard slow/fast system \cite{Wexbook}, (ii) the coupling is present in all components of the vector field, and (iii) they focus on synchronization and almost synchronization. 
\end{remark}

The outline of the manuscript is as follows. 
In Section~\ref{sec:normalform}, we develop a normal form for the symmetric folded node in symmetrically coupled, identical slow/fast oscillators with strong mutually inhibitory nonlinear coupling. We use geometric singular perturbation theory to uncover the geometry of the system and identify key features of the symmetric folded node. 
Then, in Section~\ref{sec:blowup} we perform a blow-up analysis of the symmetric folded node to rigorously determine the properties of the slow manifolds and their rotational behaviour.  
In Section~\ref{sec:ribbons}, we demonstrate the rotational properties of the symmetric folded node by computing the invariant slow manifolds. In particular, we show that there is a {\em ribboning} of the slow manifolds (cf. \cite{Hasan2017,Hasan2018}), and the parity of the ribbons determines the direction of the symmetry breaking, i.e., it determines which of the coupled oscillators is most likely to exhibit a fast transition away from the neighbourhood of the symmetric folded node. 
We then apply our results to a canonical model for the cell cycle control system in Section~\ref{sec:modelgspt}, and show that the dynamics about the symmetric folded node are responsible for the SSB rhythms in Section~\ref{sec:cellmodelsymmbreaking}.  
Finally, we conclude in Section~\ref{sec:discussion} with a discussion.

\section{The Symmetric Folded Node} \label{sec:normalform}
Motivated by dual oscillator models from biology (see Section~\ref{sec:modelgspt}), we study minimal networks of identical slow/fast oscillators with mutually inhibitory symmetric coupling of the form
\begin{equation} \label{eq:coupled_general}
  \begin{split}
    \dot u_1 &= f(u_1,v_1) - c u_1 u_2, \\ 
    \dot v_1 &= \eps g(u_1,v_1), \\ 
    \dot u_2 &= f(u_2,v_2) - c u_1 u_2, \\ 
    \dot v_2 &= \eps g(u_2,v_2),
  \end{split}
\end{equation}
where $f$ and $g$ are sufficiently smooth functions, $0<\eps \ll 1$ encodes the timescale separation, and the coupling strength $c$ is $\mathcal{O}(1)$ with respect to $\eps$. 
In this section, we develop a normal form for symmetric folded nodes, which are folded node singularities that occur at fast subsystem cusp bifurcations (Section~\ref{subsec:normalform}).  
We then use geometric singular perturbation theory \cite{Fenichel1979,Jones1995} to identify the key geometric features of the symmetric folded node (Section~\ref{subsec:normalformGSPT}). 
Subsequently, we apply center manifold reduction to study the reduced dynamics near the symmetric folded node (Section~\ref{subsec:cmreduction}). 

\subsection{Normal form}  \label{subsec:normalform}
The system \eqref{eq:coupled_general} is $\mathbb Z_2$-symmetric under the exchange of the two identical oscillators. More precisely, there is a matrix, $\mathcal R$, given by
\[ \mathcal R = \begin{bmatrix} 0 & 0 & 1 & 0 \\ 0 & 0 & 0 & 1 \\ 1 & 0 & 0 & 0 \\ 0 & 1 & 0 & 0 \end{bmatrix}, \]
such that the vector field is equivariant under this symmetry, i.e.,
\begin{equation} \label{eq:symmetry}
  \begin{split}
    \mathcal R \boldsymbol F(\boldsymbol u) = \boldsymbol F(\mathcal R \boldsymbol u),
  \end{split}
\end{equation}
where $\boldsymbol u = (u_1,v_1,u_2,v_2)^T$ denotes the state vector and $\boldsymbol F$ denotes the vector field in \eqref{eq:coupled_general}.  
Consequently, the set $\mathcal L_s = \left\{  u_1=u_2, v_1=v_2 \right\}$ is the fixed-point subspace of the map $\mathcal R$, and hence it is an invariant subspace of the system \eqref{eq:coupled_general}. We refer to $\mathcal L_s$ as the axis of symmetry (even though it is a 2D subspace). Along $\mathcal L_s$, the system \eqref{eq:coupled_general} reduces to the {\em symmetric subsystem}
\begin{equation} \label{eq:symmetricsubsystem}
  \begin{split}
    \dot u &= f(u,v) - c u^2, \\ 
    \dot v &= \eps g(u,v),
  \end{split}
\end{equation}
which is a standard slow/fast system with fast variable $u$ and slow variable $v$.

{\bf Assumption 1:} The symmetric subsystem \eqref{eq:symmetricsubsystem} possesses a non-degenerate fold point, i.e., there exists a point $(u_0,v_0)$ with $\lVert (u_0,v_0) \rVert = \mathcal{O}(1)$ with respect to $\eps$ such that 
\begin{equation}
  \begin{split}
    f(u_0,v_0)-c u_0^2 &= 0, \quad 
    \frac{\partial f}{\partial u}(u_0,v_0) = 0, \\ 
    \frac{\partial f}{\partial v}(u_0,v_0) &\neq 0, \quad
    \frac{\partial^2 f}{\partial u^2}(u_0,v_0) \neq 0.
  \end{split}
\end{equation}

The conditions in Assumption~1 imply that there exists a neighbourhood of the point $(u_0,v_0)$ such that the $u$-nullcline is locally parabolic.

{\bf Assumption 2:} The point $(u_0,v_0)$ is a regular fold point, i.e.,
\[ g(u_0,v_0) \neq 0. \]

Assumption~2 implies that the slow flow along the critical manifold of the symmetric subsystem \eqref{eq:symmetricsubsystem} is singular at the fold point $(u_0,v_0)$, see \cite{Krupa2001}.

\medskip

\begin{proposition} \label{prop:normalform}
Given the coupled oscillator system \eqref{eq:coupled_general} satisfying Assumptions~1 and ~2, there exists a smooth change of coordinates which transforms the system \eqref{eq:coupled_general} to the form 
\begin{equation} \label{eq:normalform}
  \begin{split}
    \dot u_1 &= -u_1-u_2+v_1+u_1^2-\left( 1+\tfrac{1}{2}\lambda_1 \right) u_1 v_2 -a (u_1 v_2+u_2 v_1) + P_1(u_1,u_2,v_1,v_2,\eps), \\
    \dot u_2 &= -u_2-u_1+v_2+u_2^2  - \left( 1+\tfrac{1}{2}\lambda_1 \right) u_2 v_1 -a (u_2 v_1+u_1 v_2)+ P_2(u_2,u_1,v_2,v_1,\eps), \\
    \dot v_1 &= \eps \left[ 1+\tfrac{1}{2} (\lambda_1+\lambda_2) u_1+ b u_1^2 + \tfrac{1}{4} a (\lambda_1+\lambda_2)u_2^2 + Q_1(u_1,u_2,v_1,v_2,\eps) \right], \\ 
    \dot v_2 &= \eps \left[ 1+\tfrac{1}{2} (\lambda_1+\lambda_2) u_2 + b u_2^2+ \tfrac{1}{4} a \left( \lambda_1+\lambda_2 \right) u_1^2  + Q_2(u_2,u_1,v_2,v_1,\eps) \right],
  \end{split}
\end{equation}
in a neighbourhood of the origin, where $(\lambda_1,\lambda_2,a,b) \in \mathbb R^4$ are computable constants and the overdot denotes differentiation with respect to $t$. The higher order terms are given by 
\begin{equation*}
  \begin{split}
    P_1(u_1,u_2,v_1,v_2,\eps) &= p u_1^2 u_2 + \tfrac{5a+2a^2+p}{2} u_1 u_2^2 + k \eps u_1 + \mathcal O\left( (u_1+u_2) v_1 v_2,  \eps u_1 v_2  \right), \\
    Q_1(u_1,u_2,v_1,v_2,\eps) &= \mathcal O\left( u_1 (v_1+v_2), u_2 (v_1+v_2), \eps^2 \right),
  \end{split}
\end{equation*}
where $(p,k) \in \mathbb R^2$ are computable constants, and the vector field obeys the symmetry \eqref{eq:symmetry}, so that 
\[ P_2(u_1,u_2,v_1,v_2,\eps) = P_1(u_2,u_1,v_2,v_1,\eps) \quad \text{ and } \quad Q_2(u_1,u_2,v_1,v_2,\eps) = Q_1(u_2,u_1,v_2,v_1,\eps). \] 
%
%
%
\end{proposition}

\begin{proof}
In the first step, we translate the fold point $(u_0,v_0)$ to the origin. 
Then, we Taylor expand the vector field in \eqref{eq:coupled_general} about the origin $(u_1,v_1,u_2,v_2,\eps)=(0,0,0,0,0)$.
Finally, the system \eqref{eq:normalform} is obtained via a sequence of restricted near-identity transformations that preserve both the slow/fast structure and the reflection symmetry $\mathcal R$, and are designed to eliminate non-resonant terms from the vector field. 
\end{proof}

We will show that the normal form \eqref{eq:normalform} possesses a folded singularity \cite{Szmolyan2001} at the origin, and that the coefficients $\lambda_1$ and $\lambda_2$ determine its type.
We will also demonstrate that the canard dynamics of the folded singularity are determined at leading order by the parameters $\lambda_1$ and $\lambda_2$, whereas the terms associated with the parameters $a$ and $b$ affect the canard dynamics at higher order. 

\medskip
\begin{remark}
The normal form \eqref{eq:normalform} is not unique. The dominant coupling terms in the fast $u_1$ and $u_2$ equations of \eqref{eq:normalform} are the mixed $u_1 v_2$ and $u_2 v_1$ terms. An alternative choice of the sequence of near-identity transformations in the proof of Proposition~\ref{prop:normalform} can be used to remove these mixed $u_1 v_2$ and $u_2 v_1$ terms. The trade-off, however, is that terms involving $u_1 v_1, u_2 v_1, u_1 v_2$, and $u_2 v_2$ are re-introduced and become irremovable.  Moreover, the coefficients of these 4 lowest-order coupling terms have the form $(\alpha \pm \tfrac{1}{2}\lambda_1)$, where $\lambda_1$ is again an eigenvalue of the folded singularity and $\alpha$ is a system parameter that only affects the canard dynamics at higher order. 
\end{remark}

\subsection{Geometric Singular Perturbation Analysis} \label{subsec:normalformGSPT}
The normal form \eqref{eq:normalform} is a slow/fast system with fast variables $(u_1,u_2)$ and slow variables $(v_1,v_2)$. To analyse the system, we use techniques from geometric singular perturbation theory \cite{Fenichel1979,Jones1995}. The essence of the approach is to split the dynamics into the slow motions along a key invariant manifold, $S$, and the fast motions in the complementary space $\mathbb R^4 \setminus S$.

\subsubsection{Layer problem, critical manifold, \& slow manifolds}
Setting $\eps = 0$ in the normal form \eqref{eq:normalform} yields the 2D layer problem
\begin{equation} \label{eq:layernf}
  \begin{split}
    \dot u_1 &= -u_1-u_2+v_1+u_1^2-\tfrac{2+2a+\lambda_1}{2} u_1 v_2 - a u_2 v_1 + p u_1^2 u_2 + \tfrac{5a+2a^2+p}{2} u_1 u_2^2 + \mathcal O\left( (u_1+u_2) v_1 v_2 \right), \\
    \dot u_2 &= -u_1-u_2+v_2+u_2^2 -a u_1 v_2 - \tfrac{2+2a+\lambda_1}{2} u_2 v_1 + \tfrac{5a+2a^2+p}{2} u_1^2 u_2 + p u_1 u_2^2 + \mathcal O\left( (u_1+u_2) v_1 v_2 \right),
  \end{split}
\end{equation}
which approximates the fast dynamics in the limit of infinitely slow motions, i.e., the slow variables $v_1$ and $v_2$ are fixed parameters. 
The critical manifold, $S$, is the 2D set of equilibria of the layer problem
\[ S = \left\{ (u_1,u_2,v_1,v_2) \in \mathbb R^4 : v_1= v_{1S}(u_1,u_2), \,\, v_2 = v_{2S}(u_1,u_2) \right\}, \]
where the function $v_{1S}$ is given by 
\begin{equation*}
  \begin{split}
    v_{1S} &= u_1+u_2+\tfrac{2a+\lambda_1}{2}u_1^2 + \tfrac{4a+2+\lambda_1}{2}u_1 u_2 + a u_2^2+\gamma_{30} u_1^3 + \gamma_{21} u_1^2 u_2 + \gamma_{12} u_1 u_2^2 + \gamma_{03} u_2^3+\mathcal O \left( (u_1+u_2)^4 \right), 
  \end{split}
\end{equation*}
and $v_{2S}(u_1,u_2) = v_{1S}(u_2,u_1)$ by the symmetry $\mathcal R$.
Here, the coefficients of the cubic terms are given by
\begin{equation*}
  \begin{split}
\gamma_{30} &= \tfrac{a(2+2a+\lambda_1)}{2}, \\
\gamma_{21} &= 1+3a(1+a)-p+\lambda_1+2a \lambda_1 + \tfrac{1}{4}\lambda_1^2, \\
\gamma_{12} &= \tfrac{1}{4} \left( 8a^2-2p+\lambda_1(2+\lambda_1)+a(6\lambda_1-2) \right), \\
\gamma_{03} &= a^2.
  \end{split}
\end{equation*}

The linear stability of $S$ is determined by the eigenvalues of the Jacobian of the layer problem, which has trace and determinant given by 
\begin{equation*}
  \begin{split}
    \left. \operatorname{tr} D_u f \right|_S &= -2 -\left(2a+\lambda_1\right) (u_1+u_2) + \mathcal{O}((u_1+u_2)^2), \\ 
    \left. \det D_u f \right|_{S} &= \lambda_1 (u_1+u_2) + \rho_1 u_1^2 + \rho_1 u_2^2 + \rho_2 u_1 u_2 + \mathcal{O}((u_1+u_2)^3),
  \end{split}
\end{equation*}
where the coefficients $\rho_1$ and $\rho_2$ are given by 
\[ \rho_1 = \tfrac{1}{2} \left( -2-a(1+2a)+p+\lambda_1+4a \lambda_1+\lambda_1^2 \right) \quad \text{ and } \quad \rho_2 = 4+4a^2-2p+\lambda_1(2+\lambda_1)+2a(7+2\lambda_1). \]
Sufficiently close to the origin, the trace is always negative. Thus, in a neighbourhood of the origin, the critical manifold can be partitioned as
\[ S = S_a \cup S_L \cup S_s, \]
where the attracting sheet
\[ S_a = \left\{ \left. (u_1,v_1,u_2,v_2) \in S : \det D_u f \right|_{S} > 0 \right\}, \] 
consists of stable nodes of the layer problem \eqref{eq:layernf}, the saddle sheet 
\[ S_s = \left\{ \left. (u_1,v_1,u_2,v_2) \in S : \det D_u f \right|_{S} < 0 \right\}, \]
consists of saddle equilibria of \eqref{eq:layernf}, and the fold curve
\[ S_L = \left\{ \left. (u_1,v_1,u_2,v_2) \in S : \det D_u f \right|_{S} = 0 \right\}, \] 
consists of saddle-node bifurcations of the layer problem. The fold curve can be written locally as a graph:
\begin{equation}  \label{eq:foldcurvegraph}
S_L = \left\{ (u_1,v_1,u_2,v_2) \in S : u_2 = -u_1 + \tfrac{6+15a+6a^2-3p+\lambda_1}{\lambda_1}u_1^2 + \mathcal O \left( u_1^3 \right) \right\}.
\end{equation}
This graph representation of $S_L$ reveals an important geometric feature of the mutually inhibitorily coupled system \eqref{eq:normalform}, which we state in general terms below.

\medskip

\begin{lemma} \label{lemma:orthogonality}
Suppose that the coupled oscillator system \eqref{eq:coupled_general} satisfying Assumption~1 has a regular fold point, $\boldsymbol p \in \mathcal L_s$, of the symmetric subsystem. Then (i) $\boldsymbol p$ is a fold point of the full coupled system \eqref{eq:coupled_general}, and (ii) the fold curve $S_L$ of the coupled system is orthogonal to the axis of symmetry $\mathcal L_s$ at $\boldsymbol p$. 
\end{lemma}

\begin{proof}
The Jacobian of the layer problem of \eqref{eq:coupled_general}, evaluated at a point $\boldsymbol p = (u_0,v_0,u_0,v_0) \in \mathcal L_s$, is given by 
\[ \left. D_u f \right|_{\boldsymbol p} = \begin{bmatrix} \tfrac{\partial f}{\partial u} - c u_0 & -c u_0 \\ -c u_0 & \tfrac{\partial f}{\partial u} - c u_0 \end{bmatrix}, \]
where $\tfrac{\partial f}{\partial u} = \left. \tfrac{\partial f(u_1,v_1)}{\partial u_1} \right|_{\boldsymbol p} = \left. \tfrac{\partial f(u_2,v_2)}{\partial u_2} \right|_{\boldsymbol p}$. Since this matrix is symmetric, we immediately have that the eigenvalues are  
\[ \lambda_1 = \tfrac{\partial f}{\partial u} - 2 c u_0 \quad \text{ and } \quad \lambda_2 = \tfrac{\partial f}{\partial u}. \]
The corresponding eigenvectors, 
\[ \boldsymbol w_1 = \begin{bmatrix} 1 \\ 1 \end{bmatrix} \quad \text{ and } \quad \boldsymbol w_2 = \begin{bmatrix} -1 \\ 1 \end{bmatrix},\]
are aligned with $\mathcal L_s$, and orthogonal to $\mathcal L_s$, respectively. 
By Assumption~1, the derivative $\tfrac{\partial f}{\partial u}$ vanishes at a regular fold point of the symmetric subsystem. Hence, the eigenvalue $\lambda_2$ vanishes at a regular fold point of the symmetric subsystem. Consequently, the regular fold point of the symmetric subsystem must also be an element of the fold curve $S_L$ of the fully coupled system. Moreover, the eigenvector $\boldsymbol w_2$ corresponding to the zero eigenvalue is orthogonal to $\mathcal L_s$. 
\end{proof}

\begin{figure}[h!]
  \centering
  \includegraphics[width=3.5in]{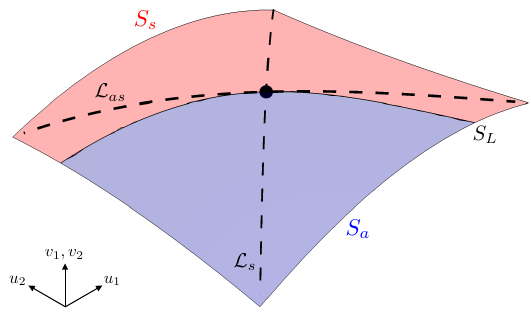}
  \caption{Critical manifold, $S$, of the layer problem in a neighbourhood of the origin (black marker) for representative values of $\lambda_1, a$ and $b$. The (blue) attracting sheet, $S_a$, and (red) saddle sheet, $S_s$, are separated by the (solid black) fold curve, $S_L$. 
  Also shown are the intersections of the symmetry and anti-symmetry axes, $\mathcal L_s$ and $\mathcal L_{as}$, with the critical manifold (dashed black curves). 
  In this class of coupled oscillator systems, $S_L$ is orthogonal to $\mathcal L_{s}$.}
  \label{fig:criticalmanifold}
\end{figure}

Thus, from both the graph representation of the fold curve $S_L$ and Lemma~\ref{lemma:orthogonality}, we have that the fold curve is tangent at the origin to the anti-symmetry axis, $\mathcal L_{as} := \left\{ u_1 = -u_2, v_1 = -v_2 \right\}$, and is orthogonal to $\mathcal L_s$.  The geometry of the layer problem \eqref{eq:layernf} is shown in Fig.~\ref{fig:criticalmanifold}.

\begin{remark}
In prior studies of SSB \cite{Awal2023,Awal2024,Epstein2024}, the fold curve was transverse and non-orthogonal to both $\mathcal L_s$ and $\mathcal L_{as}$. The reason for this is that the coupling strengths in these prior studies is weak, i.e., $\eps$-dependent. Consequently, the singular limit $\eps \to 0$ decouples the oscillators and the critical manifold of the coupled system is simply the cross-product of the critical manifolds of the individual oscillators.  Moreover, the fold curve of the coupled system is a cross-product of the fold points of the individual oscillators, i.e., the fold curve takes the form
\[ S_L = \left\{ u_1 = u_L,\,\, u_2 \in \mathbb R \right\} \cup \left\{ u_1 \in \mathbb R,\,\, u_2 = u_L \right\}, \]
where $u_L$ is the $u$-coordinate of a fold point of the uncoupled oscillator. Thus, for coupled oscillator systems with weak $\eps$-dependent coupling strengths, $S_L$ consists of lines parallel to the $u_1$- and $u_2$-axes. Hence, the intersection of $S_L$ with $\mathcal L_{s}$ is both transverse and non-orthogonal. 
\end{remark}

Fenichel's first theorem \cite{Fenichel1979,Jones1995} guarantees that the normally hyperbolic subsets of $S$, i.e., the parts of $S_a$ and $S_s$ which have spectra uniformly bounded away from the imaginary axis, will persist for $\eps$ sufficiently small and positive as invariant slow manifolds $S_{a,\eps}$ and $S_{s,\eps}$. These slow manifolds are $\mathcal O(\eps)$ perturbations of the critical manifold. Moreover, the local stable and unstable manifolds of $S_a$ and $S_s$ also perturb to $\mathcal O(\eps)$-close local stable and unstable manifolds of $S_{a,\eps}$ and $S_{s,\eps}$.

\subsubsection{Reduced flow and ordinary singularities}
The layer problem provides a good approximation of the dynamics of \eqref{eq:normalform} for the regions of phase space away from the critical manifold, i.e., in $\mathbb R^4 \setminus S$. The slow motions on $S$, however, require an alternative description. To obtain such a description, we switch to the slow time $\tau = \eps t$ in \eqref{eq:normalform} and take the singular limit $\eps \to 0$. This gives the 2D reduced problem
\begin{equation} \label{eq:reducednf}
  \begin{split}
    0 &= F_1(u_1,v_1,u_2,v_2,0), \\ 
    0 &= F_2(u_1,v_1,u_2,v_2,0), \\ 
    \dot v_1 &= G_1(u_1,v_1,u_2,v_2,0), \\ 
    \dot v_2 &= G_2(u_1,v_1,u_2,v_2,0),
  \end{split}
\end{equation}
where $F_i$ and $G_i$ denote the right-hand-sides of the $u_i$ and $v_i$ equations in \eqref{eq:normalform} for $i=1,2$. The overdot denotes differentiation with respect to $\tau$.
The algebraic equations in \eqref{eq:reducednf} constrain the phase space to the critical manifold, and the differential equations describe the slow motions along the critical manifold. 

To obtain a complete description of the slow dynamics in the full atlas of overlapping coordinate charts, we differentiate the algebraic constraints in \eqref{eq:reducednf} and solve for $\dot u_1$ and $\dot u_2$, which gives
\begin{equation}  \label{eq:reducednfatlas}
  \begin{split}
    \det D_u F \begin{bmatrix} \dot u_1 \\ \dot u_2 \end{bmatrix} &= -\operatorname{adj}D_u F \, \cdot  D_v F \cdot \, \begin{bmatrix} G_1 \\ G_2 \end{bmatrix}, \\ 
    \begin{bmatrix} \dot v_1 \\ \dot v_2 \end{bmatrix} &= \begin{bmatrix} G_1 \\ G_2 \end{bmatrix},
  \end{split}
\end{equation}
where $D_{\xi} F$ denotes the Jacobian of $F = (F_1, F_2)^T$ with respect to $\xi = (\xi_1, \xi_2)^T$ for $\xi \in \{ u,v \}$, $\operatorname{adj}$ denotes the classical adjoint, and all functions and derivatives are evaluated along $S$. 
In this formulation, we see that the reduced flow is singular along the fold set $S_L$, where $\det D_u F = 0$. To remove the finite-time blow-up of the vector field, we desingularize the system via the phase-space-dependent time transformation $dt_d = (\det D_u F)\,d\tau$, which gives the desingularized reduced system
\begin{equation} \label{eq:desing}
  \begin{split}
    \begin{bmatrix} \dot u_1 \\ \dot u_2 \end{bmatrix} &= -\operatorname{adj}D_u F \, \cdot  D_v F \cdot \, \begin{bmatrix} G_1 \\ G_2 \end{bmatrix}, \\ 
    \begin{bmatrix} \dot v_1 \\ \dot v_2 \end{bmatrix} &= \det D_u F  \begin{bmatrix} G_1 \\ G_2 \end{bmatrix},
  \end{split}
\end{equation}
where the overdot now denotes derivatives with respect to $t_d$. On the attracting sheets $S_a$ where $\det D_u F > 0$, the desingularized system \eqref{eq:desing} is topologically equivalent to the reduced system \eqref{eq:reducednf}. On the saddle sheets $S_s$ where $\det D_u F <0$, the flows of the reduced and desingularized systems have opposite orientation.

Ordinary singularities of the desingularized system,
\[ E = \left\{ (u_1,u_2,v_1,v_2) \in S : G_1  = 0 \,\, \text{ and } \,\, G_2 = 0 \right\}, \]
are isolated points in the phase space that correspond to equilibria of both the reduced and desingularized flows, \eqref{eq:reducednfatlas} and \eqref{eq:desing}. Moreover, they are the leading-order estimates of equilibria of the full system \eqref{eq:normalform}. (In systems where the functions $F$ and $G$ are independent of $\eps$, the ordinary singularities are precisely the equilibria of the full system.) We distinguish between {\em symmetric ordinary singularities}, $E_{\rm sym} := E \cap \mathcal L_s$, and {\em asymmetric ordinary singularities}, $E_{\rm asym} := E \setminus E_{\rm sym}$.

Fenichel's second theorem \cite{Fenichel1979,Jones1995} guarantees that the slow flow on the invariant slow manifolds are $\mathcal O(\eps)$ perturbations of the reduced flow on the critical manifold. That is, the singular limit $\eps \to 0$ of the slow flow on $S_{a,\eps}$ and $S_{s,\eps}$ will converge to the flow produced by system \eqref{eq:reducednfatlas}. A consequence of this is that if $E$ is a hyperbolic ordinary singularity of \eqref{eq:reducednfatlas}, then it will persist as a hyperbolic equilibrium of the full system for $\eps$ sufficiently small.

\subsubsection{Symmetric folded singularities of coupled oscillator systems}
The desingularized system possesses another important class of singularities. Folded singularities, $M$, are points along the fold set $S_L$ where the right-hand-side of \eqref{eq:desing} vanishes, i.e.,
\[ M = \left\{ (u_1,u_2,v_1v_2) \in S : \det D_u F = 0 \,\, \text{ and } \,\, \operatorname{adj}D_u F \, \cdot  D_v F \cdot \, \begin{bmatrix} G_1 \\ G_2 \end{bmatrix} = 0 \right\}. \]
Note that since $\det D_u F = 0$ at a folded singularity, the rows of $D_u F$ are linearly dependent and the vector equation in $M$ can be reduced to a single scalar equation. Hence, the set of folded singularities consists of isolated points in the phase space along which singular points of the reduced flow may become regular (since there is possibly a cancellation of a simple zero in the $u$-equations of \eqref{eq:reducednfatlas}). We distinguish between {\em symmetric folded singularities}, $M_{\rm sym} := M \cap \mathcal L_s$, and {\em asymmetric folded singularities}, $M_{\rm asym} := M \setminus M_{\rm sym}$.

Substituting the expressions for $F$ and $G$ from \eqref{eq:normalform} and recalling that along $S$, we have $v_1=v_{1S}(u_1,u_2)$ and $v_2=v_{2S}(u_1,u_2)$, the desingularized reduced system is 
\begin{equation} \label{eq:desingnf}
  \begin{split}
    \dot u_1 &= \tfrac{\lambda_1+\lambda_2}{2} u_1 + \tfrac{\lambda_1 - \lambda_2}{2} u_2 + \xi_{20} u_1^2 + \xi_{02} u_2^2 + \xi_{11} u_1 u_2 + \mathcal O \left( (u_1+u_2)^3 \right), \\ 
    \dot u_2 &= \tfrac{\lambda_1-\lambda_2}{2} u_1 + \tfrac{\lambda_1+\lambda_2}{2} u_2 + \xi_{02} u_1^2 + \xi_{20} u_2^2 + \xi_{11} u_1 u_2 + \mathcal O \left( (u_1+u_2)^3 \right),
  \end{split}
\end{equation}
where the coefficients $\xi_{20}, \xi_{02}$ and $\xi_{11}$ are given by
\begin{equation*} 
  \begin{split}
	\xi_{20} &= \tfrac{1}{4} \left( -4a^2+2p+a(\lambda_2-\lambda_1-14)+(2+\lambda_1)(\lambda_1+\lambda_2-2)+4b \right), \\
	\xi_{02} &= \tfrac{1}{4} \left( 12a+\lambda_1 a -\lambda_2 a - 4b \right), \\ 
	\xi_{11} &= \tfrac{1}{4} \left( 8+4a(7+2a)-4p-2\lambda_2+\lambda_1(2+\lambda_1+\lambda_2) \right).
  \end{split}
\end{equation*}
The desingularized system \eqref{eq:desingnf} possesses a symmetric folded singularity, $M$, at the origin. 
The associated eigenvalues and eigenvectors are given by 
\[ \lambda=\lambda_1, \,\, \boldsymbol w_1 = \begin{bmatrix} 1 \\ 1 \end{bmatrix} \quad {\rm and } \quad \lambda = \lambda_2, \,\, \boldsymbol w_2 = \begin{bmatrix} -1 \\ 1 \end{bmatrix}. \]
That is, the linear subspace associated with the eigenvalue $\lambda_1$ is aligned with the axis of symmetry $\mathcal L_s$.
Similarly, the linear subspace associated with the eigenvalue $\lambda_2$ is aligned with the anti-symmetry axis $\mathcal L_{as}$, which is orthogonal to $\mathcal L_s$ and tangent to $S_L$ at the origin.  

{\bf Assumption 3:} The symmetric folded singularity is a folded node, i.e.,
\[ \lambda_1 <0 \quad \text{ and } \quad \lambda_2 < 0. \]
%

We distinguish two subcases: (i) the {\em strong symmetric folded node} for which $\lambda_1<\lambda_2<0$, and (ii) the {\em weak symmetric folded node} for which $\lambda_2<\lambda_1<0$. 

\medskip \textbf{Strong symmetric folded node ($\boldsymbol{\lambda_1<\lambda_2<0}$)} \medskip \\
For the desingularized system \eqref{eq:desingnf}, the strong stable eigenvalue is $\lambda_1$ and the singular strong canard, $\gamma_0$, is the unique solution of \eqref{eq:desingnf} tangent at the origin to the strong stable eigenvector $\boldsymbol w_1$, which lies along the projection of $\mathcal L_s$ onto the $(u_1,u_2)$-plane.
Thus, $\gamma_0$ coincides with the axis of symmetry. Moreover, the basin of attraction (also known as the funnel) of the strong symmetric folded node (SFN) is the region of $S_a$ enclosed by $\gamma_0$ and the fold curve $S_L$. 
Due to the symmetry, the funnel of the SFN is (locally) all of $S_a$.
The singular strong canard $\gamma_0$ is also a solution of the reduced system \eqref{eq:reducednfatlas} that connects the attracting manifold $S_a$ to the  saddle manifold $S_s$. 

For the desingularized system \eqref{eq:desingnf}, the weak stable eigenvalue is $\lambda_2$ and the weak stable eigenvector $\boldsymbol w_2$ is aligned with the anti-symmetry axis, i.e., $\boldsymbol w_2$ is tangent to the fold curve $S_L$ at the SFN. 
(Recall from \eqref{eq:foldcurvegraph} that the fold curve also has a graph representation $u_2=u_{2L}(u_1)$.)
Let $\gamma_w$ be a solution of \eqref{eq:desingnf} that approaches the origin along the weak eigendirection $\boldsymbol w_2$. 
Then, $\gamma_w$ and the fold curve $S_L$ have the same tangent lines at the SFN, i.e., 
\[ \lim_{t_d \to \infty} \left. \frac{du_2}{d u_1} \right|_{\gamma_w} = \lim_{t_d \to \infty} \left. \frac{\dot u_2}{\dot u_1} \right|_{\gamma_w} = -1 
\quad \text{ and } \quad
\left. \frac{du_{2L}}{du_1}  \right|_{u_1=0} = -1. \]
Moreover, for most parameter values, the curvatures $\kappa(\gamma_w)$ and $\kappa(S_L)$ of $\gamma_w$ and $S_L$ do not coincide at the SFN. More specifically, 
\begin{equation*} 
  \begin{split}
	\lim_{t_d \to \infty}\kappa\left( \gamma_w\right) &= \lim_{t_d \to \infty} \left. \frac{d^2u_2}{d u_1^2} \right|_{\gamma_w} = \lim_{t_d \to \infty} \left. \frac{\dot u_1 \ddot u_2 - \dot u_2 \ddot u_1}{\dot u_1^3} \right|_{\gamma_w}  = \tfrac{2}{\lambda_2^2}(b+\lambda_2)(-6-15a-6a^2+2p-\lambda_1+2\lambda_2), \\
	\kappa\left(S_L\right) &= \left. \frac{d^2}{du_1^2} u_{2L} \right|_{u_1=0} = \tfrac{2}{\lambda_1} \left( 6+15a+6a^2-3p+\lambda_1 \right).
  \end{split}
\end{equation*}
Together, the slope and curvature calculations show that $\gamma_w$ has a quadratic tangency with the fold curve $S_L$ at the origin.
Thus, depending on parameters, the solution $\gamma_w$ of the desingularized system corresponding to the weak stable eigenvector approaches the origin either from the attracting sheet or from the saddle sheet, and it touches the fold curve at the origin tangentially. Thus, $\gamma_w$ is a solution of the reduced problem \eqref{eq:reducednfatlas} that stays strictly on $S_a$ or stays strictly on $S_s$. Hence, $\gamma_w$ is not a canard solution. 


The phase plane of the strong SFN is shown in Fig.~\ref{fig:reduced}(a).

\begin{figure}[h!]
  \centering
  \includegraphics[width=5in]{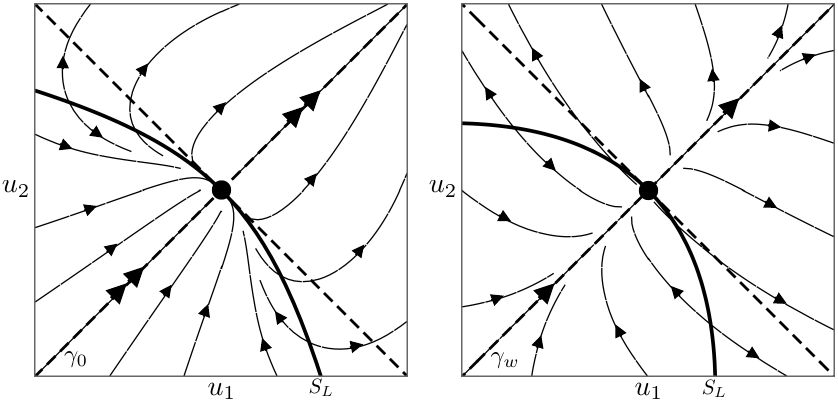}
  \put(-361,164){(a)}
  \put(-177,164){(b)}
  \caption{Reduced flow \eqref{eq:reducednfatlas} of the SFN in a neighbourhood of the origin. The diagonal and anti-diagonal (dashed lines) correspond to the symmetry and anti-symmetry axes, $\mathcal L_s$ and $\mathcal L_{as}$, respectively, onto the $(u_1,u_2)$ plane. The region below (resp., above) the fold curve $S_L$ corresponds to the attracting (resp., saddle) sheet of the critical manifold. There is only one singular canard solution, which coincides with $\mathcal L_s$. (a) For the strong SFN, the singular canard is the strong canard $\gamma_0$. (b) For the weak SFN, the singular canard is the weak canard $\gamma_w$.}
  \label{fig:reduced}
\end{figure}

\medskip \textbf{Weak symmetric folded node ($\boldsymbol{\lambda_2<\lambda_1<0}$)} \medskip \\
For the desingularized system \eqref{eq:desingnf}, the strong stable eigenvalue is $\lambda_2$ and the strong stable eigendirection $\boldsymbol w_2$ is aligned with the anti-symmetry axis. Moreover, the strong stable direction $\boldsymbol w_2$ is tangent to the fold curve. By the same arguments as in the strong SFN case, the solution $\gamma_0$ of the desingularized system corresponding to the strong stable direction stays strictly on one sheet of the critical manifold (either $S_a$ or $S_s$, depending on parameters). Thus, $\gamma_0$ does not correspond to a canard solution. 

For the desingularized system \eqref{eq:desingnf}, the weak stable eigenvalue is $\lambda_1$ and the weak stable eigendirection $\boldsymbol w_1$ is aligned with the axis of symmetry. Consequently, the solution $\gamma_w$ of the desingularized system \eqref{eq:desingnf} corresponding to $\boldsymbol w_1$ coincides with the axis of symmetry. Moreover, it is a solution of the reduced problem \eqref{eq:reducednfatlas} that connects the attracting and saddle sheets of the critical manifold. Hence, $\gamma_w$ is a canard solution of the reduced problem \eqref{eq:reducednfatlas}. 
The funnel of the SFN in this case is the region of $S_a$ enclosed by the fold curve. That is, the funnel corresponds to the entirety of $S_a$. 

The phase plane of the weak SFN is shown in Fig.~\ref{fig:reduced}(b).

\medskip
\begin{remark}
The symmetric folded node here is precisely the `cusped node singularity' of \cite{Kristiansen2023}. More specifically, the cusped node singularity studied in \cite{Kristiansen2023} is a weak SFN of the coupled FitzHugh-Nagumo system with linear repulsive coupling. (A strong SFN is also possible in this system, however, the authors in \cite{Kristiansen2023} restricted their parameter space to focus on the weak SFN and the folded saddle-node bifurcation that occurs when an ordinary singularity passes through it.) 

The cusp structure of the critical manifold in the normal form \eqref{eq:normalform} is not immediately obvious in the original $(u_1,v_1,u_2,v_2)$ coordinates. It becomes evident, however, when the critical manifold is viewed in the $(\hat v_1,\hat v_2, u_2)$ coordinates, where $\hat v_1 = \tfrac{1}{2} \left( v_1+v_2 \right)$ and $\hat v_2 = \tfrac{1}{2} \left( v_1-v_2 \right)$; see Fig.~\ref{fig:cusp}.

\begin{figure}[h!]
  \centering
  \includegraphics[width=3in]{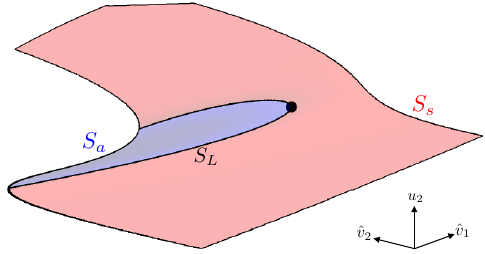}
  \caption{Critical manifold, $S$, of the normal form \eqref{eq:normalform} in the $(\hat v_1,\hat v_2, u_2)$ projection. For each fixed $\hat v_1 < 0$, $S$ is cubic-shaped with a two-fold structure. These two folds coalesce at $\hat v_1 = 0$ at the cusp point (black marker), corresponding to the symmetric folded singularity. For each fixed $\hat v_1 > 0$, $S$ is no longer folded.}
  \label{fig:cusp}
\end{figure}
\end{remark}

By the regular fold condition (Assumption~2), the system \eqref{eq:desingnf} has no ordinary singularities $E$ in a neighbourhood of the origin. Violation of the regular fold condition will result in bifurcations of the symmetric folded singularities. 
We refer forward to Section~\ref{subsec:SFSNs} for examples of such bifurcations in a model system. See also \cite{Kristiansen2023} for a detailed study of this scenario (referred to there as the `cusped saddle-node singularity') in a pair of strongly linearly coupled FitzHugh-Nagumo oscillators, and \cite{Pedersen2026} for additional examples.

\subsection{Center Manifold Reduction} \label{subsec:cmreduction}
To simplify the analysis, we now switch to a coordinate system in which the symmetry and anti-symmetry axes are aligned with the coordinate axes. Let
\[ \begin{bmatrix} \hat u_1 \\ \hat u_2 \end{bmatrix} = \frac{1}{2} \begin{bmatrix} 1 & 1 \\ 1 & -1  \end{bmatrix} \begin{bmatrix} u_1 \\ u_2 \end{bmatrix} \quad {\rm and } \quad \begin{bmatrix} \hat v_1 \\ \hat v_2 \end{bmatrix} = \frac{1}{2} \begin{bmatrix} 1 & 1 \\ 1 & -1 \end{bmatrix} \begin{bmatrix} v_1 \\ v_2 \end{bmatrix}, \]
so that the $\left\{ \hat u_1 = 0, \hat v_1 = 0 \right\}$ subspace corresponds to the anti-symmetry axis $\mathcal L_{as}$ and the $\left\{ \hat u_2 = 0, \hat v_2 = 0 \right\}$ subspace corresponds to the symmetry axis $\mathcal L_s$.
Then the system \eqref{eq:normalform} transforms to 
\begin{equation} \label{eq:nfsymasym}
  \begin{split}
    \frac{d\hat u_1}{dt} &= -2\hat u_1 + \hat v_1 + \hat u_1^2+\hat u_2^2 - \tfrac{4a+2+\lambda_1}{2} \hat u_1 \hat v_1 + \tfrac{2+\lambda_1}{2}\hat u_1 \hat v_2+\tfrac{4a+2+\lambda_1}{2} \hat u_2 \hat v_2 + k \eps \hat u_1 + \hat P_1, \\
    \frac{d\hat u_2}{dt} &= \hat v_2+2 \hat u_1 \hat u_2-\tfrac{2+\lambda_1}{2} \hat u_2 \hat v_1+k \eps \hat u_2 + \hat P_2, \\
    \frac{d\hat v_1}{dt} &= \eps \left[ 1+\tfrac{\lambda_1+\lambda_2}{2} \hat u_1 + \tfrac{4b+a(\lambda_1+\lambda_2)}{4} \hat u_1^2 + \tfrac{4b-a(\lambda_1+\lambda_2)}{2} \hat u_1 \hat u_2 + \tfrac{4b+a(\lambda_1+\lambda_2)}{4} \hat u_2^2 + \hat Q_1 \right], \\
    \frac{d\hat v_2}{dt} &= \eps \left[ \tfrac{\lambda_1+\lambda_2}{2} \hat u_2 +\tfrac{4b-a(\lambda_1+\lambda_2)}{2} \hat u_1 \hat u_2 + \hat Q_2 \right], \\
    \frac{d\eps}{dt} &= 0,
  \end{split}
\end{equation}
where we have appended the trivial equation for $\eps$ to the system, and the higher order terms $\hat P_1, \hat P_2, \hat Q_1$, and $\hat Q_2$ are given by
\begin{equation*}
  \begin{split}
    \hat P_1 &= \tfrac{5a+2a^2+3p}{2} \hat u_1^3 - \tfrac{5a+2a^2+3p}{2} \hat u_1 \hat u_2^2 + \mathcal O \left( \hat u_1 \hat v_1^2, \eps \hat u_1 \hat v_1,\eps \hat u_2 \hat v_2 \right), \\ 
    \hat P_2 &= -\tfrac{5a+2a^2-p}{2} \hat u_1^2 \hat u_2+\tfrac{5a+2a^2-p}{2}\hat u_2^3 + \mathcal O \left( \hat u_2 \hat v_1^2,\hat u_1\hat v_2^2,\hat u_2 \hat v_2^2,\eps \hat u_2 \hat v_1,\eps \hat u_1 \hat v_2 \right),  \\
    \hat Q_1 &= \tfrac{4b+a(\lambda_1+\lambda_2)}{4} \hat u_2^2 + \mathcal O \left( \hat u_1 \hat v_1,\hat u_2 \hat v_1, \hat u_1 \hat v_2, \hat u_2 \hat v_2, \eps^2 \right),  \\ 
    \hat Q_2 &= \mathcal O \left( \hat u_2 \hat v_1,\hat u_1 \hat v_2 \right).
  \end{split}
\end{equation*}

\medskip

\begin{proposition}
For the system \eqref{eq:nfsymasym}, there exists an attracting 4D center manifold of the origin given by 
\[ W^c = \left\{ (u_1,v_1,u_2,v_2,\eps) : 
u_1 = h(u_2,v_1,v_2,\eps)
\right\}, \]
where the function $h$ is given by 
\[ h = \tfrac{1}{2}u_2^2 -\tfrac{4a+1+\lambda_1}{2}v_1^2 - \tfrac{4a+\lambda_1}{8} v_2^2+\tfrac{4a+\lambda_1}{4} u_2 v_2 + \tfrac{2+\lambda_1}{4} v_1 v_2 +\tfrac{4a+4k-\lambda_2}{4} \eps v_1 - \tfrac{2+\lambda_1}{8} \eps v_2 + \mathcal O \left( \eps^2,(u_2+v_1+v_2)^3 \right). \]
Moreover, the reduced dynamics on the center manifold are given by
\begin{equation} \label{eq:dynamicsonWc}
  \begin{split}
    \dot u_2 &= v_2 - \lambda_1 u_2 v_1 + \gamma_1 u_2 v_1^2 + \gamma_2 u_2^3 + \gamma_3 \eps z + \mathcal O \left( u_2^2 v_2, u_2 v_2^2  \right) + \eps \mathcal O \left( u_2 v_1, u_2 v_2, v_1 v_2 \right), \\
    \dot v_1 &= \eps \left[ 1 + \tfrac{\lambda_1+\lambda_2}{2} v_1 + \alpha_1 v_1^2 + \alpha_2 u_2 v_1 + \alpha_3 u_2^2 + \mathcal O \left( u_2 v_2, v_2^2, \eps^2 \right) \right], \\
    \dot v_2 &= \eps \left[ (\lambda_1+\lambda_2) u_2 + \beta u_2 v_1 + \mathcal O \left( v_1 v_2, \eps^2 \right) \right],
  \end{split}
\end{equation}
where the coefficients $(\alpha_1,\alpha_2,\alpha_3,\beta,\gamma_1,\gamma_2,\gamma_3)$ can be computed in terms of the coefficients of \eqref{eq:normalform}.
\end{proposition}

\begin{proof}
The fixed point at the origin of the extended system \eqref{eq:nfsymasym} has stable spectrum $\sigma_s = \{ -2 \}$ and center spectrum $\sigma_c = \{ 0,0,0,0 \}$, with corresponding stable and center subspaces, respectively, given by
\[ \mathbb E^s(0) = \operatorname{span} \begin{bmatrix} 1 \\ 0 \\ 0 \\ 0 \\ 0 \end{bmatrix} \quad {\rm and } \quad \mathbb E^c(0) = \operatorname{span}  \left\{ 
\begin{bmatrix} 1 \\ 0 \\ 2 \\ 0 \\ 0 \end{bmatrix} , 
\begin{bmatrix} 0 \\ 1 \\ 0 \\ 0 \\ 0 \end{bmatrix} , 
\begin{bmatrix} 0 \\ 0 \\ 0 \\ 1 \\ 0 \end{bmatrix} ,
\begin{bmatrix} 0 \\ 0 \\ \tfrac{1}{2} \\ 0 \\ 1 \end{bmatrix} 
\right\}. \]
We align the center directions with the coordinate axes via the change of coordinates
\[ \widetilde u_1 = \hat u_1-\tfrac{1}{2} \hat v_1 + \tfrac{1}{4} \eps, \quad \widetilde v_1 = \tfrac{1}{2}\hat v_1 - \tfrac{1}{4} \eps, \quad \widetilde u_2 = \hat u_2, \quad \widetilde v_2 = \hat v_2, \quad \text{ and } \quad \widetilde \eps = \tfrac{1}{2} \eps. \]
Transforming and dropping tildes, the linear part of the vector field has coefficient matrix
%
\[ Df = \begin{bmatrix} -2 & 0 & 0 & 0 & 0 \\ 0 & 0 & 0 & 0 & 1 \\ 0 & 0 & 0 & 1 & 0 \\ 0 & 0 & 0 & 0 & 0 \\ 0 & 0 & 0 & 0 & 0 \end{bmatrix}. \]
%
By the Center Manifold Theorem, there exists a 4D center manifold, $W^c$, of the origin given by 
\[ W^c = \left\{ (u_1,u_2,v_1,v_2,\eps) : u_1 = h(u_2,v_1,v_2,\eps) \right\} \]
where the function $h$ satisfies the invariance equation 
\begin{equation*} 
  \begin{split}
\tfrac{\partial h}{\partial u_2} \dot u_2 + \tfrac{\partial h}{\partial v_1} \dot v_1 + \tfrac{\partial h}{\partial v_2} \dot v_2 = -2h &+ h^2+u_2^2-(1+\lambda_1+4a)v_1^2-(4a+\lambda_1) h v_1  \\ 
& + \tfrac{2+\lambda_1}{2} v_2 \left( h+\tfrac{4a+2+\lambda_1}{2+\lambda_1}u_2 + v_1 \right) + \tfrac{4k-(2+\lambda_1+4a+\lambda_1+\lambda_2)}{2} \eps (h+v_1) + H,
  \end{split}
\end{equation*}
and $H$ denotes the higher-order terms.
By expanding $h$ as a power series of the form $\displaystyle h = \sum_k h_k$, where $h_k$ is a homogeneous polynomial of degree $k$ in the variables $(u_2,v_1,v_2,\eps)$, and equating coefficients of like terms, we obtain the function $h$ given in the statement of the proposition. 
%
%
The reduced dynamics on the center manifold are then obtained by substituting $u_1 = h$ into the ODEs and relabelling the coefficients.
\end{proof}

The reduced flow \eqref{eq:dynamicsonWc} on the center manifold is a slow/fast system in standard form. The critical manifold in this case is a hyperbolic paraboloid to leading order,
\[ S = \left\{ v_2 = \lambda_1 u_2 v_1 - \gamma_2 u_2^3 + \mathcal{O} \left( u_2 v_1^2 \right) \right\}. \]
The Jacobian of the layer problem of \eqref{eq:dynamicsonWc} is 
\[ \tfrac{\partial f}{\partial u_2} = -\lambda_1 v_1  + 3\gamma_2 u_2^2 + \mathcal O(v_1^2). \]
Hence, to leading order, the critical manifold has attracting and repelling sheets, $S_a$ and $S_r$, given by 
\begin{equation*}
  \begin{split}
    S_a &= \left\{ (u_2,v_1,v_2) \in S : v_1 < \tfrac{3\gamma_2}{\lambda_1} u_2^2 + \mathcal O(u_2^4) \right\}, \\ 
    S_r &= \left\{ (u_2,v_1,v_2) \in S : v_1 > \tfrac{3\gamma_2}{\lambda_1} u_2^2 + \mathcal O(u_2^4) \right\}.
  \end{split}
\end{equation*}
Moreover, the fold curve is given by 
\[ S_L = \left\{ (u_2,v_1,v_2) \in S : v_1 = \tfrac{3\gamma_2}{\lambda_1} u_2^2 + \mathcal O(u_2^3) \right\}. \]
Thus, the fold curve is tangent to the $u_2$-axis, which reflects the fact that the fold curve is tangent to the anti-symmetry axis $\mathcal L_{as}$ in the original coordinates \eqref{eq:normalform}.

The desingularized reduced system of \eqref{eq:dynamicsonWc} is 
\begin{equation}
  \begin{split}
    \dot u_2 &= \lambda_2 u_2 + \tfrac{2\beta+4\gamma_1 - \lambda_1(\lambda_1+\lambda_2)}{2} u_2 v_1 + \mathcal O\left( (u_2+v_1)^3 \right), \\
    \dot v_1 &= \lambda_1 v_1 -3\gamma_2 u_2^2 + \tfrac{\lambda_1(\lambda_1+\lambda_2)-2\gamma_1}{2} v_1^2 + \mathcal O\left( (u_2+v_1)^3 \right).
  \end{split}
\end{equation}
As expected, the origin is a SFN with eigenvalues $\lambda_1$ and $\lambda_2$. The eigenvector associated with the eigenvalue $\lambda_1$ is tangent to the $v_1$-axis and the eigenvector associated with the eigenvalue $\lambda_2$ is tangent to the $u_2$-axis. 
The critical manifold and slow flow on it are shown in Fig.~\ref{fig:hyperbolicparaboloid}.
%
%

\begin{figure}[h!]
  \centering
  \includegraphics[width=5in]{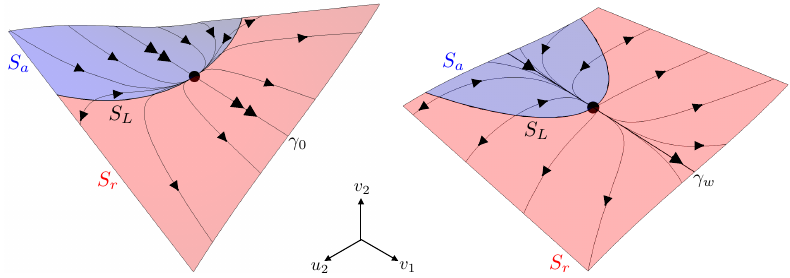}
  \caption{Critical manifold (blue and red surfaces) and reduced flow (black curves with arrows) for the system restricted to the center manifold \eqref{eq:dynamicsonWc} for the (a) strong SFN and (b) weak SFN. The only singular canard solution (the strong canard $\gamma_0$ in (a) and the weak canard $\gamma_w$ in (b)) is aligned with the $v_1$-axis, which corresponds to the axis of symmetry. The funnel of the SFN is (locally) the entire attracting manifold $S_a$.}
  \label{fig:hyperbolicparaboloid}
\end{figure}

From our analysis of the singular limit, we have established that the SFN has significant differences from the classical theory of folded nodes (such as the presence of only one singular canard solution). To determine how these differences in the singular limit translate to the full system dynamics (i.e., for $\eps$ sufficiently small and positive), we now turn our attention to the rigorous analysis of the SFN. See also \cite{Kristiansen2023}.

\section{Geometric Blow-Up}  \label{sec:blowup}
To analyse the dynamics around the symmetric folded node, we use the blow-up technique \cite{Dumortier1996,Krupa2001}, which inflates the SFN to a topological hemisphere. In so doing, the solution curves that approach the SFN asymptotically with differing orders of tangency are teased apart, and enough hyperbolicity is restored that the flow on the hemisphere can be analysed using standard dynamical systems techniques and invariant manifold theory. 

As is usual with the blow-up analysis of folded singularities \cite{Szmolyan2001}, we consider the extended system
\begin{equation} \label{eq:reducednormalform}
  \begin{split}
    \dot x &= \eps \left[ 1+\tfrac{\lambda_1+\lambda_2}{2} x + \alpha_1 x^2 + \alpha_2 x z + \alpha_3 z^2 + \mathcal O(y z, y^2,\eps^2)  \right] ,\\ 
    \dot y &= \eps \left[ (\lambda_1+\lambda_2) z + \beta x z + \mathcal O(xy,\eps^2) \right], \\ 
    \dot z &= y - \lambda_1 x z + \gamma_1 x^2 z + \gamma_2 z^3 + \gamma_3 \eps z + \mathcal O(yz^2,y^2z) + \eps \mathcal O \left( x z, yz, xy \right), \\
    \dot \eps &= 0,
  \end{split}
\end{equation}
in which $\eps$ is included as a dynamic variable, and the mapping between systems \eqref{eq:reducednormalform} and \eqref{eq:dynamicsonWc} is 
\[ (x,y,z) \mapsto (v_1,v_2,u_2). \]
The SFN at the origin of \eqref{eq:reducednormalform} is a nilpotent equilibrium with quadruple zero eigenvalue. To analyse the dynamics near this equilibrium, we introduce the blow-up transformation
\begin{equation} \label{eq:blowup}
  \begin{split}
    x = r \overline x, \quad y = r^2 \overline y, \quad z = r \overline z, \quad \text{ and } \quad \eps = r^2 \overline \eps,
  \end{split}
\end{equation}
where $(\overline x,\overline y,\overline z,\overline \eps) \in \mathbb S^3$, i.e., $\overline x^2+\overline y^2+\overline z^2+\overline \eps^2 = 1$, and $r \in [0,r_0]$, with $r_0$ sufficiently small, is a dynamic variable that measures the distance from the 3-sphere $\mathbb S^3$.

For the analysis of canard solutions, we focus on the following three overlapping coordinate charts:
\begin{itemize}
  \item The entry chart, $K_1 := \left\{ \overline x = -1 \right\}$, describes solutions entering the neighbourhood of the SFN.
  \item The central (or rescaling) chart, $K_2 := \left\{ \overline \eps = 1 \right\}$, tracks solutions over the surface of the hemisphere.  
  \item The exit chart, $K_3 := \left\{ \overline x = 1 \right\}$, describes solutions leaving the neighbourhood of the SFN.
\end{itemize}
Following convention, we use $\square_i$ to denote the blown-up variable $\overline \square$ in chart $K_i$. 
Since the analysis of canards of folded nodes is well-established, we only sketch an outline of the analysis here whilst highlighting the key differences between SFNs and classical folded nodes. 
The blown-up hemisphere, together with the leading order dynamics on it, is shown in Fig.~\ref{fig:blowupschematic}. 

\begin{figure}[h!]
  \centering
  \includegraphics[width=5in]{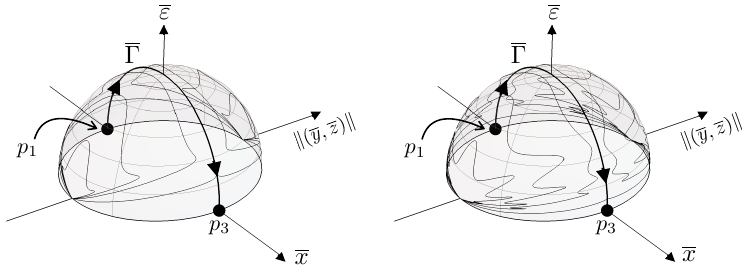}
  \put(-360,120){(a)}
  \put(-180,120){(b)}
  \caption{Dynamics of the symmetric folded node $\mathbb S^3$ in the blown-up coordinates $\left(\overline x, \pm \sqrt{\overline y^2 + \overline z^2}, \overline \eps \right)$ for the (a) strong case and (b) weak case. There is a special solution $\overline \Gamma$ (Section~\ref{subsec:centralchart}) corresponding to the $\overline x$-axis that connects an equilibrium point $p_1$ on the equator of the hemisphere in chart $K_1$ (Section~\ref{subsec:entrychart}) to another equilibrium point $p_3$ on the equator of the hemisphere in chart $K_3$ (Section~\ref{subsec:exitchart}). Solutions (thin black curves) exhibit oscillatory dynamics.}
  \label{fig:blowupschematic}
\end{figure}

\begin{remark}
In the case of FitzHugh-Nagumo oscillators with linear, repulsive coupling \cite{Kristiansen2023}, the weights of the blow-up are different from those used in \eqref{eq:blowup}. This is due to the fact that their system has strong linear coupling instead of strong nonlinear coupling.
\end{remark}

\subsection{Dynamics in the central chart} \label{subsec:centralchart}
We begin our analysis with the dynamics in the central chart, $K_2 = \{ \overline \eps =1\}$, which corresponds to the top of the hemisphere in Fig.~\ref{fig:blowupschematic}.
The blow-up transformation in chart $K_2$ is 
\[ x = r_2 x_2, \quad y = r_2^2 y_2, \quad z = r_2 z_2, \quad \text{ and } \quad \eps = r_2^2, \]
which is an $\eps$-dependent rescaling of the vector field. 
Transforming \eqref{eq:reducednormalform} and desingularizing by a factor of $r_2$ (i.e., setting $dt_2 = r_2 dt$) gives
\begin{equation} \label{eq:blownupK2}
  \begin{split}
    \dot x_2 &= 1+r_2 F_2 (x_2,y_2,z_2,r_2), \\ 
    \dot y_2 &= (\lambda_1+\lambda_2) z_2 + r_2 G_2(x_2,y_2,z_2,r_2), \\ 
    \dot z_2 &= y_2 - \lambda_1 x_2 z_2 + r_2 H_2(x_2,y_2,z_2,r_2),
  \end{split}
\end{equation}
where the overdot now denotes the derivative with respect to $t_2$, and the functions $F_2, G_2,$ and $H_2$ are
\begin{equation*}
  \begin{split}
    F_2(x_2,y_2,z_2,r_2) &= \tfrac{\lambda_1+\lambda_2}{2} x_2 + r_2 (\alpha_1 x_2^2 + \alpha_2 x_2 z_2 + \alpha_3 z_2^2) + r_2^2 \mathcal O(y_2 z_2, r_2), \\ 
    G_2(x_2,y_2,z_2,r_2) &= \beta x_2 z_2 + r_2 \mathcal O(x_2 y_2,r_2) ,\\ 
    H_2(x_2,y_2,z_2,r_2) &= \gamma_1 x_2 z_2 + \gamma_2 z_2^3 + \gamma_3 z_2 +r_2 \mathcal O(x_2 z_2, y_2 z_2^2, r_2x_2y_2,r_2 y_2 z_2).
  \end{split}
\end{equation*}
Thus, in the coordinate chart $K_2$, the blow-up transformation converts the singular perturbation problem \eqref{eq:reducednormalform} into a regular perturbation problem \eqref{eq:blownupK2}. 

The unperturbed problem, obtained by setting $r_2 = 0$ in \eqref{eq:blownupK2}, is given by 
\begin{equation} \label{eq:unperturbed}
  \begin{split}
    \dot x_2 &= 1, \\ 
    \dot y_2 &= (\lambda_1+\lambda_2) z_2, \\ 
    \dot z_2 &= y_2 - \lambda_1 x_2 z_2.
  \end{split}
\end{equation}
Dividing by the $x_2$-equation, we obtain a linear non-autonomous system, which can be solved explicitly. 
%
%
(The general solution is a linear combination of Hermite functions and generalized hypergeometric functions.)
From this, we determine that the system \eqref{eq:unperturbed} possesses exactly one algebraic solution 
\[ \Gamma_2 = \left\{ (x_2,y_2,z_2,r_2) = \left( t_2, 0, 0, 0 \right)  \right\}, \]
which reflects the fact that there is only ever one singular canard. 
If the SFN is strong, i.e., $\lambda_1<\lambda_2<0$, then $\Gamma_2$ corresponds to the strong canard. 
If the SFN is weak, i.e., $\lambda_2<\lambda_1<0$, then $\Gamma_2$ corresponds to the weak canard. 

\medskip

\begin{remark}
In the classical theory of folded nodes \cite{Szmolyan2001,Wex2005}, there are two explicit algebraic solutions for the unperturbed problem in the central chart of the blow-up. These algebraic solutions correspond to the strong and weak canards of the folded node. 
\end{remark}

\medskip

\begin{proposition} \label{prop:rotations}
For the normal form \eqref{eq:normalform}, if the symmetric folded node is
\begin{enumerate}[label=(\roman*)]
  \item Strong ($\lambda_1<\lambda_2<0$): there exists a maximal canard solution, $\gamma_0^{\eps}$, that lies along the axis of symmetry and corresponds to the singular strong canard. Moreover, the slow manifolds $S_{a,\eps}$ and $S_{s,\eps}$ of \eqref{eq:normalform} can exhibit up to two twists about $\gamma_0^\varepsilon$ in a neighbourhood of the fold curve. 
  \item Weak ($\lambda_2<\lambda_1<0$): there exists a maximal canard solution, $\gamma_w^{\eps}$, that lies along the axis of symmetry and corresponds to the singular weak canard. Moreover, if $n-1 < \tfrac{\lambda_2}{\lambda_1} < n$ for some integer $n\geq 2$, then the slow manifolds $S_{a,\eps}$ and $S_{s,\eps}$ of \eqref{eq:normalform} twist up to $n+1$ times about $\gamma_w^{\eps}$ in a neighbourhood of the fold curve.
\end{enumerate}
\end{proposition}

\begin{proof}
The details of this proof are similar to Lemma 4.4 of \cite{Szmolyan2001}, Section~2.1 of \cite{Mitry2017}, and Section~3 of \cite{Kristiansen2023}. As such, we only sketch an outline of the proof here to highlight the key differences between the SFN and classical folded nodes.

Using $x_2$ as the time variable, the blown-up vector field \eqref{eq:blownupK2} is equivalent to the non-autonomous system
\begin{equation*}
  \begin{split}
    \begin{bmatrix} y_2^\prime \\ z_2^\prime \end{bmatrix} = \begin{bmatrix} 0 & \lambda_1+\lambda_2 \\ 1 & -\lambda_1 x_2 \end{bmatrix} \begin{bmatrix} y_2 \\ z_2 \end{bmatrix} + r_2 \begin{bmatrix} G_2 - (\lambda_1+\lambda_2) z_2 F_2 \\ H_2 - (y_2 - \lambda_1 x_2 z_2) F_2 \end{bmatrix} + \mathcal O(r_2^2),
  \end{split}
\end{equation*}
where the prime denotes the derivative with respect to $x_2$.
The variational equation along $\Gamma_2$ is 
\begin{equation} \label{eq:variationalproof}
  \begin{split}
    \begin{bmatrix} y_2^\prime \\ z_2^\prime \end{bmatrix} = \begin{bmatrix} 0 & \lambda_1+\lambda_2 \\ 1 & -\lambda_1 x_2 \end{bmatrix} \begin{bmatrix} y_2 \\ z_2 \end{bmatrix},
  \end{split}
\end{equation}
which is equivalent to the unperturbed problem \eqref{eq:unperturbed}. 

To measure rotation in the $(y_2,z_2)$-space, we introduce the affine coordinate $w_2 = \frac{z_2}{y_2}$. The variational equation then transforms to 
\[ w_2^\prime = 1-\lambda_1 x_2 w_2 - (\lambda_1+\lambda_2) w_2^2. \]
This Riccati equation can be transformed into a second-order linear Weber equation by first making the coordinate transformation
\[ w_2(x_2) = \frac{1}{\lambda_1+\lambda_2} \frac{\psi_2^{\prime}(x_2)}{\psi_2(x_2)}, \]
followed by the rescaling $x_2 = \left( -\tfrac{1}{2} \lambda_1 \right)^{-1/2} \tau_2$.
The resulting Weber equation is 
\begin{equation}  \label{eq:weber}
  \begin{split}
	\psi_2^{\prime \prime} - 2 \tau_2 \psi_2^\prime + 2 \left( 1+\frac{\lambda_2}{\lambda_1} \right) \psi_2 = 0,
  \end{split}
\end{equation}
where the primes now denote derivatives with respect to $\tau_2$. For this Weber equation, there exists a pair of linearly independent solutions 
\begin{equation}  \label{eq:lisolutions}
  \begin{split}
    \psi_{2,1} &= {}_1 F_1\left( -\tfrac{1}{2}-\tfrac{1}{2} \tfrac{\lambda_2}{\lambda_1};  \tfrac{1}{2}; \tau_2^2 \right) \quad \text{ and } \quad
    \psi_{2,2} = \tau_2 \cdot  {}_1 F_1\left( -\tfrac{1}{2} \tfrac{\lambda_2}{\lambda_1}; \tfrac{3}{2};  \tau_2^2 \right),
  \end{split}
\end{equation}
where $_1 F_1(a;b;\xi)$ denotes the confluent hypergeometric function of the first kind with parameters $a$ and $b$. 
The zeros of $\psi_2$ correspond to poles of $\tfrac{z_2}{y_2}$, which is the tangent of the angle that the solution makes with the $x_2$-axis. Hence, by counting the number of zeros of $\psi_2$, we may infer the number of twists that solutions of \eqref{eq:variationalproof} exhibit. 

For the strong SFN, i.e., for $\lambda_1<\lambda_2<0$, the function $\psi_{2,1}$ has the asymptotic behaviour
\[ \psi_{2,1}(\tau_2) \sim \exp (\tau_2^2) \tau_2^{-(2+\lambda_2/\lambda_1)} \quad {\rm as } \quad \tau_2 \to \pm \infty.  \]
Moreover, the even function $\psi_{2,1}$ possesses two zeros for all $\tfrac{\lambda_2}{\lambda_1} \in (0,1)$.
Similarly, the function $\psi_{2,2}$ has the asymptotic behaviour
\[ \psi_{2,2}(\tau_2) \sim \exp(\tau_2^2) \tau_2^{-(2+\lambda_2/\lambda_1)} \quad {\rm as } \quad \tau_2 \to \pm \infty.  \]
Moreover, the odd function $\psi_{2,2}$ possesses three zeros.
Together, these results imply that for $\lambda_1<\lambda_2<0$, non-trivial solutions of \eqref{eq:weber} possess up to three zeros (Fig.~\ref{fig:weber}(a)). 
Since two consecutive zeros correspond to a twist (i.e., a 180$^{\circ}$-rotation) in the $(y_2,z_2)$ plane, it follows that non-trivial solutions of \eqref{eq:weber} twist up to two times about $\Gamma_2$. 
This completes the proof of part (i). 

\begin{figure}[h!]
   \centering
   \includegraphics[width=5in]{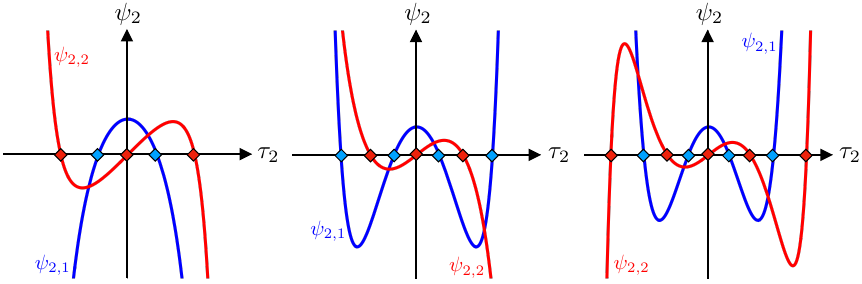}
   \put(-360,110){(a)}
   \put(-240,110){(b)}
   \put(-120,110){(c)}
   \caption{Linearly independent solutions $\psi_{2,1}$ (blue) and $\psi_{2,2}$ (red) for (a) $\mu \in (0,1)$, (b) $\mu \in (1,2)$, and (c) $\mu \in (2,3)$.}
   \label{fig:weber}
\end{figure}

For the weak SFN, i.e., for $\lambda_2 < \lambda_1 < 0$, if $\tfrac{\lambda_2}{\lambda_1}$ is an integer, then one of the solutions in \eqref{eq:lisolutions} simplifies to a Hermite polynomial. More precisely, for $\tfrac{\lambda_2}{\lambda_1} = 2k-1$ with $k \in \mathbb N \setminus \{ 1 \}$, 
\[ \psi_{2,1} = {}_1 F_1(-k,\tfrac{1}{2},\tau_2^2) = (-1)^k \frac{k!}{(2k)!} H_{2k}(\tau_2), \]
where $H_n(\tau_2)$ is a Hermite polynomial of degree $n$.
Similarly, for $\tfrac{\lambda_2}{\lambda_1} = 2k$ with $k \in \mathbb N \setminus \{ 0 \}$, 
\[ \psi_{2,2} = \tau_2 \cdot {}_1 F_1(-k,\tfrac{3}{2},\tau_2^2) = (-1)^k \frac{k!}{2 (2k+1)!} H_{2k+1}(\tau_2). \]
Since a Hermite polynomial of degree $k$ has exactly $k$ zeros, it follows that for $n-1 < \frac{\lambda_2}{\lambda_1} < n$ where $n \in \mathbb N \setminus \{0,1\}$, solutions of the Weber equation possess up to $n+2$ zeros (Fig.~\ref{fig:weber}(b) and (c)). Hence, solutions near $\Gamma_2$ possess at most $n+1$ twists. This completes the proof of part (ii).
\end{proof}

\begin{remark}
The rotational behaviour of solutions and of the slow manifolds is encoded in the Weber equation \eqref{eq:weber}. For the SFN, equation \eqref{eq:weber} has an important difference from the Weber equations in the central charts of classical folded nodes and classical folded saddles. Namely, the Weber equations of classical folded nodes \cite{Szmolyan2001,Wex2005} and folded saddles \cite{Mitry2017} are given by
\[ \psi_2^{\prime \prime} - 2 \tau_2 \psi_2^\prime + 2 \frac{\lambda_2}{\lambda_1}  \psi_2 = 0. \]
Thus, solutions near a SFN can exhibit a greater number of twists than those near a classical folded node. 
\end{remark}

\subsection{Dynamics in the entry chart $\boldsymbol{K_1}$}   \label{subsec:entrychart}
To study the unbounded branch of the special solution $\Gamma_2$ in backward time, we use the directional blow-up defined by $\overline x=-1$. The blow-up transformation is
\[ x = -r_1, \quad y = r_1^2 y_1, \quad z = r_1 z_1, \quad \text{ and } \quad \eps = r_1^2 \eps_1. \]
Transforming \eqref{eq:reducednormalform} and desingularizing by a factor of $r_1$ (i.e., setting $dt_1 = r_1 dt$) gives
\begin{equation}  \label{eq:blownupK1}
  \begin{split}
      \dot r_1 &= -r_1 \eps_1 \left( 1+ r_1 F_1 \right), \\
      \dot y_1 &= \eps_1 \left( 2y_1 + (\lambda_1+\lambda_2)z_1 + r_1  \left( G_1+2 y_1 F_1\right) \right), \\ 
      \dot z_1 &= y_1 + \lambda_1 z_1 + \eps_1 z_1 + r_1 \left(H_1+\eps_1 z_1 F_1\right), \\ 
      \dot \eps_1 &= 2 \eps_1^2 \left( 1 + r_1 F_1 \right).
  \end{split}
\end{equation}
Here, the overdot denotes the derivative with respect to $t_1$, and the functions $F_1, G_1$, and $H_1$ are given by
\begin{equation*}
  \begin{split}
      F_1(r_1,y_1,z_1,\eps_1) &= -\tfrac{\lambda_1+\lambda_2}{2} + r_1 (\alpha_1 - \alpha_2 z_1 + \alpha_3 z_1^2) + r_1^2 \mathcal O \left( y_1z_1,r_1 y_1^2, r_1 \eps_1^2 \right), \\ 
      G_1(r_1,y_1,z_1,\eps_1) &= -\beta z_1 + r_1 \mathcal O(y_1, r_1^2 \eps_1^2), \\ 
      H_1(r_1,y_1,z_1,\eps_1) &= \gamma_1 z_1 + \gamma_2 z_1^3 + \gamma_3 \eps_1 z_1 + r_1 \mathcal O \left( y_1 z_1,  \eps_1 z_1 \right).
  \end{split}    
\end{equation*}

\begin{proposition}  \label{prop:entrychart}
The 4D system \eqref{eq:blownupK1} possesses a surface $S_1 = \left\{ r_1 \geq 0, y_1 \in \mathbb R, z_1 = -\tfrac{y_1}{\lambda_1}+\mathcal O(r_1), \eps_1 = 0 \right\}$ of equilibria with stable eigenvalue given by $\lambda_1 + \mathcal O(r_1)$ and triple zero eigenvalue. Moreover, the following statements hold.
\begin{enumerate}[label=(\roman*)]
  \item There exists an attracting 3D center manifold, $W^c(S_1)$, of $S_1$. The center manifold can be expressed as a graph $z_1 = h_a(y_1,r_1,\eps_1)$ for a sufficiently small neighbourhood of $r_1=0,\eps_1=0$. The branch of $W^c(S_1)$ in $\{ r_1=0, \eps_1>0 \}$ is unique.
  \item There exists a stable invariant foliation, $\mathcal F^s$, with base $W^c(S_1)$ and 1D fibers. For any $k>\lambda_1 + \mathcal O(r_1)$, the contraction along $F^s$ during a time interval $[0,T]$ is stronger than $e^{kT}$.
\end{enumerate} 
\end{proposition}

We give details of the proof in Appendix~\ref{appendix:entrychart}.

\subsection{Dynamics in the exit chart $\boldsymbol{K_3}$}  \label{subsec:exitchart}
To study the unbounded branch of the special solution $\Gamma_2$ in forward time, we use the directional blow-up defined by $\overline x = 1$. The blow-up transformation in this case is 
\[ x = r_3, \quad y = r_3^2 y_3, \quad z = r_3 z_3, \quad \text{ and } \quad \eps = r_3^2 \eps_3. \]
Transforming \eqref{eq:reducednormalform} and desingularizing by a factor of $r_3$, i.e., setting $dt_3 = r_3 dt$, we obtain 
\begin{equation} \label{eq:blownupK3}
  \begin{split}
    \dot r_3 &= r_3 \eps_3 \left( 1+r_3 F_3 \right), \\ 
    \dot y_3 &= \eps_3 \left( - 2y_3 + \left( \lambda_1+\lambda_2 \right) z_3 + r_3 (G_3 - 2y_3 F_3) \right), \\ 
    \dot z_3 &= y_3 - \lambda_1 z_3 - \eps_3 z_3 + r_3 H_3, \\ 
    \dot \eps_3 &= -2\eps_3^2 \left( 1 + r_3 F_3 \right).
  \end{split}
\end{equation}
Here, the overdot denotes the derivative with respect to $t_3$, and the functions $F_3, G_3$, and $H_3$ are given by
\begin{equation*}
  \begin{split}
    F_3 (r_3,y_3,z_3,\eps_3) &= \tfrac{\lambda_1+\lambda_2}{2} + r_3 (\alpha_1+\alpha_2 z_3 + \alpha_3 z_3^2) + r_3^2 \mathcal O(y_3z_3,r_3y_3^2,r_3 \eps_3^2), \\ 
    G_3 (r_3,y_3,z_3,\eps_3) &= \beta z_3 + r_3 \mathcal O(y_3,r_3^2 \eps_3^2), \\ 
    H_3 (r_3,y_3,z_3,\eps_3) &= \gamma_1 z_3 + \gamma_2 z_3^2 + \gamma_3 \eps_3 z_3 + r_3 \mathcal O (\eps_3z_3,y_3z_3^2,r_3y_3z_3).
  \end{split}
\end{equation*}

\begin{proposition}  \label{prop:exitchart}
The 4D system \eqref{eq:blownupK3} possesses a surface, $S_3 = \left\{ r_3 \geq 0, y_3 \in \mathbb R, z_1 = \tfrac{y_3}{\lambda_1}+\mathcal O(r_3), \eps_3 = 0 \right\}$, of equilibria with unstable eigenvalue given by $-\lambda_1 + \mathcal O(r_3)$ and triple zero eigenvalue. Moreover, the following statements hold. 
\begin{enumerate}[label=(\roman*)]
  \item There exists a repelling 3D center manifold, $W^c(S_3)$, of $S_3$. The center manifold can be expressed as a graph $z_3 = h_3(y_3,r_3,\eps_3)$ for a sufficiently small neighbourhood of $r_3=0,\eps_3=0$. The branch of $W^c(S_3)$ in $\left\{ r_3 =0, \eps_3>0 \right\}$ is unique. 
  \item There exists an unstable invariant foliation, $\mathcal F^u$, with base $W^c(S_3)$ and 1D fibers. For any $k < -\lambda_1 + \mathcal O(r_3)$, the expansion along $\mathcal F^u$ during a time interval $[0,T]$ is stronger than $e^{kT}$.
\end{enumerate}
\end{proposition}

We give details of the proof in Appendix~\ref{appendix:exitchart}.

\subsection{Intersection of the center-stable and center-unstable manifolds}   \label{subsec:transitionmaps}
We now track solutions as they move over the topological hemisphere (chart $K_2$) and approach the equator of the hemisphere (charts $K_1$ and $K_3$). 

\begin{lemma}  \label{lemma:transitionmaps}
The transition map, $\kappa_{21}$, from $K_2$ to $K_1$ is given by 
\[ \kappa_{21}\left( x_2,y_2,z_2,r_2 \right) = (r_1,y_1,z_1,\eps_1) = \left( -r_2 x_2, x_2^{-2} y_2, -x_2^{-1}z_2, x_2^{-2} \right), \]
where $x_2<0$.  The transition map, $\kappa_{23}$, from $K_2$ to $K_3$ is given by
\[ \kappa_{23}\left( x_2,y_2,z_2,r_2 \right) = (r_3,y_3,z_3,\eps_3) = \left( r_2 x_2, x_2^{-2}y_2, x_2^{-1}z_2, x_2^{-2} \right),   \]
where $x_2 > 0$.
\end{lemma}

By Lemma~\ref{lemma:transitionmaps}, the special solution $\Gamma_2$ in the entry chart $K_1$ for $t_2 \in (-\infty,-T)$, with $T$ sufficiently large and positive, is $\kappa_{21} \left( \Gamma_2 \right) = \left( 0,0,0,t_2^{-2} \right)$. It follows then that
\[ \lim_{t_2 \to -\infty} \kappa_{21} \left( \Gamma_2 \right) = \left( 0,0,0,0 \right) =: p_1.  \]
Thus, the special solution $\Gamma_2$ limits on the equilibrium point $p_1 \in S_1$ in chart $K_1$ (see Fig.~\ref{fig:blowupschematic}). 
Moreover, the tangent vector of $\Gamma_2$ in chart $K_1$ at the equilibrium point $p_1$ is given by
\[ \lim_{t_2 \to -\infty} \frac{\tfrac{d}{dt_2} \kappa_{21} \left( \Gamma_2 \right)}{\lVert \tfrac{d}{dt_2} \kappa_{21} \left( \Gamma_2 \right) \rVert} = \lim_{t_2 \to -\infty} \frac{(0,0,0,-2t_2^{-3})}{\lVert (0,0,0,-2t_2^{-3}) \rVert} = (0,0,0,1). \]
Hence, the solution $\Gamma_2$ approaches the point $p_1$ tangent $W^c(S_1)$ in the entry chart $K_1$ (see Appendix~\ref{appendix:entrychart} for the generalized eigenspaces of $S_1$). That is, the part of $\Gamma_2(t_2)$ corresponding to $t_2 \in (-\infty,-T)$ is part of the unique branch of the center manifold $W^c(S_1)$ in the subspace $\{ r_1 = 0, \eps_1 >0 \}$ in chart $K_1$.

Similarly, using Lemma~\ref{lemma:transitionmaps}, we find that the special solution $\Gamma_2$ in the exit chart $K_3$ for $t_2 \in (T,\infty)$, with $T$ sufficiently large and positive, is $\kappa_{23} \left( \Gamma \right) = \left( 0,0,0,t_2^{-2} \right)$. It follows then that
\[ \lim_{t_2 \to \infty} \kappa_{23} \left( \Gamma \right) = \left( 0,0,0,0 \right) =: p_3.  \]
Thus, the special solution $\Gamma_2$ limits on the equilibrium point $p_3 \in S_3$ in chart $K_3$ (see Fig.~\ref{fig:blowupschematic}).
Moreover, the tangent vector of $\Gamma_2$ in chart $K_3$ at the equilibrium point $p_3\in S_3$ is given by
\[ \lim_{t_2 \to \infty} \frac{\tfrac{d}{dt_2} \kappa_{23} \left( \Gamma_2 \right)}{\lVert \tfrac{d}{dt_3} \kappa_{23} \left( \Gamma_2 \right) \rVert} = \lim_{t_2 \to \infty} \frac{(0,0,0,-2t_2^{-3})}{\lVert (0,0,0,-2t_2^{-3}) \rVert} = (0,0,0,1). \]
Hence, the solution $\Gamma_2$ approaches the point $p_3$ tangent to $W^c(S_3)$ in the exit chart $K_3$ (see Appendix~\ref{appendix:exitchart} for the generalized eigenspaces of $S_3$). That is, the part of $\Gamma_2(t_2)$ corresponding to $t_2 \in (T,\infty)$ is part of the unique branch of the center manifold $W^c(S_3)$ in the subspace $\{ r_3 = 0, \eps_3 >0 \}$ in chart $K_3$. 

By the above, the center manifolds $W^c(S_1)$ and $W^c(S_3)$ intersect along the special solution $\Gamma_2$ on the sphere $\mathbb S^3$. 
Hence, as in the classical folded node, we have established the persistence of the maximal canard solution $\gamma_0$ (for the strong SFN) or $\gamma_w$ (for the weak SFN). 
We refer the reader to \cite{Mitry2017,Szmolyan2001,Wex2005} for further details.

\medskip
\begin{remark}
The intersection between $W^c(S_1)$ and $W^c(S_3)$ can be shown to be transverse provided the eigenvalue ratio $\tfrac{\lambda_2}{\lambda_1}$ is not an integer. Conversely, when $\tfrac{\lambda_2}{\lambda_1}$ is an integer, the tangent spaces of $W^c(S_1)$ and $W^c(S_3)$ coincide and these manifolds form a single twisted band. This tangency of $W^c(S_1)$ and $W^c(S_3)$ at integral values may then break under perturbations in $\tfrac{\lambda_2}{\lambda_1}$, which can lead to the creation of additional intersections, i.e., secondary canards.
\end{remark}

\section{Twisted slow manifolds, ribbons, \& symmetry-breaking} \label{sec:ribbons}
Having established the properties of the strong and weak SFNs, we now examine their effect on symmetry-breaking in the 4D symmetrically coupled oscillator system. Recall by Fenichel theory that the normally hyperbolic segments, $S_a$ and $S_s$, of the critical manifold persist for sufficiently small $\eps$ as invariant slow manifolds, $S_{a,\eps}$ and $S_{s,\eps}$, respectively. The invariant slow manifolds near fold and cusp singularities (where normal hyperbolicity breaks down) can become complicated \cite{Broer2013,Kristiansen2023,Szmolyan2001}, often exhibiting localized twisting behaviour. Our analysis shows that these slow manifolds are guaranteed to twist in the neighbourhood of an SFN around the persistent primary maximal canard (which is aligned with the axis of symmetry). 

\subsection{The strong symmetric folded node has two types of rotation}
We now compute the attracting slow manifold, $S_{a,\eps}$, of the strong SFN in the 4D coupled oscillator normal form \eqref{eq:nfsymasym}. 
To do this, we implement the two-point boundary value problem setup developed in \cite{Desroches2008,Hasan2017} and outlined in Appendix~\ref{appendix:bvp}. A representative example is shown in Fig.~\ref{fig:ribbonsstrong}. The main differences here from classical folded node theory are (i) there is no weak canard, (ii) the strong canard plays the role of the axis of rotation, and (iii) the funnel region is (locally) the entire attracting manifold $S_{a,\eps}$, so that every solution on $S_{a,\eps}$ exhibits at least one twist. 

\begin{figure}[h!]
  \centering
  \includegraphics[width=5in]{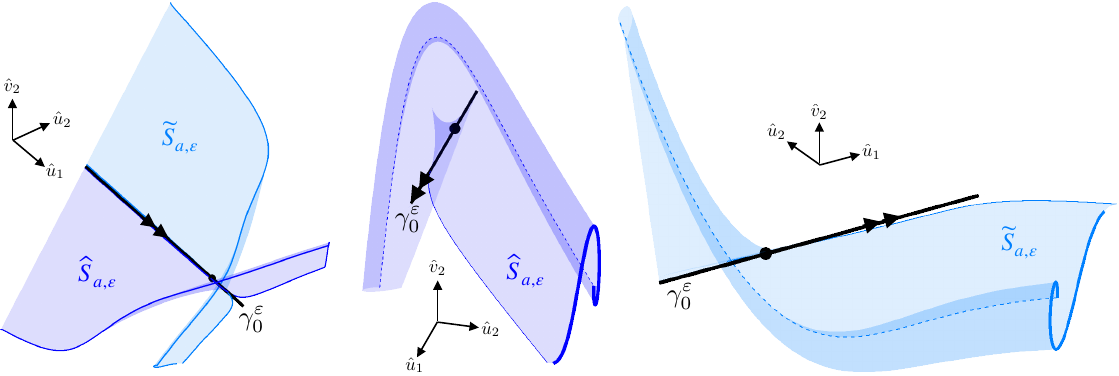}
  \put(-360,112){(a)}
  \put(-258,112){(b)}
  \put(-174,112){(c)}
  \caption{Attracting slow manifold near a strong SFN (black marker) in the 4D coupled oscillator normal form \eqref{eq:nfsymasym} for $\lambda_1 = -2, \lambda_2 = -0.5, a=0.1, b=-0.2$, and $\eps = 0.025$ with all other parameters set to zero. 
  (a) The transition past the strong SFN causes the two halves of the slow manifold, $\widehat{S}_{a,\eps}$ and $\widetilde{S}_{a,\eps}$, to reverse their `polarities' relative to the strong canard $\gamma_0^{\eps}$. 
  (b) Zoom of $\widehat{S}_{a,\eps}$ near the SFN. 
  (c) Zoom of $\widetilde{S}_{a,\eps}$ near the SFN. 
  After the primary rotation about $\gamma_0^{\eps}$, there is additional twisting of the slow manifolds about a secondary axis of rotation, approximated here by the dashed curves in (b) and (c).}
  \label{fig:ribbonsstrong}
\end{figure}

Now, since there is a reflection symmetry $\mathcal R$ in the system, we distinguish the two halves of the slow manifold $S_{a,\eps}$ as follows. Let $\widehat{S}_{a,\eps}$ be the subset of $S_{a,\eps}$ for which $\hat u_2 <0$ (corresponding to $u_2 < u_1$ in the original coordinates of \eqref{eq:normalform}). These are the points on the slow manifold that, in the absence of coupling, encounter the fold curve of oscillator 1 before that of oscillator 2. Hence, the points on $\widehat{S}_{a,\eps}$ are those that favour a fast jump in oscillator 1. By the reflection symmetry, the manifold $\widetilde{S}_{a,\eps} := \mathcal R \widehat{S}_{a,\eps}$ is the subset of $S_{a,\eps}$ for which $\hat u_2 >0$ (corresponding to $u_2>u_1$ in the original $(u_1,v_1,u_2,v_2)$ coordinates). These are the points on the slow manifold that, in the absence of coupling, encounter the fold curve of oscillator 2 before that of oscillator 1 and hence favour fast jumps in oscillator 2.

The overall structure of $S_{a,\eps} = \widehat{S}_{a,\eps} \cup \widetilde{S}_{a,\eps}$ is shown in Fig.~\ref{fig:ribbonsstrong}(a). Since the strong canard $\gamma_0^{\eps}$ lies along the axis of symmetry, the solutions on $\widehat{S}_{a,\eps}$ (which start with $\hat u_2<0$) twist around $\gamma_0^{\eps}$ and end in a region of phase space where $\hat u_2>0$. Similarly, the solutions on $\widetilde{S}_{a,\eps}$ (which start with $\hat u_2>0$) twist around $\gamma_0^{\eps}$ and end in a region of phase space where $\hat u_2<0$. That is, the local passage past the strong SFN causes solutions that initially favour a fast jump in oscillator $i$ for $i=1,2$ to switch to a region that favours a fast jump in the other oscillator.

The two halves of $S_{a,\eps}$ are shown in greater detail in Figs.~\ref{fig:ribbonsstrong}(b) and (c). In panel (b) we see that, in addition to the twist (i.e., half-rotation) about $\gamma_0^{\eps}$, the lower half of the attracting slow manifold, $\widehat{S}_{a,\eps}$, winds around a secondary axis of rotation (dashed blue curve). Similarly, $\widetilde{S}_{a,\eps}$ winds around a secondary axis of rotation (Fig.~\ref{fig:ribbonsstrong}(c); dashed light blue curve). We conjecture that this secondary axis of rotation corresponds to the solution $\gamma_w$ which is tangent in the singular limit to the weak eigendirection of the strong SFN. The rigorous analysis of this secondary rotation is technical and challenging, and we leave it to future work. 

\medskip
\begin{remark}
Here, we sketch an outline of the approach needed to rigorously determine the properties of the secondary rotations, i.e., of solutions that escape the neighbourhood of the SFN along its weak eigendirection rather than its strong eigendirection. 
In the first step, we require information from the other coordinate charts in the blow-up analysis. 
In the coordinate chart defined by $K_4 := \left\{ \overline{z} = 1 \right\}$ for instance, the vector field, after transformation and desingularization by a factor of $r_4$, is given by
\begin{equation} \label{eq:chart4}
  \begin{split}
    \dot x_4 &= \eps_4 - x_4 \left( y_4 - \lambda_1 x_4 + r_4 R_4 \right) + r_4 \eps_4 X_4, \\ 
    \dot y_4 &= (\lambda_1+\lambda_2) \eps_4 - 2y_4 \left( y_4 - \lambda_1 x_4 + r_4 R_4 \right) + r_4 \eps_4 Y_4, \\ 
    \dot r_4 &= r_4 \left( y_4 - \lambda_1 x_4 + r_4 R_4 \right), \\ 
    \dot \eps_4 &= -2\eps_4 \left( y_4 - \lambda_1 x_4 + r_4 R_4 \right),
  \end{split}
\end{equation}
where the overdot denotes derivatives with respect to $t_4$, and the functions $R_4, X_4,$ and $Y_4$ are 
\begin{equation*}
  \begin{split}
    R_4 &= \gamma_1 x_4^2 + \gamma_2 + \gamma_3 \eps_4 + \phi r_4 y_4 + \mathcal O \left( r_4 \eps_4 x_4 \right), \\ 
    X_4 &= \tfrac{\lambda_1+\lambda_2}{2} x_4 + r_4 \left( \alpha_1 x_4^2 + \alpha_2 x_4 + \alpha_3 \right) + \mathcal O \left( r_4^2 y_4 \right), \\
    Y_4 &= \beta x_4 + r_4 \mathcal O \left( x_4 y_4, r_4 \eps_4^2 \right).
  \end{split}
\end{equation*}
The system \eqref{eq:chart4} possesses a 2D manifold of fixed points,
\[ S_4(x_4,r_4) = \left\{ x_4 \in \mathbb R, y_4 - \lambda_1 x_4 + r_4 R_4 =0, r_4 \in \mathbb R, \eps_4 =0 \right\}, \]
which corresponds to the critical manifold. For most points in $S_4$, the spectrum consists of a single hyperbolic eigenvalue (which can take either sign depending on $x_4$ and $r_4$) and a triple zero eigenvalue. Moreover, there is a 1D curve, $x_4 = \hat x_4(r_4)$, along which the spectrum of $\mathcal C_4(r_4) := S_4(\hat x_4(r_4), r_4)$ consists of a quadruple zero eigenvalue. To study the dynamics around this 1D nilpotent curve, we would need to perform a cylindrical blow-up in which $\mathcal C_4$ is transformed to a topological cylinder. A suitable atlas of overlapping coordinate charts would then be needed to analyse the dynamics. Since, as best as we can determine, there is no special algebraic solution to follow, the analysis of the behaviour around $\mathcal C_4$ is substantially different from the analysis of the behaviour around the special solution $\Gamma_2$ in chart $K_2$, and is beyond the scope of the current work. 
\end{remark}

\subsection{Ribbons of the weak symmetric folded node and symmetry-breaking}
As in the previous subsection, we compute the attracting slow manifold in the case of the weak SFN. A representative example is shown in Fig.~\ref{fig:ribbons}. Since the structure of $S_{a,\eps}$ is much more intricate here compared to the strong SFN, we only show $\widehat{S}_{a,\eps}$ and recall that the other half $\widetilde{S}_{a,\eps}$ can be obtained by applying the reflection symmetry $\mathcal R$. 

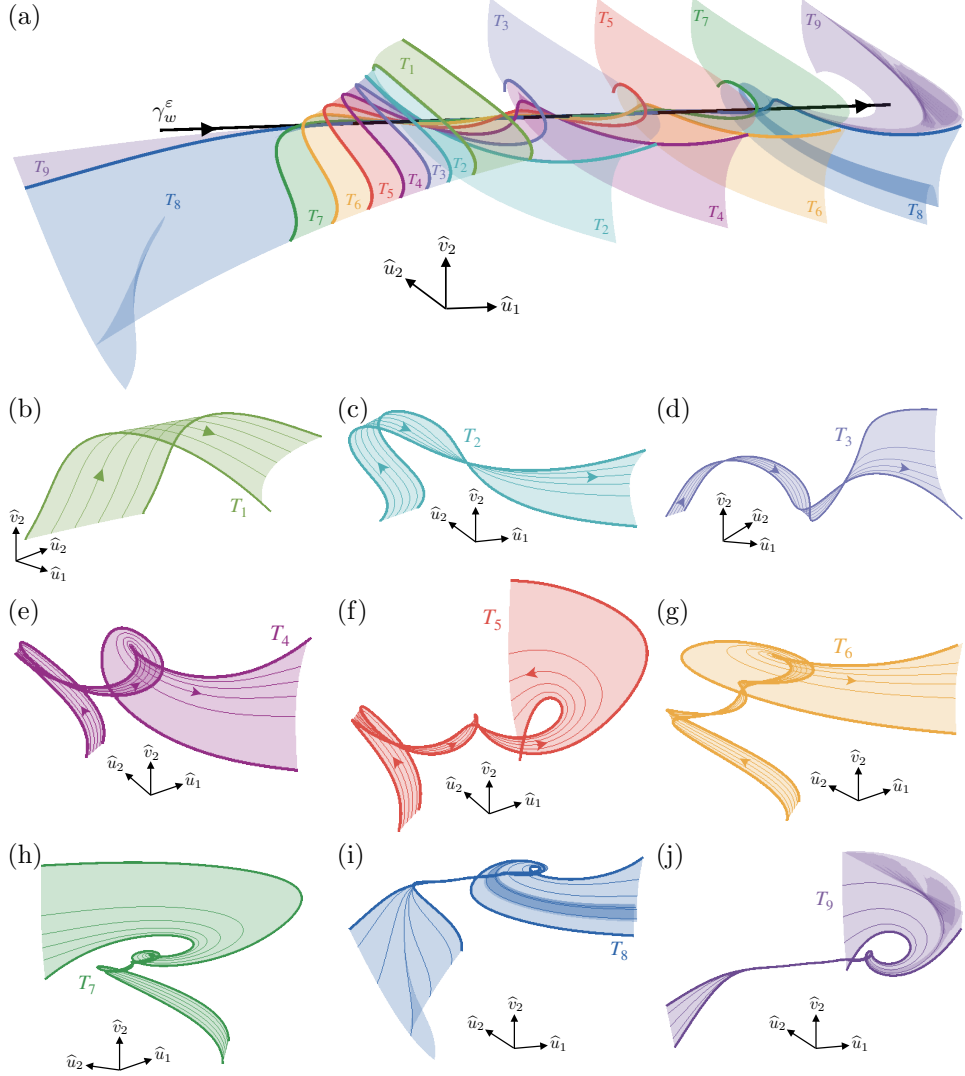
\begin{figure}[h!]
  \centering
  \includegraphics[width=5in]{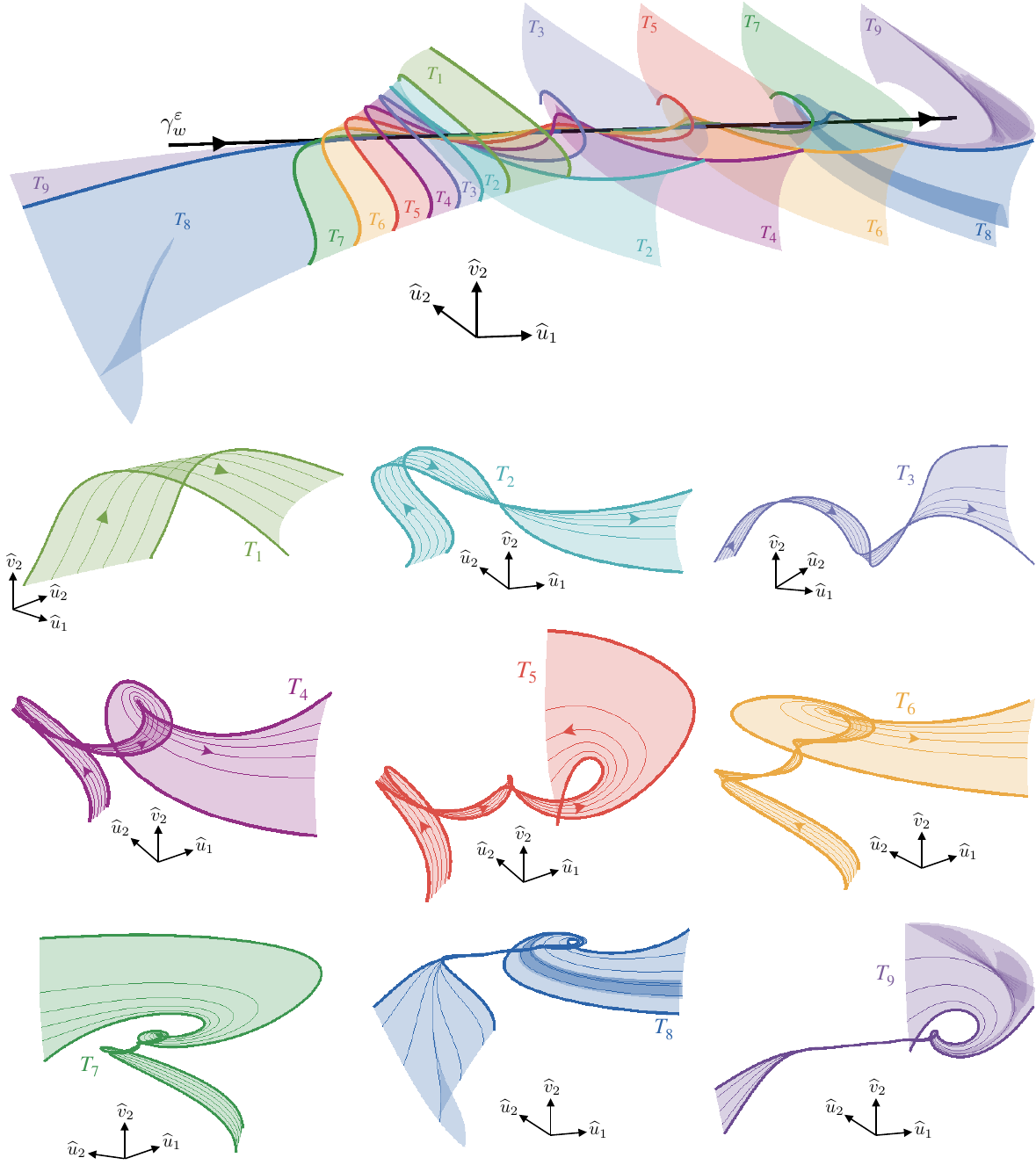}
  \put(-360,396){(a)}
  \put(-360,248){(b)}
  \put(-236,248){(c)}
  \put(-116,248){(d)}
  \put(-360,172){(e)}
  \put(-236,172){(f)}
  \put(-116,172){(g)}
  \put(-360,80){(h)}
  \put(-236,80){(i)}
  \put(-116,80){(j)}
  \caption{Attracting slow manifold, $\widetilde{S}_{a,\eps} = \cup_{i} T_i$, of the normal form \eqref{eq:nfsymasym} near the weak SFN. Here, $T_i$ denotes the subset of $\widehat{S}_{a,\eps}$ that exhibits $i$ twists about $\gamma_w^{\eps}$. The bold curves that separate the twist regions are the maximal canards, and were computed using the methods in Appendix~\ref{appendix:bvp}. The parameters are $\lambda_1=-0.275, \lambda_2 = -2$, and $\eps=0.001$ with all other parameters set to zero. Since $\mu = \tfrac{\lambda_2}{\lambda_1} \approx 7.273 \in (7,8)$, there are 9 twist sectors. Magnified views of (b) $T_1$, (c) $T_2$, (d) $T_3$, (e) $T_4$, (f) $T_5$, (g) $T_6$, (h) $T_7$, (i) $T_8$, and (j) $T_9$.}
  \label{fig:ribbons}
\end{figure}

For the attracting slow manifold shown in Fig.~\ref{fig:ribbons}, $\lambda_1 = -0.275$ and $\lambda_2 = -2$, so that the eigenvalue ratio is $\mu = \tfrac{\lambda_2}{\lambda_1} \approx 7.273 \in (7,8)$. By Proposition~\ref{prop:rotations} and as confirmed numerically in Fig.~\ref{fig:ribbons}, we have that solutions of \eqref{eq:nfsymasym} exhibit up to $9$ twists about the weak canard (which is aligned with the $\hat u_1$ axis in these coordinates). The key differences here from the classical folded node theory are (i) there is no strong canard, and (ii) the funnel region is (locally) the entire attracting manifold $S_a^{\eps}$, so that every solution on $S_{a,\eps}$ exhibits at least one twist about $\gamma_w^{\eps}$.

As is usual with slow manifolds near folded singularities, the twist number changes at the intersections of $S_{a,\eps}$ and $S_{s,\eps}$. These 1D intersection curves are solutions known as secondary maximal canards, $\gamma_i^{\eps}$, for $i=1,2,\ldots, n+1$. 
The numerical method for computing the maximal canards is given in Appendix~\ref{appendix:bvp}.
The maximal canards are ordered in the sense that solutions in the subsets of $S_{a,\eps}$ enclosed by $\gamma_{k-1}^{\eps}$ and $\gamma_k^{\eps}$ make $k$ twists about $\gamma_w^{\eps}$ in an $\mathcal O(\sqrt{\eps})$ neighbourhood of the weak SFN. We label the sector consisting of orbits with $i$ twists by $T_i$.  

The sectors of $\widehat{S}_{a,\eps}$ with odd parity (i.e., $T_1,T_3,\ldots$) leave the neighbourhood of the weak SFN with $\hat u_2 >0$, i.e., they escape in directions that favour fast jumps in oscillator 2. Conversely, the sectors of $\widehat{S}_{a,\eps}$ with even parity (i.e., $T_{2}, T_4, \ldots$) leave the neighbourhood of the weak SFN with $\hat u_2 <0$, i.e., they escape in directions that favour fast jumps in oscillator 1. 
Thus, the parity of the sector $T_i$ (also called a {\em ribbon} in \cite{Hasan2017}) is a key diagnostic for the local symmetry-breaking.

\section{Application to a model of the Eukaryotic Cell Cycle} \label{sec:modelgspt}
In this section, we apply our results to a canonical model of cell cycle regulation \cite{Dragoi2024}.
We describe the model in Section~\ref{subsec:modelconstruction}.  
Then, we examine the critical manifold and associated reduced flow in Section~\ref{subsec:modelS}. 
We study the bifurcation structure of the folded singularities in Section~\ref{subsec:strongweaksfns} and demonstrate that both strong and weak SFNs occur robustly.
We then discuss in Section~\ref{subsec:SFSNs} the geometric mechanisms by which bifurcations of the symmetric folded singularities may occur.

\subsection{Coupled oscillator model of the eukaryotic cell cycle} \label{subsec:modelconstruction}
The eukaryotic cell cycle is a sequence of complex processes through which a cell grows, replicates its DNA, and divides into two daughter cells. 
It consists of four phases \cite{Cooper2000,Morgan2007}: 
\begin{enumerate}[label=(\roman*)]
\setlength{\itemsep}{0pt}
\item G1 (gap 1): the cell grows and synthesizes proteins necessary for DNA replication. 
\item S (DNA synthesis): DNA replication occurs.
\item G2 (gap 2): proteins are synthesized in preparation for mitosis. 
\item M (mitosis and cell division): the replicated chromosomes are separated, followed by cell division into two daughter cells.
\end{enumerate}
One model that has been proposed to understand how the events of the cell cycle are controlled treats the cell cycle as a pair of arrestable and mutually-inhibiting, doubly amplified, negative feedback oscillators that control chromosome replication and segregation \cite{Dragoi2024}.

The construction of the model is as follows.
The basic oscillator, $\Omega$, consists of an activator, $A$, and an inhibitor, $I$, locked in a doubly amplified negative feedback loop. 
The governing equations are 
\begin{equation} \label{eq:single}
  \begin{split}
    \frac{dA}{dt} &= \alpha + \beta \frac{A^3}{\gamma(1+r I)+A^3}-\delta A =: f(A,I), \\ 
    \frac{dI}{dt} &= \eps \left( A + \kappa \frac{I^3}{\ell+I^3} - \rho I \right) =: \eps g(A,I), 
  \end{split}
\end{equation}
where $0< \eps \ll 1$ measures the timescale separation, and the parameters $\alpha, \beta, \gamma, \delta, \kappa, \rho, r$ and $\ell$ are all positive. Typical values for the parameter are provided in Table~\ref{tab:params}. 
Both activactor and inhibitor components have autocatalytic synthesis as described by the third-order Hill functions. However, the autocatalytic synthesis of $A$ is inhibited by $I$ via increases in the effective binding-constant term $\gamma(1+r I)$. Also, both activator and inhibitor components have constant degradation rates, $\delta$ and $\rho$. 

\begin{table}[h]
   \centering
   \begin{tabular}{|r|l|r|l|r|l|r|l|r|l|} 
   \hline
   $\alpha$ & $0.015$ & $\beta$ & $0.16$ & $\gamma$ & $0.01$ & $\delta$ & $0.2$ & $l$ & 0.07 \\ 
   \hline
   $\kappa$ & $1.95$ & $\rho$ & $3$ & $r$ & $25$ & $\eps$ & $0.01$ & {c} & {0.5} \\ 
   \hline
   \end{tabular}
   \caption{Standard parameter values taken from \cite{Dragoi2024}, except for $\eps$ and $c$, which have been modified from their original values of $\eps = 0.05$ and $c=0.25$.}
   \label{tab:params}
\end{table}

The full cell cycle model is constructed by taking two copies, $\Omega_1$ and $\Omega_2$, of the basic oscillator and coupling them by mutual inhibition. 
The model equations then are
\begin{equation} \label{eq:coupled}
  \begin{split}
    \dot{A}_1 &=  f_1(A_1,I_1) - c A_1 A_2, \\ 
    \dot{I}_1 &= \eps g_1(A_1,I_1), \\
    \dot{A}_2 &= f_2(A_2,I_2) - c A_1 A_2, \\ 
    \dot{I}_2 &= \eps g_2(A_2,I_2), 
  \end{split}
\end{equation}
which is of the form \eqref{eq:coupled_general}. 
Here, $(A_k,I_k)$ denotes the activator-inhibitor pair of oscillator $\Omega_k$, and $f_k(A_k,I_k) = f(A_k,I_k)$ and $g_k(A_k,I_k) = g(A_k,I_k)$ for $k=1,2$.  
The coupling can be viewed as each activator promoting the other's degradation. 
We take $\alpha$ (the constitutive synthesis rate of $A_i$) to be the main bifurcation parameter, whilst keeping the other parameters fixed at the values listed in Table~\ref{tab:params}.

\subsection{Critical manifold, reduced flow, and singularities} \label{subsec:modelS}
We now apply geometric singular perturbation theory to the cell cycle model \eqref{eq:coupled}. 
Since these methods have already been detailed in Section~\ref{subsec:normalformGSPT}, we will simply quote the results of their application to \eqref{eq:coupled} here.

The critical manifold, $S$, of the layer problem is the 2D surface given by
\[ S = \left\{ (A_1,I_1,A_2,I_2) : I_1 = I_{1S}(A_1,A_2) \,\, \text{ and } \,\, I_2 = I_{2S}(A_1,A_2)  \right\}, \]
where the function $I_{1S}$ is given by 
\begin{align*} 
I_{1S}(A_1,A_2) &= \frac{A_1^3(\alpha+\beta-\delta A_1)-\gamma(\delta A_1-\alpha)-c_1 A_1 A_2(\gamma+A_1^3)}{r \gamma \left( \delta A_1 - \alpha+c_1 A_1 A_2 \right)}, 
\end{align*}
and $I_{2S}(A_1,A_2) = I_{1S}(A_2,A_1)$.
For the range of parameters we consider, the layer problem possesses fold curves, $L$, that partition $S$ into attracting and saddle sheets, $S_a$ and $S_s$, respectively. 
In order to track physically relevant solutions, some care must be taken to avoid poles of the vector field and poles of the critical manifold. 
By imposing this restriction, the attracting sheet $S_a$ splits into 4 subsheets, $S_a = S_a^u \cup S_a^d \cup S_a^{\ell} \cup S_a^r$, where the superscripts indicate the sheet is {\em up} along, {\em down} along, {\em left} of, or {\em right} of the axis of symmetry $\mathcal L_s$ (Fig.~\ref{fig:modelcriticalmanifold}(a)), relative to the origin. 
Similarly, the restriction of $L$ to the physically relevant domain partitions the fold set into 4 disjoint curves, $L = L^u \cup L^d \cup L^{\ell} \cup L^r$.

\begin{figure}[h!]
  \centering
  \includegraphics[width=5in]{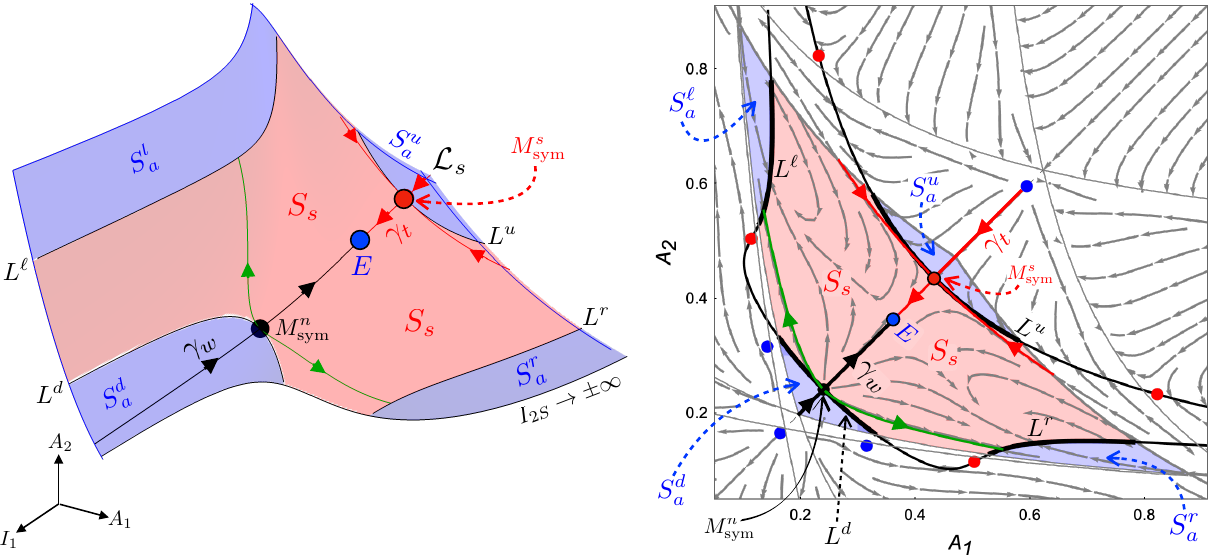}
  \put(-360,155){(a)}
  \put(-180,155){(b)}
  \caption{Geometric structure of the cell cycle model for the standard parameters with $\alpha = 0.0592$. (a) Projection of $S$ (red and blue surfaces) into the positive $(A_1,A_2,I_1)$ orthant. The fold curves ($L^u,  L^d, L^{\ell}, L^r$) separate the attracting sheets ($S_a^u, S_a^d, S_a^{\ell}, S_a^r$) from the saddle sheet $S_s$.  There is a weak SFN (black marker, $M_{\rm sym}^n$) on $L^d$, a symmetric saddle equilibrium (blue marker, $E$) on $S_s$, and a symmetric folded saddle (red marker, $M_{\rm sym}^s$) on $L^u$. 
Note: $L^u$ and $L^d$ are orthogonal to $\mathcal L_s$ at $M_{\rm sym}^s$ and $M_{\rm sym}^n$, respectively. 
  The lower thin black boundary of $S$ is the curve along which there is a pole in $I_{2S}$. (b) Projection of the reduced flow (grey arrows) onto the $(A_1,A_2)$ plane. The blue and red shaded regions are the physically relevant parts of $S$ in which $I_{1S}$ and $I_{2S}$ are both finite and positive (corresponding to the blue and red surfaces in (a)). 
The weak canard $\gamma_w$ of $M_{\rm sym}^n$ is a heteroclinic from $M_{\rm sym}^n$ to $E$.   
Similarly, the true canard $\gamma_t$ of $M_{\rm sym}^s$ is a heteroclinic from $M_{\rm sym}^s$ to $E$. 
Also shown are the ordinary and folded saddle singularities that lie outside the physically meaningful region (unlabelled blue and red markers, respectively).}
  \label{fig:modelcriticalmanifold}
\end{figure}

The desingularized reduced flow on the critical manifold is given by 
\begin{equation} \label{eq:modeldesingularized}
  \begin{split}
    \dot A_1 &= F_1(A_1,A_2) := \left( c A_1 - \tfrac{\partial f_2}{\partial A_2} \right) \tfrac{\partial f_1}{\partial I_1} g_1 - c \tfrac{\partial f_2}{\partial I_2} A_1 g_2, \\
    \dot A_2 &= F_2(A_1,A_2) := \left( c A_2 - \tfrac{\partial f_1}{\partial A_1} \right) \tfrac{\partial f_2}{\partial I_2} g_2 - c \tfrac{\partial f_1}{\partial I_1} A_2 g_1,
  \end{split}
\end{equation}
where all functions and derivatives are evaluated along $S$. 
Recall that the desingularized reduced system \eqref{eq:modeldesingularized} is topologically equivalent to the reduced flow on $S_a$, but has opposite orientation to the reduced flow on $S_s$. The projection of the reduced flow into the $(A_1,A_2)$ plane is shown in Fig.~\ref{fig:modelcriticalmanifold}(b). 

The desingularized system \eqref{eq:modeldesingularized} possesses 3 symmetric ordinary singularities and 2 asymmetric ordinary singularities.
Of these, only one of the symmetric ordinary singularities lies in the physically relevant part of phase space (Fig.~\ref{fig:modelcriticalmanifold}, blue marker $E$). For the parameters we consider, $E$ is a saddle equilibrium on $S_s$ with stable manifold along $\mathcal L_s$ and unstable manifold orthogonal to $\mathcal L_s$. 

The desingularized system \eqref{eq:modeldesingularized} possesses 2 symmetric folded singularities and 4 asymmetric folded singularities. 
In Fig.~\ref{fig:modelcriticalmanifold}, only the symmetric folded singularities lie in the physically relevant part of phase space. 
The symmetric folded singularity (black marker, $M_{\rm sym}^n$) at the cusp of $L^d$ is a weak SFN. 
(Recall that if the strong/weak eigendirection of the SFN is aligned with $\mathcal L_s$, then it is strong/weak type.)
Thus, $M_{\rm sym}^n$ has a single weak canard solution ($\gamma_w$, black arrows), which is contained in $\mathcal L_s$. 
The weak canard intersects the stable manifold of $E$ and forms a heteroclinic connection between $M_{\rm sym}^n$ and $E$. 
The solution of \eqref{eq:modeldesingularized} tangent to the strong eigendirection at $M_{\rm sym}^n$ stays entirely on $S_s$ and is not a canard (green curve in Fig.~\ref{fig:modelcriticalmanifold}).

The symmetric folded singularity (red marker, $M_{\rm sym}^s$) at the cusp of $L^u$ is a symmetric folded saddle (SFS). 
It has a single true canard solution $\gamma_t$, which lies along $\mathcal L_s$ and connects $S_a^u$ to $S_s$.
Moreover, $\gamma_t$ is a heteroclinic orbit that joins $E$ to $M_{\rm sym}^s$. 
The solution of \eqref{eq:modeldesingularized} tangent to the unstable eigendirection at $M_{\rm sym}^s$ stays entirely on $S_s$ and is not a faux canard (red curve orthogonal to $\mathcal L_s$ at $M_{\rm sym}^s$ in Fig.~\ref{fig:modelcriticalmanifold}).  

By Fenichel's theorems, we have that the attracting sheets of $S$ persist as attracting invariant slow manifolds $S_{a,\eps}^u \cup S_{a,\eps}^d \cup S_{a,\eps}^{\ell} \cup S_{a,\eps}^r$ for $\eps$ sufficiently small and positive. Similarly, the saddle sheet persists as an invariant saddle slow manifold $S_{s,\eps}$. Moreover, the local stable and unstable manifolds of $S$ perturb to nearby local stable and unstable manifolds of $S_{\eps}$. The ordinary singularity $E$ persists as a saddle equilibrium of the full 4D cell cycle model, and the slow flow on the invariant slow manifolds is an $\mathcal O(\eps)$ perturbation of the reduced flow on $S$. The dynamics in the non-hyperbolic region containing the SFN are described by canard theory (Section~\ref{sec:blowup} and \cite{Kristiansen2023,Szmolyan2001,Wex2005}). 

\subsection{Strong and weak symmetric folded nodes}  \label{subsec:strongweaksfns}
We now study the bifurcations of the folded singularities. 
To this end, we construct a bifurcation diagram in which the eigenvalue ratios, $\mu$, of the folded singularities are used as the main diagnostic (Fig.~\ref{fig:modelevalueratio}). 
More precisely, let $\lambda_1$ and $\lambda_2$ denote the eigenvalues of a folded singularity $M$ of \eqref{eq:modeldesingularized}, where $\left| \lambda_1 \right| \geq \left| \lambda_2 \right|$.  
Then, we define $\mu := \tfrac{\lambda_2}{\lambda_1}$, so that $M$ is a folded node for $\mu \in (0,1]$, folded saddle for $\mu <0$, and folded focus for $\mu \in \mathbb C$. 
For completeness, we show all folded singularities including those that may not be physically relevant. 

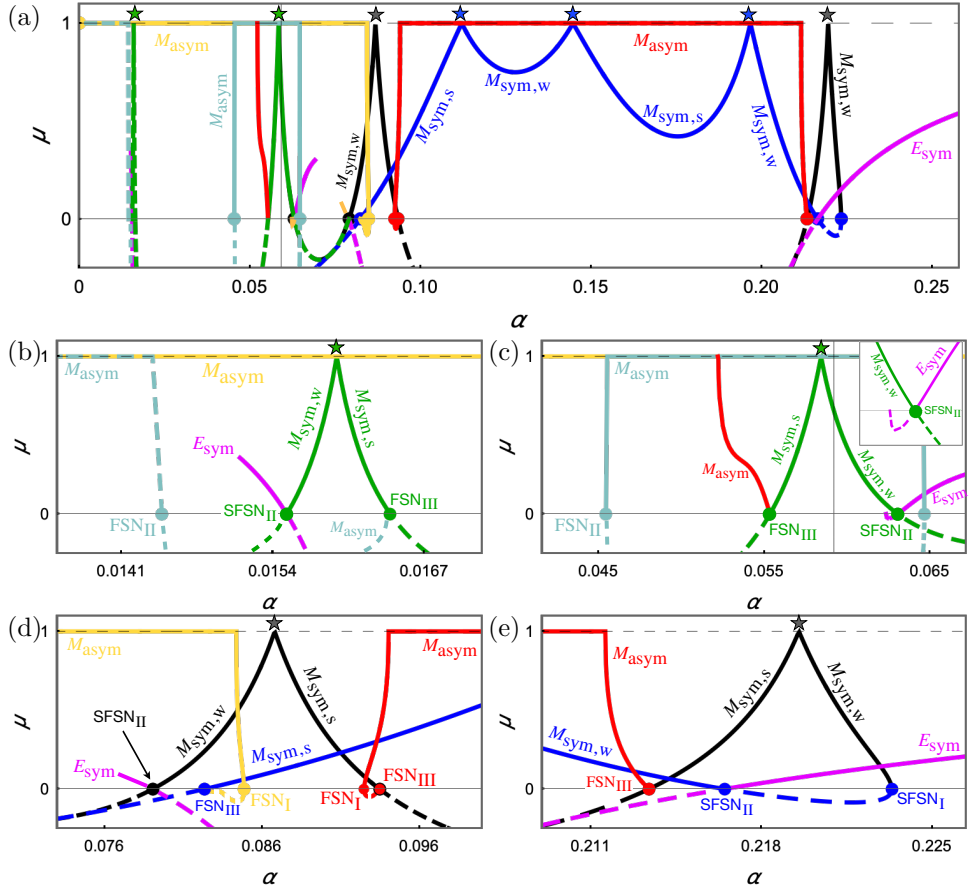
\begin{figure}[h!]
    \centering
    \includegraphics[width=5in]{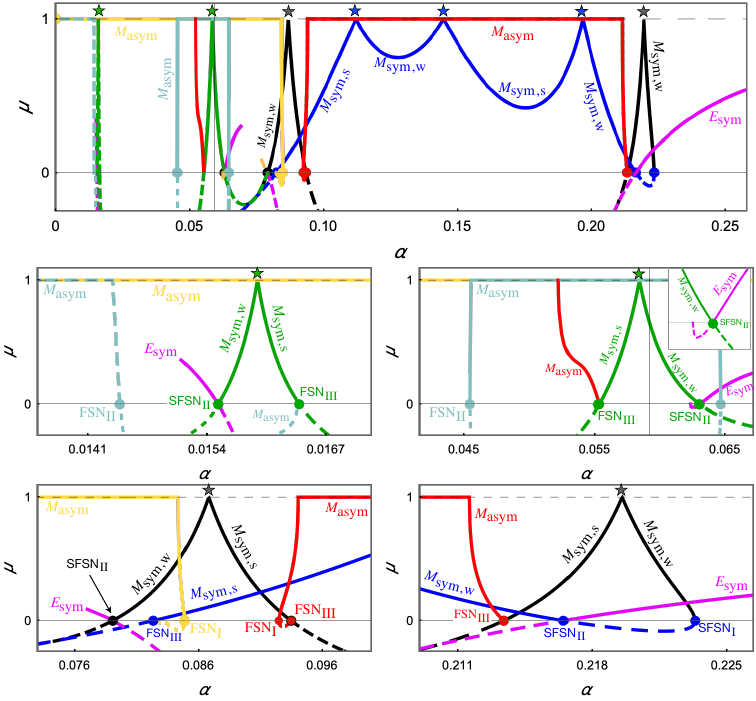}
    \put(-362,326){(a)}
    \put(-362,200){(b)}
    \put(-181,200){(c)}
    \put(-362,96){(d)}
    \put(-181,96){(e)}
    \caption{Classification of the folded singularities of \eqref{eq:modeldesingularized}. 
    (a) Eigenvalue ratio, $\mu$, of the symmetric and asymmetric folded singularities as functions of $\alpha$. 
    Where appropriate, segments of the branches of symmetric ordinary singularities (purple curves) have been included to show the transcritical and pitchfork structures of the folded saddle-node II and III bifurcations, respectively. 
    The symmetric folded singularities (green, black, and blue curves) are SFNs for $\mu \in (0,1)$, and have resonances at $\mu=1$ (star markers) in which they change from strong SFNs ($M_{\rm sym,s}$) to weak SFNs ($M_{\rm sym,w}$). 
    The asymmetric folded singularities (yellow, steel, and red curves) may be folded nodes, folded saddles, or folded foci. (To indicate folded foci, we use horizontal line segments along $\mu=1$ with the understanding that $\mu$ is actually complex.)
    Solid curves indicate that the eigenvalues have negative real parts and dashed curves indicate that at least one eigenvalue has positive real part.
    The remaining panels show the neighbourhoods of the
    (b) leftmost green SFN branch;
    (c) rightmost green SFN branch; 
    (d) leftmost black SFN branch;
    (e) rightmost black SFN branch. 
    Note: the thin black vertical lines along $\alpha = 0.0592$ in panels (a) and (c) indicate the $\alpha$ value used in Figs.~\ref{fig:modelcriticalmanifold}, \ref{fig:coupledbifnasymmetric}(b), \ref{fig:modelsectors}, and \ref{fig:deconperiodicmmos}.
    }
    \label{fig:modelevalueratio}
\end{figure}

In Fig.~\ref{fig:modelevalueratio}, the symmetric folded singularities are shown as the green, black, and blue curves (denoted $M_{\rm sym})$, whereas the asymmetric folded singularities are shown as the yellow, steel, and red curves (denoted $M_{\rm asym}$). 
The green, black, and blue $M_{\rm sym}$ branches form a single continuously connected curve that exists for $0 \leq \alpha \lesssim 0.22333$. 
That is, all of the symmetric folded singularities belong to a single self-intersecting branch, and the green, black, and blue colour coding has been introduced to make the discussion easier to follow. 
The symmetric folded singularities are SFSs for $\alpha=0$ up to $\alpha \approx 0.0155$ where there is a folded saddle-node (FSN) bifurcation of type II \cite{Krupa2010} (Fig.~\ref{fig:modelevalueratio}(b), SFSN$_{\rm II}$ point).  
This FSN II bifurcation has $\mathbb Z_2$ symmetry.  
Note: here, and for other FSN bifurcations, we include segments of the symmetric equilibrium branches to demonstrate the transcritical and pitchfork natures of the FSN bifurcations. (See also Fig.~\ref{fig:desingbifn} in Appendix~\ref{subsec:desingbifn}.)

The symmetric folded singularities are SFNs for 
\begin{itemize}
\setlength{\itemsep}{0pt}
\item $\alpha \in (0.01550, 0.01641)$, i.e., the green branch between the $\mathbb Z_2$-symmetric FSN II and FSN III bifurcations in Fig.~\ref{fig:modelevalueratio}(b). 
(See Section~\ref{subsec:SFSNs} for detailed discussion of these FSN bifurcations.) 
The SFNs are weak for $\alpha \in (0.01550, 0.01595)$ and strong for $\alpha \in (0.01595, 0.01641)$. 
The transition between weak and strong SFN occurs at $\alpha \approx 0.01595$ (Fig.~\ref{fig:modelevalueratio}(b), green star) where $\mu = 1$ and the SFN is degenerate. 
A segment of the (purple) symmetric equilibrium branch has been included to show that the $\mathbb Z_2$-symmetric FSN II is a transcritical bifurcation of a symmetric ordinary singularity and symmetric folded singularity. 
The FSN III in this case is a pitchfork bifurcation in which a pair of asymmetric folded saddles (steel $M_{\rm asym}$ branch) bifurcate from a SFN\footnote{Only one asymmetric branch can be seen emanating from the FSN III. This is because the pair of asymmetric folded singularities that emerge from the FSN III have the same eigenvalues and hence produce identical curves in the $(\alpha,\mu)$ plane.}.
\item $\alpha \in (0.05523,0.06303)$, i.e., the green branch between the FSN III and $\mathbb Z_2$-symmetric FSN II bifurcations in Fig.~\ref{fig:modelevalueratio}(c). 
They are strong for $\alpha \in (0.05523,0.05837)$, weak for $\alpha \in (0.05837,0.06303)$, and degenerate (green star) at $\alpha \approx 0.05837$.
The FSN III in this case is a pitchfork bifurcation in which a pair of asymmetric folded nodes (red $M_{\rm asym}$ branch) bifurcate from a SFS. 
The $\mathbb Z_2$-symmetric FSN II is again a transcritical bifurcation in which a symmetric ordinary singularity and a symmetric folded singularity pass through each other and exchange types (Fig.~\ref{fig:modelevalueratio}(c), inset). Note: the thin black vertical line in Fig.~\ref{fig:modelevalueratio}(c) indicates the $\alpha$ value used in Fig.~\ref{fig:modelcriticalmanifold}. At this value, the system exhibits a SFN on $L^d$ (green branch) and a SFS on $L^u$ (blue branch; outside the plot range of Fig.~\ref{fig:modelevalueratio}(c)).
\item $\alpha \in (0.07904,0.09341)$, i.e., the black branch between the black $\mathbb Z_2$-symmetric FSN II and red FSN III bifurcations in Fig.~\ref{fig:modelevalueratio}(d). 
These are weak for $\alpha \in (0.07904,0.08680)$, strong for $\alpha \in (0.08680,0.09341)$, and degenerate (black star) at $\alpha \approx 0.08680$.
Moreover, on the sub-interval $\alpha \in (0.08228,0.09341)$, i.e., between the blue FSN III and red FSN III bifurcations, the system possesses two SFNs (black and blue branches), one on $L^d$ and the other on $L^u$. The blue branch consists of strong SFNs. 
The black $\mathbb Z_2$-symmetric FSN II in this case is a transcritical bifurcation in which a symmetric node equilibrium (purple $E_{\rm sym}$ branch) and a SFS pass through each other and swap types. 
The red FSN III is a pitchfork bifurcation in which a pair of asymmetric folded saddles (red $M_{\rm asym}$ branch) emerge from a SFN. 
\item $\alpha \in (0.08228,0.21641)$, i.e., the blue branch between the FSN III bifurcation in Fig.~\ref{fig:modelevalueratio}(d) and the blue $\mathbb Z_2$-symmetric FSN II bifurcation in Fig.~\ref{fig:modelevalueratio}(e). 
As $\alpha$ is increased, the blue SFNs switch from strong to weak at $\alpha \approx 0.11180$ (Fig.~\ref{fig:modelevalueratio}(a), leftmost blue star), weak to strong at $\alpha \approx 0.14473$ (Fig.~\ref{fig:modelevalueratio}(a), middle blue star), and then strong to weak again at $\alpha \approx 0.19677$ (Fig.~\ref{fig:modelevalueratio}(a), rightmost blue star). 
The blue FSN III in Fig.~\ref{fig:modelevalueratio}(d) is a pitchfork bifurcation in which a pair of asymmetric folded saddles (yellow $M_{\rm asym}$ branch) bifurcate from a SFN. 
The blue $\mathbb Z_2$-symmetric FSN II bifurcation in Fig.~\ref{fig:modelevalueratio}(e) is another transcritical bifurcation involving a purple branch of symmetric equilibria and the blue branch of symmetric folded singularities. 
\item $\alpha \in (0.21335,0.22334)$, i.e., the black branch between the red FSN III bifurcation and the blue $\mathbb Z_2$-symmetric FSN I bifurcation in Fig.~\ref{fig:modelevalueratio}(e). These are strong for $\alpha \in (0.21335,0.21952)$, weak for $\alpha \in (0.21952,0.22334)$, and degenerate at $\alpha \approx 0.21952$.
Moreover, on the sub-interval $(0.21333,0.21645)$ enclosed by the red FSN III and blue $\mathbb Z_2$-symmetric FSN II points, the system has two SFNs (black and blue branches). The blue branch consists of weak SFNs. 
The red FSN III in this case is a pitchfork bifurcation in which a pair of asymmetric folded nodes (red $M_{\rm asym}$ branch) emerge from a SFS.
The $\mathbb Z_2$-symmetric FSN I is a saddle-node bifurcation in which a SFN and a SFS collide and annihilate each other. 
\end{itemize}
For all other $\alpha$ shown in Fig.~\ref{fig:modelevalueratio}, the symmetric folded singularities are SFSs. 

The asymmetric folded singularities (yellow, steel, and red curves) may be folded nodes, folded saddles, or folded foci depending on the value of $\alpha$. 
These asymmetric folded singularities change types at classical FSN I \cite{Vo2015}, FSN II \cite{Krupa2010}, and FSN III \cite{Roberts2015,RobertsThesis} bifurcations. 
There is direct interaction between the symmetric and asymmetric folded singularities at the various FSN III bifurcations shown in Fig.~\ref{fig:modelevalueratio}. 

\medskip
\begin{remark}
The full bifurcation structure of the desingularized system \eqref{eq:modeldesingularized} is extremely rich and intricate. 
For completeness, a more detailed discussion of this bifurcation structure is provided in  Appendix~\ref{subsec:desingbifn}. 
\end{remark}

\subsection{The $\mathbb Z_2$-symmetric folded saddle-node bifurcations}  \label{subsec:SFSNs}
As demonstrated in Fig.~\ref{fig:modelevalueratio}, the symmetric folded singularities undergo three main types of $\mathbb Z_2$-symmetric bifurcations: FSN I, FSN II, and FSN III.
Here, we examine the geometric mechanisms by which these bifurcations occur. 
To do this, we identify the singularities of the desingularized system \eqref{eq:modeldesingularized} in the $(A_1,A_2)$ plane, and track how these objects (and the nullclines that determine them) change with parameters (Fig.~\ref{fig:modelsfsns}). 
Recall that a folded singularity is a point on the (black) fold curve $L$ along which the right-hand-sides, $F_1$ and $F_2$, of the $A_1$- and $A_2$-equations in \eqref{eq:modeldesingularized} vanish (red and green curves, respectively). 
An ordinary singularity is a point at which the (red) $A_1$-, (green) $A_2$-, (cyan) $I_1$-, and (magenta) $I_2$-nullclines intersect. 

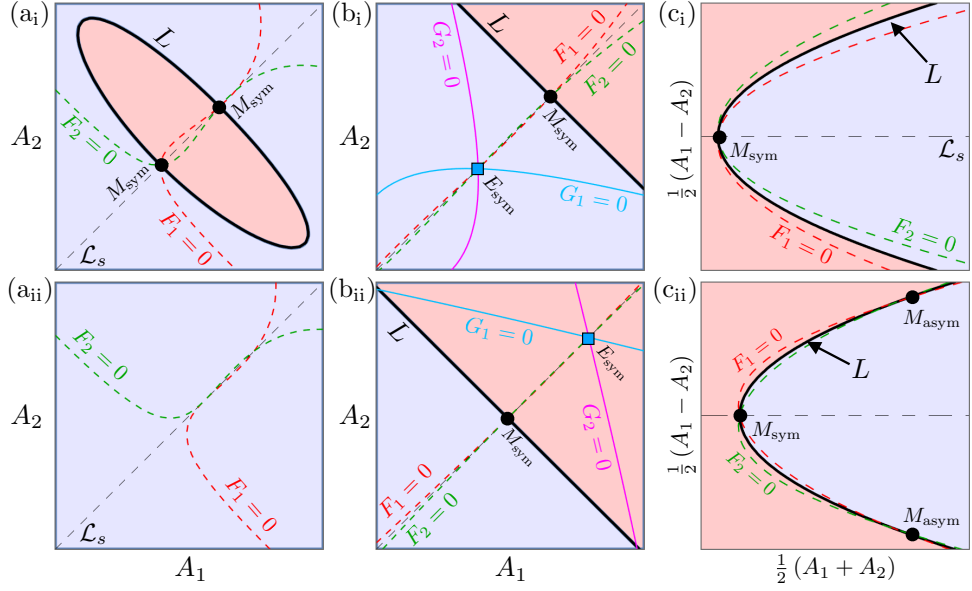
\begin{figure}[h!]
   \centering
   \includegraphics[width=5in]{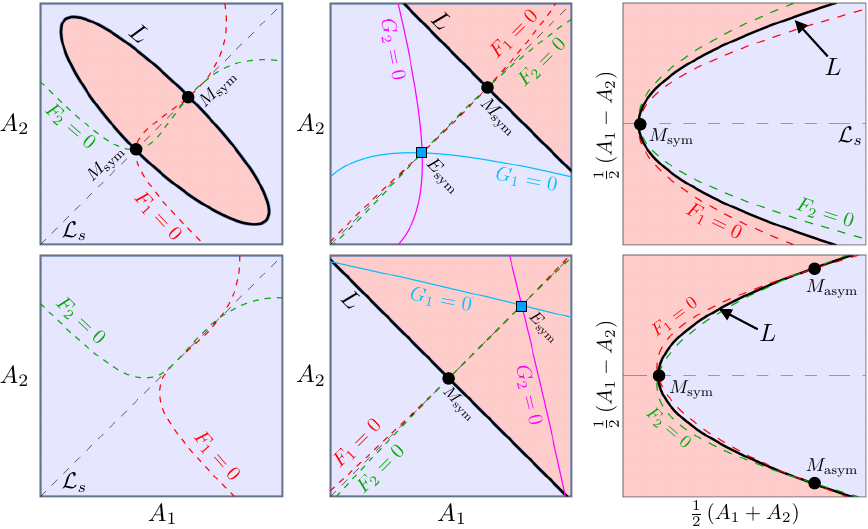}
   \put(-362,212){(a$_{\rm i}$)}
   \put(-363,108){(a$_{\rm ii}$)}
   \put(-241,212){(b$_{\rm i}$)}
   \put(-242,108){(b$_{\rm ii}$)}
   \put(-118,212){(c$_{\rm i}$)}
   \put(-120,108){(c$_{\rm ii}$)}
   \caption{Local geometry of \eqref{eq:modeldesingularized} for the $\mathbb Z_2$-symmetric FSN bifurcations. 
   The blue and red regions correspond to $S_a$ and $S_s$, respectively. 
   Folded singularities (black markers) lie in the intersections of the black fold curve $L$, dashed red $\{F_1=0\}$ contour, and dashed green $\{F_2=0\}$ contour. 
   Ordinary singularities (blue squares) lie in $\{F_1=0\} \cap \{F_2=0\} \cap \{G_1=0\} \cap \{G_2=0\}$, where $G_k := g_k(A_k,I_{kS})$ for $k=1,2$, and $\{G_1=0\}$ and $\{G_2=0\}$ are shown as cyan and magenta curves, respectively. 
   Left column: transition through the $\mathbb Z_2$-symmetric FSN I from Fig.~\ref{fig:modelevalueratio}(e). 
   (a$_{\rm i}$) Before the FSN I ($\alpha = 0.22$), $L$ is a closed curve with a SFN and SFS ($M_{\rm sym}$ points). 
   (a$_{\rm ii}$) After the FSN I ($\alpha = 0.2233$), $L$ and $M_{\rm sym}$ cease to exist. 
   Middle column: transition through the rightmost $\mathbb Z_2$-symmetric FSN II from Fig.~\ref{fig:modelevalueratio}(c).
   (b$_{\rm i}$) Before the FSN II ($\alpha = 0.0626$), $E_{\rm sym}$ is a saddle and $M_{\rm sym}$ is a SFN.
   (b$_{\rm ii}$) After the FSN II ($\alpha = 0.0632$), $E_{\rm sym}$ is a node and $M_{\rm sym}$ a SFS.
   Right column: transition through the FSN III from Fig.~\ref{fig:modelevalueratio}(d) shown in the $\left(\tfrac{A_1+A_2}{2}, \tfrac{A_1-A_2}{2} \right)$ plane.
   (c$_{\rm i}$) Before the FSN III ($\alpha = 0.082$), the $A_1$- and $A_2$-nullclines intersect $L$ once at $M_{\rm sym}$.
   (c$_{\rm ii}$) After the FSN III ($\alpha = 0.0831$), there is an extra pair of intersections away from $\mathcal L_s$ at $M_{\rm asym}$.
   }
   \label{fig:modelsfsns}
\end{figure}

The $\mathbb Z_2$-symmetric FSN I bifurcation occurs when a SFN and a SFS collide and annihilate each other. 
Geometrically, for $\alpha$ close to the $\mathbb Z_2$-symmetric FSN I value, the fold $L$ is an oval-shaped curve that encloses $S_s$ (Fig.~\ref{fig:modelsfsns}(a$_{\rm i}$)). 
A SFN lies at one extreme of $L$ and a SFS lies at the other. 
As $\alpha$ approaches the $\mathbb Z_2$-symmetric FSN I, the fold set (and $S_s$) shrinks and the symmetric folded singularities converge to a common point. 
Then, at the $\mathbb Z_2$-symmetric FSN I, the fold set and the two symmetric folded singularities coincide at a single point on $\mathcal L_s$. 
For $\alpha$ values beyond the $\mathbb Z_2$-symmetric FSN I, the fold set and the symmetric folded singularities cease to exist (Fig.~\ref{fig:modelsfsns}(a$_{\rm ii}$)).  

\medskip
\begin{remark}
The $\mathbb Z_2$-symmetric FSN I is more complicated than the classical FSN I because a sheet of $S$ is created/destroyed at the bifurcation. Whilst examples of the $\mathbb Z_2$-symmetric FSN I have been reported in other case studies (e.g., Fig.~7(b) and (c) of \cite{Awal2024}), rigorous analysis of this bifurcation has yet to be done. 
\end{remark}

The $\mathbb Z_2$-symmetric FSN II bifurcation occurs when a symmetric folded singularity $M_{\rm sym}$ and a symmetric ordinary singularity $E_{\rm sym}$ coalesce and change types in a hybrid transcritical bifurcation. On one side of the bifurcation, $E_{\rm sym}$ is a node or saddle on $S_a$ (blue region), whilst $M_{\rm sym}$ is a SFS or SFN. 
At the $\mathbb Z_2$-symmetric FSN II, $E_{\rm sym}$ and $M_{\rm sym}$ coincide. 
Then, on the other side of the bifurcation, $E_{\rm sym}$ lies on $S_s$ (red region) and switches type from node to saddle (or vice versa), whilst $M_{\rm sym}$ changes type from SFS to SFN (or vice versa). 
For the example shown in Fig.~\ref{fig:modelsfsns}(b$_{\rm i}$), $E_{\rm sym}$ is a saddle on $S_a$ and $M_{\rm sym}$ is a SFN before the bifurcation.
Then, as $\alpha$ is increased through the bifurcation, $E_{\rm sym}$ is a node on $S_s$ and $M_{\rm sym}$ is a SFS (Fig.~\ref{fig:modelsfsns}(b$_{\rm ii}$)). 

\medskip
\begin{remark}
Rigorous analysis of the $\mathbb Z_2$-symmetric FSN II was carefully performed using blow-up methods in \cite{Kristiansen2023} in a pair of linearly coupled FitzHugh-Nagumo oscillators. These results were then extended in \cite{Pedersen2026} to the class of coupled oscillators with any type of mutually inhibitory coupling.
\end{remark}

The FSN III bifurcations we find in Fig.~\ref{fig:modelevalueratio} are pitchfork bifurcations involving branches of symmetric folded singularities and asymmetric folded singularities. 
More precisely, they possess a persistent branch of symmetric folded singularities $M_{\rm sym}$ from which a pair of asymmetric folded singularities bifurcate. 
There are 2 basic sub-types:
\begin{enumerate}[label=(\roman*)]
\item Supercritical: $M_{\rm sym}$ changes from SFN to SFS, and a pair of asymmetric folded nodes bifurcates from $M_{\rm sym}$. These asymmetric folded nodes exist for the $\alpha$ values where $M_{\rm sym}$ is a SFS.
\item Subcritical: $M_{\rm sym}$ changes from SFN to SFS, and a pair of asymmetric folded saddles bifurcates from $M_{\rm sym}$. These asymmetric folded saddles exist for the $\alpha$ values where $M_{\rm sym}$ is a SFN. 
\end{enumerate} 
The green and red FSN III bifurcations in Fig.~\ref{fig:modelevalueratio}(c) and (e) are both supercritical. 
The green, blue, and red FSN III bifurcations in Fig.~\ref{fig:modelevalueratio}(b) and (d) are subcritical. 
Rigorous blow-up analysis of this class of FSN III bifurcations, in which asymmetric folded singularities bifurcate from a symmetric folded singularity, has yet to be done.

\medskip
\begin{remark}
Other types of FSN III bifurcations exist. The first instance was in a model of synaptically coupled respiratory neurons in the pre-B\"otzinger complex \cite{Best2005,Roberts2015}. There, the FSN III was a pitchfork bifurcation in which a pair of asymmetric folded singularities bifurcated from a persistent symmetric ordinary singularity $E_{\rm sym}$. In the 
\begin{enumerate}[label=(\roman*)]
\item Supercritical case: $E_{\rm sym}$ changes from stable node to saddle, and a pair of asymmetric folded nodes bifurcates from $E_{\rm sym}$. The asymmetric folded nodes exist for parameters where $E_{\rm sym}$ is a saddle. 
\item Subcritical case: $E_{\rm sym}$ changes from stable node to saddle, and a pair of asymmetric folded saddles bifurcates from $E_{\rm sym}$. The asymmetric folded saddles exist for parameters where $E_{\rm sym}$ is a stable node. 
\end{enumerate}
Further subdivisions of these FSN III bifurcations are made based on whether $E_{\rm sym}$ is a stable node on $S_a$ or $S_s$, see \cite{Roberts2015,RobertsThesis}. 
Blow-up analysis of these FSN III bifurcations was done in \cite{RobertsThesis}. 
\end{remark}

\section{Symmetry-Breaking Rhythms in the Cell Cycle Model} \label{sec:cellmodelsymmbreaking}
With the results of Section~\ref{sec:modelgspt} in hand, we now examine how the SFNs induce symmetry-breaking in the cell cycle model \eqref{eq:coupled}. We describe the dynamics in Section~\ref{subsec:modelmmos}, with particular emphasis on the symmetry-breaking attractors. Then, we combine our analysis of SFNs with the theory of canard-induced mixed-mode oscillations \cite{Brons2006} to demonstrate the origin and properties of the symmetry-breaking in Sections~\ref{subsec:modeldecon} and \ref{subsec:modelMMOdecons}.

\subsection{Broken-symmetry mixed-mode oscillations} \label{subsec:modelmmos}
The bifurcation structure of \eqref{eq:coupled} with respect to $\alpha$ is shown in Fig.~\ref{fig:coupledbifnasymmetric}. There are 3 main types of attractors: symmetric equilibria, $E_{\rm sym}$, anti-phase (AP) limit cycles in which the two oscillators $\Omega_1$ and $\Omega_2$ generate identical rhythms in phase space but half a period out of phase with each other, and broken-symmetry rhythms in which $\Omega_1$ and $\Omega_2$ exhibit qualitatively different oscillatory dynamics. 
The $\alpha$ interval shown has been restricted to focus on the broken-symmetry rhythms. A more complete account of the bifurcation structure of the cell cycle model is provided in Appendix~\ref{subsec:modeldynamics}.

\begin{figure}[h!]
  \centering
  \includegraphics[width=5in]{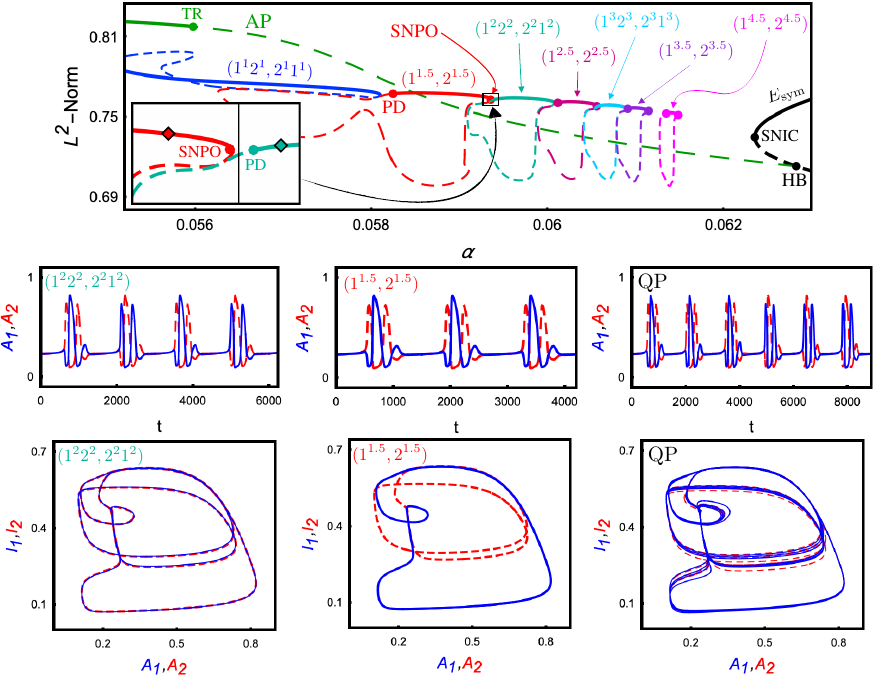}
  \put(-340,271){(a)}
  \put(-364,168){(b$_{\rm i}$)}
  \put(-242,168){(c$_{\rm i}$)}
  \put(-121,168){(d$_{\rm i}$)}
  \put(-366,96){(b$_{\rm ii}$)}
  \put(-244,96){(c$_{\rm ii}$)}
  \put(-123,96){(d$_{\rm ii}$)}
  \caption{Dynamics of system \eqref{eq:coupled}. 
  (a) Bifurcation diagram with respect to $\alpha$. The symmetric equilibria, $E_{\rm sym}$ (black curves), exist to the right of the SNIC bifurcation. The lower $E_{\rm sym}$ branch undergoes a subcritical HB, from which a (green) branch of AP limit cycles emerges. 
  The AP branch is unstable between the HB and TR bifurcations, and stable to the left of the TR bifurcation.
  Between the TR and SNIC points, the attractor is a periodic AP MMO (blue, turquoise, and cyan), periodic SSB MMO (red, maroon, purple, and pink), or a quasi-periodic SSB MMO (not shown). 
  Each MMO attractor exists on an isola, the width of which becomes increasingly narrow as $\alpha$ approaches the SNIC bifurcation. 
  Inset: zoom on a neighbourhood of the thin $\alpha$-interval between the stable plateaus of the red $(1^{1.5},2^{1.5})$ and turquoise $(1^22^2,2^2 2^2)$ branches, where the attractor is quasi-periodic. 
  The remaining panels show representative time series (middle row) and $(A,I)$ projections (bottom row) of  
  (b) a periodic AP MMO with $\alpha = 0.05946$ (turquoise diamond in (a) inset); 
  (c) a periodic SSB MMO with $\alpha = 0.0592$ (red diamond in (a) inset); and 
  (d) a quasi-periodic SSB MMO with $\alpha = 0.05937$ ($\alpha$ value indicated by the vertical black line in (a) inset).
  Blue curves correspond to oscillator $\Omega_1$ and dashed red curves correspond to oscillator $\Omega_2$.
}
  \label{fig:coupledbifnasymmetric}
\end{figure}
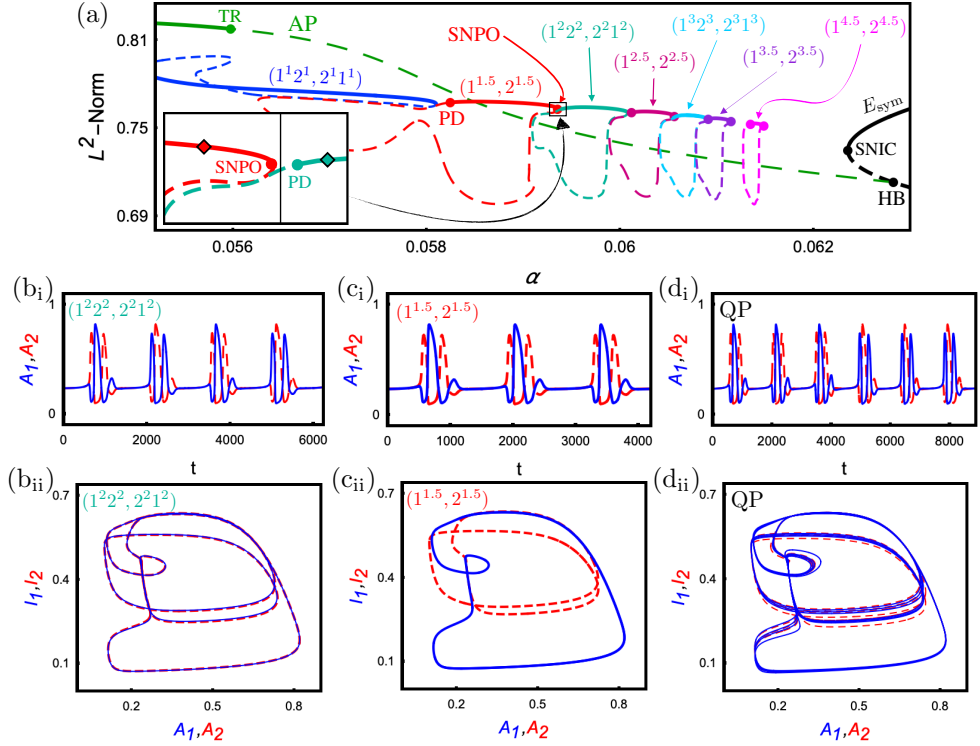

Starting at the right edge of Fig.~\ref{fig:coupledbifnasymmetric}, there is a pair of symmetric equilibria, one stable and the other unstable. 
As $\alpha$ is decreased, the symmetric equilibria collide and annihilate each other at a saddle-node on invariant circle (SNIC) bifurcation at $\alpha_{\rm SNIC} \approx 0.06235$. (The family of periodic orbits that terminate at this SNIC bifurcation exist over an extremely thin $\alpha$ interval and are not shown.)
Moreover, there is a subcritical Hopf bifurcation (HB) at $\alpha \approx 0.06282$ along the lower $E_{\rm sym}$ branch.

Emanating from the subcritical HB is a family of AP limit cycles, which are unstable between the HB and the torus (TR) bifurcation at $\alpha_{\rm TR} \approx 0.05598$. 
For $\alpha < \alpha_{\rm TR}$, the AP branch consists of AP relaxation oscillations with segments of slow variation interspersed by fast transitions (see Appendix~\ref{subsec:modeldynamics}).

For the $\alpha$ values in the interval $I_{\alpha} = (\alpha_{\rm TR},\alpha_{\rm SNIC}) = (0.05598,0.06235)$ enclosed by the TR and SNIC bifurcations, the attractor of the system is a mixed-mode oscillation (MMO).
There are three main types:
(i) periodic anti-phase (Fig.~\ref{fig:coupledbifnasymmetric}(b)),
(ii) periodic strong-symmetry breaking (Fig.~\ref{fig:coupledbifnasymmetric}(c)), and
(iii) quasi-periodic strong-symmetry breaking (Fig.~\ref{fig:coupledbifnasymmetric}(d)). 
We will show in Section~\ref{subsec:modeldecon} that these MMOs are canard-induced. 
  
The first type of MMO in $I_{\alpha}$ is the periodic AP MMO, which is a periodic solution of \eqref{eq:coupled} in which the oscillator $\Omega_1$ has two MMO events per period with signature\footnote{The Farey notation $L^s$ denotes an MMO with $L$ large-amplitude oscillations followed by $s$ small-amplitude oscillations; $(L_1^{s_1}, L_2^{s_2})$ denotes a periodic orbit in which $\Omega_1$ exhibits an $L_1^{s_1}$ MMO and $\Omega_2$ exhibits an $L_2^{s_2}$ MMO. Half-integer values of $s$ indicate the MMO has an odd number of twists, i.e., $s=k+\tfrac{1}{2}$ indicates the MMO exhibits $2k+1$ twists.} $1^{s} 2^{s}$ and the oscillator $\Omega_2$ has MMO signature $2^{s} 1^{s}$. 
A representative is shown for $s=2$ in Fig.~\ref{fig:coupledbifnasymmetric}(b), corresponding to the turquoise diamond in the inset of Fig.~\ref{fig:coupledbifnasymmetric}(a). 
The time series shows that when $\Omega_1$ exhibits its $1^s$ MMO, $\Omega_2$ exhibits its $2^s$ MMO, and vice versa. 
Since $\Omega_1$ and $\Omega_2$ are in anti-phase, they trace out the same paths in the $(A,I)$ phase plane. 

\medskip
\begin{remark}
The number of twists that the MMOs exhibit is difficult to discern from the time series. 
The geometric deconstruction (Section~\ref{subsec:modeldecon}) will provide a systematic method for determining the twist number. 
\end{remark}

In the interval $I_{\alpha}$ and for each $s \in \left\{ 1,2,\ldots \right\}$, the family of AP $(1^{s} 2^{s},2^{s} 1^{s})$ MMOs forms an isola.  
The isola has a stable plateau delimited by a period-doubling (PD) bifurcation on its left edge and a saddle-node of periodics (SNPO) bifurcation on its right edge. 
Outside of the stable plateau, the rest of the isola is unstable. 
The $s=1,2$, and $3$ isolas are shown in Fig.~\ref{fig:coupledbifnasymmetric}(a) in blue, turquoise, and cyan, respectively. 
Additional AP MMO isolas with $s>3$ exist in $I_{\alpha}$ to the right of the cyan $s=3$ isola (not shown). 
The AP MMOs are the most robust of the three MMO types. 
That is, the stable plateaus of the AP MMO isolas occupy the greatest fraction of $I_{\alpha}$. 
The widths of these AP MMO isolas shrink as $\alpha$ approaches $\alpha_{\rm SNIC}$. 

The second type of MMO in $I_{\alpha}$ is the periodic strong-symmetry breaking (SSB) MMO, which is a periodic solution of \eqref{eq:coupled} in which the oscillator $\Omega_1$ exhibits a $1^{s+1/2}$ MMO whereas $\Omega_2$ exhibits a $2^{s+1/2}$ MMO in each period. 
A representative is shown for $s=1$ in Fig.~\ref{fig:coupledbifnasymmetric}(b), corresponding to the red diamond in the inset of Fig.~\ref{fig:coupledbifnasymmetric}(a). 
The projection of this SSB MMO into the $(A,I)$ plane shows that the MMOs of $\Omega_1$ and $\Omega_2$ are structurally different since they trace out distinct paths in the $(A,I)$ plane.
As $\alpha$ is increased, the value of $s$ increases as does the duration of the interspike interval. 
In the limit as $\alpha$ approaches the SNIC bifurcation, the interspike interval becomes arbitrarily long and the state values of $\Omega_1$ and $\Omega_2$ converge to a symmetric equilibrium. 

In the interval $I_{\alpha}$ and for each $s \in \left\{ 1,2,\ldots \right\}$, the family of $(1^{s+1/2},2^{s+1/2})$ SSB MMOs forms an isola. 
These isolas have a stable segment and an unstable segment, with the stable segment delimited by a PD bifurcation on its left edge and by a SNPO bifurcation on its right edge. 
The $s=1,2,3$, and $4$ isolas are shown in Fig.~\ref{fig:coupledbifnasymmetric}(a) in red, maroon, purple, and pink, respectively. 
Additional SSB MMO isolas with $s>4$ exist in $I_{\alpha}$ to the right of the pink $s=4$ isola (not shown). 
The widths of these isolas become increasingly small as $\alpha$ is increased towards $\alpha_{\rm SNIC}$. 
Like the AP MMOs, the SSB MMOs occur robustly in $I_{\alpha}$ as their stable plateaus also occupy a substantial fraction of $I_{\alpha}$. 

Whilst the stable plateaus of the periodic AP MMOs and periodic SSB MMOs occupy most of $I_{\alpha}$, we observe that there are thin subsets of $I_{\alpha}$ in which neither is stable. 
For instance, in the thin $\alpha$-interval enclosed by the right SNPO point of the $(1^12^1,2^11^1)$ branch and the left PD point of the $(1^{1.5},2^{1.5})$ branch, the attractor is a quasi-periodic SSB MMO in which the MMOs can have 2 or 3 twists, i.e., $\Omega_1$ and $\Omega_2$ exhibit a mixture of $1^1, 2^1, 1^{1.5}$, and $2^{1.5}$ MMOs.
More generally, for each $\alpha$ interval delimited by the SNPO of the $(1^s2^s,2^s1^s)$ branch and the PD of the $(1^{s+1/2},2^{s+1/2})$ branch, the attractor of the system is quasi-periodic and features a mixture of $1^s, 2^s, 1^{s+1/2}$, and $2^{s+1/2}$ MMOs in its time series. 

For each $\alpha$ interval enclosed by the SNPO of the $(1^{s+1/2},2^{s+1/2})$ SSB MMO branch and the PD of the $(1^{s+1}2^{s+1},2^{s+1}1^{s+1})$ AP MMO branch, the attractor of the system is quasi-periodic and features a mixture of $1^{s+1/2}, 2^{s+1/2}, 1^{s+1}$, and $2^{s+1}$ MMOs in its time series.  For example, in the $\alpha$ interval enclosed by the SNPO of the $(1^{1.5},2^{1.5})$ SSB MMO branch and the PD of the $(1^22^2,2^21^2)$ AP MMO branch, the attractor is quasi-periodic (Fig.~\ref{fig:coupledbifnasymmetric}(c)) and its MMOs vary between exhibiting 3 or 4 twists, i.e., $\Omega_1$ and $\Omega_2$ exhibit a mixture of $1^{1.5},2^{1.5},1^2$, and $2^2$ MMOs in their time series.

Hence, as $\alpha$ is increased, the quasi-periodic MMOs mediate the transitions from one stable periodic AP MMO branch to the next stable periodic SSB MMO branch by increasing the twist number by one. Similarly, the transition from one stable periodic SSB MMO branch to the next stable AP MMO branch is mediated by a family of quasi-periodic MMOs that increment the twist number by one.

\subsection{Geometric deconstruction of a representative periodic SSB MMO}\label{subsec:modeldecon}
We now perform a detailed geometric deconstruction of the representative periodic SSB MMO from Fig.~\ref{fig:coupledbifnasymmetric}(c). 
Recall from Fig.~\ref{fig:modelcriticalmanifold} that the critical manifold restricted to the physically meaningful domain partitions as 
\[ S = S_a \cup L \cup S_s = S_a^d \cup L^d \cup S_a^{\ell} \cup L^{\ell} \cup S_a^u \cup L^u \cup S_a^r \cup  L^r \cup S_s. \]
Moreover, there are only three physically relevant singularities: a symmetric saddle equilibrium on $S_s$, a SFN on $L^d$, and a SFS on $L^u$. 
For the parameter values in Fig.~\ref{fig:coupledbifnasymmetric}(c), the SFN is weak with eigenvalues $\lambda_1 \approx -0.0053$ and $\lambda_2 \approx -0.0035$. 
Thus, the eigenvalue ratio is $\mu \approx 1.52$. 
By Proposition~\ref{prop:rotations}, solutions passing near the weak SFN can exhibit up to three twists, where a twist corresponds to half a rotation. 

\begin{figure}[h!]
   \centering
   \includegraphics[width=5in]{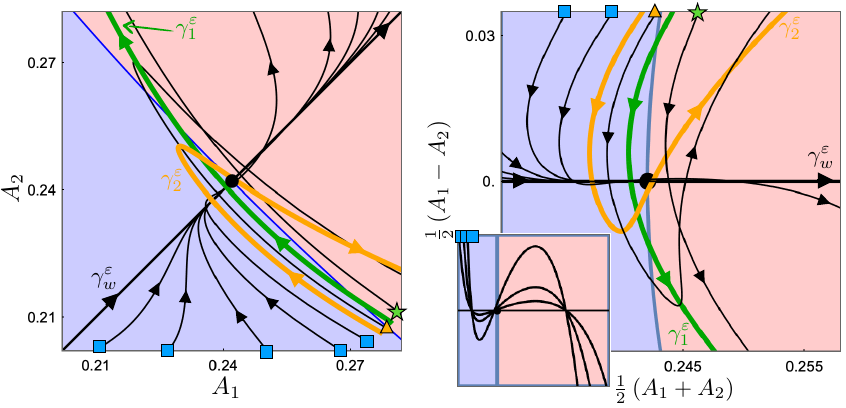}
   \put(-360,165){(a)}
   \put(-180,165){(b)}
   \caption{Local twisting of solutions about $\gamma_w^{\eps}$ for $\alpha = 0.0592$. Since the SFN (black marker) is weak with $\mu \approx 1.52$, there are up to $3$ twists. The secondary canards $\gamma_1^{\eps}$ (green) and $\gamma_2^{\eps}$ (orange) separate the subsets of $S_{a,\eps}^{d}$ with different numbers of twists. The star, triangle, and square markers correspond to initial conditions in the different twist ribbons. (a) Projection into the $(A_1,A_2)$ plane. (b) Projection into the $\left( \tfrac{A_1+A_2}{2}, \tfrac{A_1-A_2}{2} \right)$ plane. Inset: zoom on a small tube of size $6.5 \times 10^{-5}$ around $\gamma_w^{\eps}$.}
   \label{fig:modelsectors}
\end{figure}

We demonstrate the local twisting behaviour in Fig.~\ref{fig:modelsectors} for the half of $S_{a,\eps}^d$ below the axis of symmetry $\mathcal L_s$.
First, by Proposition~\ref{prop:rotations}, the weak canard persists as a maximal canard $\gamma_w^{\eps} \subset \mathcal L_s$. 
In addition to $\gamma_w^{\eps}$, there are two secondary maximal canards $\gamma_1^{\eps}$ and $\gamma_2^{\eps}$ (green and orange curves, respectively).
Solutions in the funnel of the weak SFN with initial conditions that lie to the right of $\gamma_1^{\eps}$ exhibit one twist around $\gamma_w^{\eps}$ (Fig.~\ref{fig:modelsectors}, solution with initial condition at the green star).
Solutions on $S_{a,\eps}^{d}$ enclosed by $\gamma_1^{\eps}$ and $\gamma_2^{\eps}$ exhibit two twists around $\gamma_w^{\eps}$ (Fig.~\ref{fig:modelsectors}, solution with initial condition at the orange triangle).
Solutions on $S_{a,\eps}^{d}$ enclosed by $\gamma_2^{\eps}$ and $\gamma_w^{\eps}$ exhibit three twists around $\gamma_w^{\eps}$ (Fig.~\ref{fig:modelsectors}, solutions with initial conditions at the blue squares).
We observe that most solutions in this maximal twist sector rapidly converge to $\gamma_w^{\eps}$ and follow $\gamma_w^{\eps}$ closely until they escape. 
That is, for most solutions in the maximal twist sector, the twists are not visible. (The vertical scale in the inset of Fig.~\ref{fig:modelsectors}(b) is $[-6.5 \times 10^{-5},6.5\times 10^{-5}]$, which shows that whilst the twists are present, they occur on a very small scale.)

\begin{figure}[h!]
    \centering
    \includegraphics[width=5in]{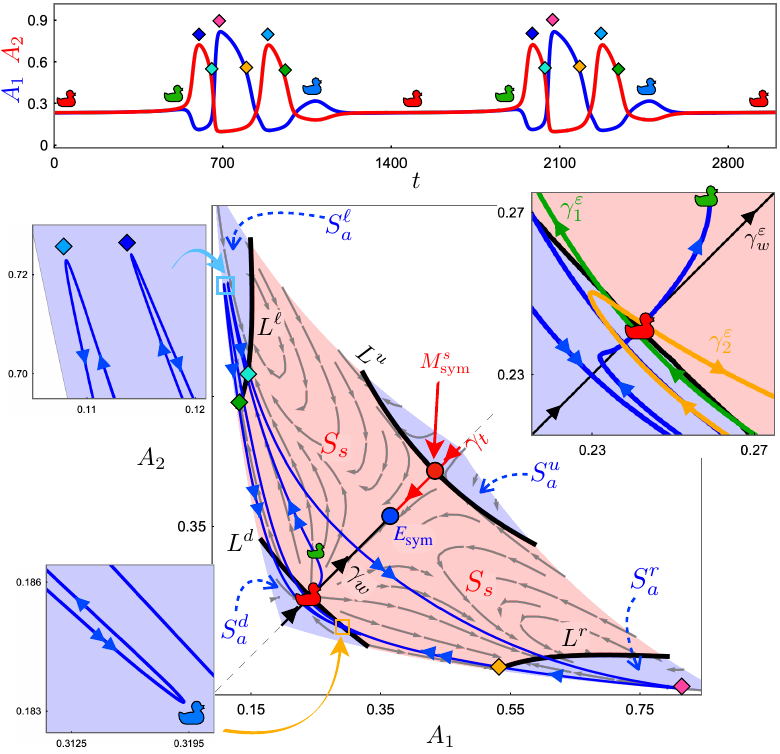}
    \put(-362,340){(a)}
    \put(-362,250){(b)}
    \caption{Geometric deconstruction of the periodic SSB $(1^{1.5},2^{1.5})$ MMO for $\alpha = 0.0592$. (a) Time series showing two full periods together with markers corresponding to key points in the geometric deconstruction. (b) Projection of the dynamics, restricted to the physically meaningful domain, into the $(A_1,A_2)$ plane together with the SSB MMO attractor (blue trajectory). 
    The geometry of the critical manifold and the reduced flow are the same here as in Fig.~\ref{fig:modelcriticalmanifold}.
    Upper right inset: zoom on a small neighbourhood of the weak SFN (red ducky). 
    Upper left inset: zoom on the neighbourhood of $S_a^{\ell}$ where the solution is projected to (diamond markers) after its fast jump from either the yellow diamond on $L^r$ or from the green duck on $S_s$.
    Lower left inset: zoom on the neighbourhood of $S_a^d$ where the solution is projected to after its fast jump from the green diamond on $L^{\ell}$.}
    \label{fig:deconperiodicmmos}
\end{figure}

With the twisting behaviour of $S_{a,\eps}^d$ near the weak SFN resolved, we proceed now to deconstruct the representative SSB MMO. It consists of 9 distinct segments (Fig.~\ref{fig:deconperiodicmmos}).
\begin{enumerate}[label=(\roman*)]
\setlength{\itemsep}{0pt}
\item Starting at the weak SFN (red ducky), there is a slow drift (Fig.~\ref{fig:deconperiodicmmos}(b), upper right inset, single blue arrow) along $\gamma_w^{\eps}$, corresponding to a portion of the interspike interval. 
\item When the orbit is a sufficient distance from the SFN, it peels off $\mathcal L_s$ and initiates a fast jump at the green ducky on $S_s$ and lands at the dark blue diamond on $S_a^{\ell}$ (Fig.~\ref{fig:deconperiodicmmos}(b), upper left inset). This fast jump corresponds to a fast up-jump in oscillator $\Omega_2$.
\item Then, the solution slowly drifts along $S_a^{\ell}$ from the dark blue diamond to the cyan diamond on $L^{\ell}$.  
\item Subsequently, there is a fast jump from the cyan diamond to the magenta diamond on $S_a^r$, corresponding to a fast up-jump in oscillator $\Omega_1$ and a fast down-jump in oscillator $\Omega_2$. 
\item From there, the solution drifts along $S_a^r$ until it reaches the yellow diamond on $L^r$. 
\item Next, the solution has a fast jump to the light blue diamond on $S_a^{\ell}$ (Fig.~\ref{fig:deconperiodicmmos}(b), upper left inset), corresponding to a fast down-jump in $\Omega_1$ and a fast up-jump in $\Omega_2$.
\item Then, the solution slowly drifts along $S_a^{\ell}$ until it reaches the fold curve $L^{\ell}$ at the green diamond. 
\item The solution then executes a fast jump (corresponding to a fast down-jump in $\Omega_2$) that lands at a point in the funnel of the weak SFN on $S_a^d$ (Fig.~\ref{fig:deconperiodicmmos}(b), lower left inset, blue ducky). In fact, the solution lands inside the maximal twist sector, i.e., below the maximal canard $\gamma_2^{\eps}$. Since the solution is in the portion of the funnel that belongs to the lower half of $S_a^d$ and the twist number is 3 (i.e., is odd), the solution must eventually escape $\gamma_w^{\eps}$ in the direction of $S_a^{\ell}$ (which we have already confirmed in segment (ii) at the green ducky).  
\item The solution converges to $\gamma_w^{\eps}$ (Fig.~\ref{fig:deconperiodicmmos}(b), upper right inset, single blue arrow) and closely follows it, corresponding to the initiation of the interspike interval. The solution closely tracks $\gamma_w^{\eps}$ until it returns to the SFN (red ducky), which completes the cycle. 
\end{enumerate}
Thus, we have shown that the SSB MMO consists of a local mechanism (the weak SFN) which is responsible for the (very) small-amplitude oscillations around $\gamma_w^{\eps}$ and interspike intervals, combined with a global return mechanism which is responsible for re-injecting orbits into the funnel of the weak SFN. 
Thus, the SSB MMO is canard-induced \cite{Brons2006}.
Moreover, the symmetry breaking is due to the local twisting of the slow manifolds in the neighbourhood of the weak SFN. 
Whether the fast jump away from $\gamma_w^{\eps}$ is in the direction of $S_a^{\ell}$ or $S_a^r$ (i.e., which of $\Omega_1$ and $\Omega_2$ will execute the fast jump) is determined by the parity of the twist number and by whether the solution is in the portion of the funnel above or below the axis of symmetry $\mathcal L_s$.

\subsection{Deconstructions for the other MMO types}\label{subsec:modelMMOdecons}
Having carefully presented the geometric deconstruction of the periodic SSB $(1^{1.5},2^{1.5})$ MMO, we now examine the deconstructions for the other MMOs in the interval $I_{\alpha} = (\alpha_{\rm TR}, \alpha_{\rm SNIC}) = (0.05598,0.06235)$. 
We first recall from Section~\ref{subsec:strongweaksfns} that SFNs exist for all $\alpha \in (0.05523,0.06303) \supset I_{\alpha}$.
Thus, since an SFN exists for all $\alpha \in I_{\alpha}$, the three classes of MMOs in $I_{\alpha}$ can all be shown to be canard-induced. 

For each $s \in \left\{ 2,3,\ldots \right\}$, the deconstruction of the periodic SSB $(1^{s+1/2},2^{s+1/2})$ MMO is similar to the one presented for the $s=1$ case in Section~\ref{subsec:modeldecon}. 
That is, each of the periodic SSB $(1^{s+1/2},2^{s+1/2})$ MMOs is canard-induced with a SFN that provides the local twisting and symmetry-breaking mechanism. 
The main difference between the different families of periodic SSB MMOs is that as $\alpha$ is increased, the eigenvalue ratio of the SFN decreases (recall Fig.~\ref{fig:modelevalueratio}(c)) and the maximal number of twists increases. 
Since the global return injects solutions into the maximal twist sector, the magnitudes of the small-amplitude oscillations around the weak canard are usually too small to be observed. 
However, the impact of having more and more twists around $\gamma_w^{\eps}$ is that the solution spends longer and longer times closely following $\gamma_w^{\eps}$ on $\mathcal L_s$. 
Consequently, we observe longer interspike intervals in the time series of the solution as $\alpha$ is increased. 
%

For each $s \in \left\{ 1,2,\ldots \right\}$, the deconstruction of the periodic AP $(1^s2^s,2^s 1^s)$ MMO also consists of a local symmetry-breaking twisting mechanism (the SFN on $L^d$) together with a global return mechanism. 
For the AP MMO attractor, the global return alternately projects obits to a point $p$ on the lower half of $S_a^d$ (below $\mathcal L_s$) in the funnel of the SFN or to its reflection $\mathcal R p$ (above $\mathcal L_s$), which is also in the funnel of the SFN. 
We show the projection of the periodic AP $(1^22^2,2^2 1^2)$ MMO in the $(A_1,A_2)$ plane in Fig.~\ref{fig:deconapqp}(a) but refrain from listing out the segments of the deconstruction, since the details are very similar to the deconstruction of the periodic SSB MMO from Section~\ref{subsec:modeldecon}. 

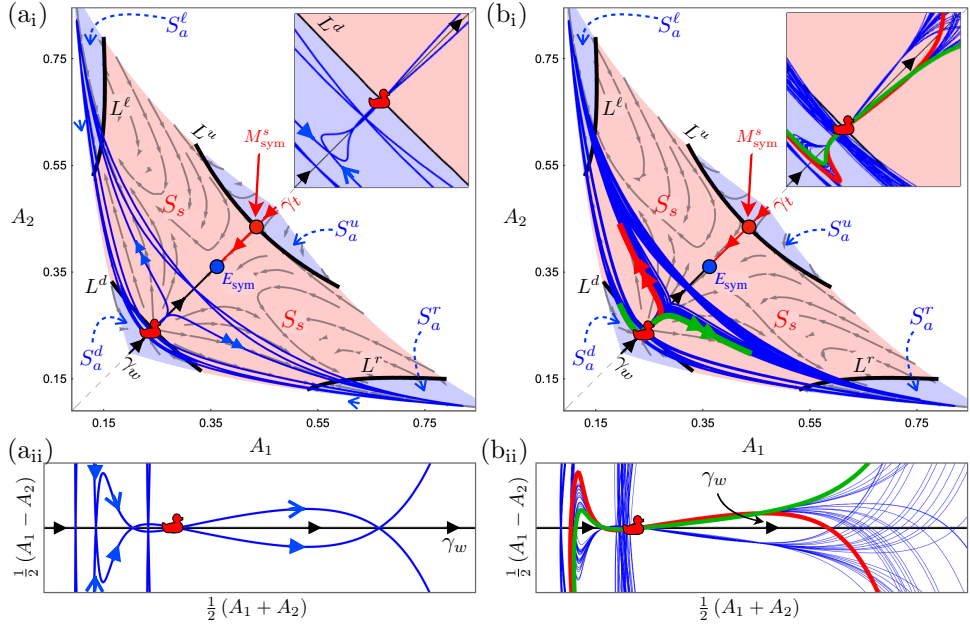
\begin{figure}[h!]
  \centering
  \includegraphics[width=5in]{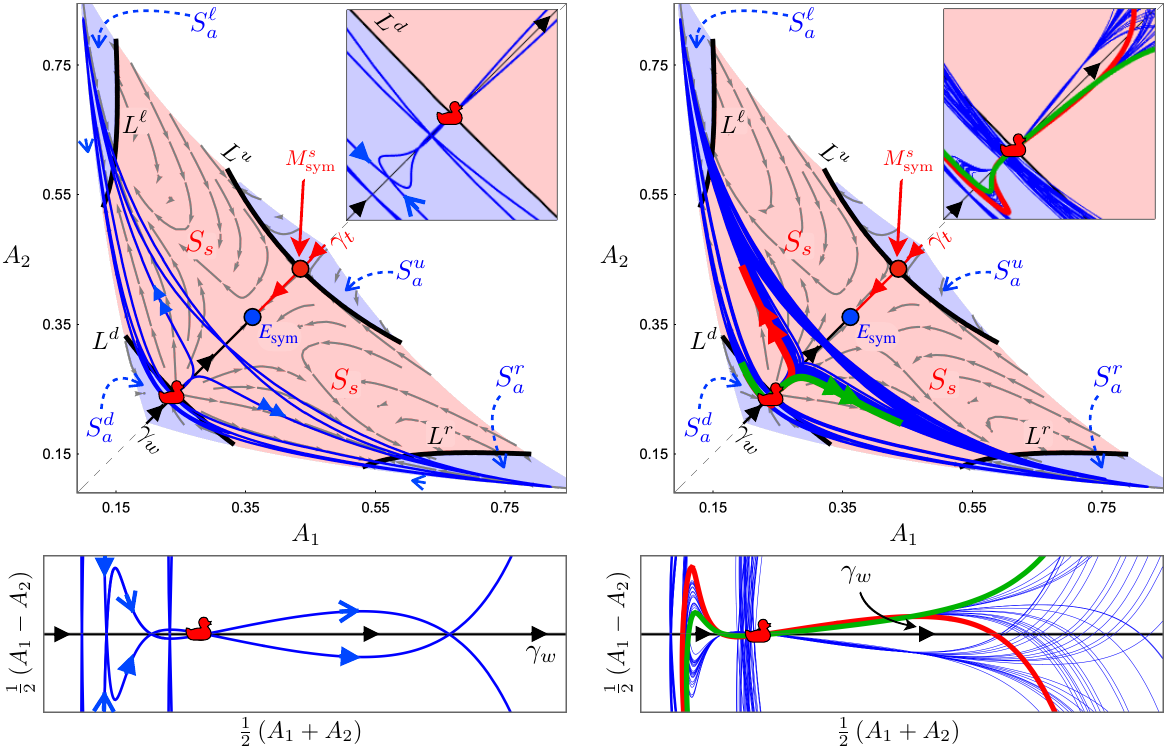}
  \put(-362,226){(a$_{\rm i}$)}
  \put(-362,62){(a$_{\rm ii}$)}
  \put(-182,226){(b$_{\rm i}$)}
  \put(-184,62){(b$_{\rm ii}$)}
  \caption{Deconstructions of a representative periodic AP MMO (left column) for $\alpha = 0.05946$ and a representative quasi-periodic SSB MMO for $\alpha = 0.05937$ (right column). 
  Top row: projections into the $(A_1,A_2)$ plane.
  Bottom row: projections into the $(\tfrac{A_1+A_2}{2},\tfrac{A_1-A_2}{2})$ plane in a small neighbourhood of $\gamma_w$. 
  (a) The global return alternately projects orbits to a point $p$ (resp., $\mathcal R p$) in the funnel in the lower (resp., upper) half of $S_a^d$.
  (b) The global return projects solutions to points in the funnel that are very close to, and straddle, a maximal canard. Because of the exponential sensitivity of solutions near maximal canards, the escape of solutions may be in either the direction of $S_a^{\ell}$ (red orbit segment) or in the direction of $S_a^r$ (green orbit segment). This also implies that solutions that escape in different directions have different numbers of twists, since they land on different sides of a maximal canard.
  }
  \label{fig:deconapqp}
\end{figure}

\medskip
\begin{remark}
From Fig.~\ref{fig:coupledbifnasymmetric}, the AP $(1^12^1,2^11^1)$ MMO family exists for $0.05505 \lesssim \alpha \lesssim 0.05810$. 
We also observe that the SFN on $L^d$ is strong type for $\alpha_{\rm FSN\, III} \lesssim \alpha \lesssim \alpha_{\rm DSFN}$ (Section~\ref{subsec:strongweaksfns}), where $\alpha_{\rm FSN\, III} = 0.05523$ denotes a FSN III bifurcation of the desingularized system and $\alpha_{\rm DSFN} = 0.05837$ denotes a degenerate SFN, where the SFN changes type from strong to weak.
Thus, the family of AP $(1^12^1,2^11^1)$ MMOs are examples of rhythms with strong SFNs that exhibit two twists around the strong canard. 
Moreover, the rightmost boundary of the $(1^12^1,2^11^1)$ MMO isola occurs at $\alpha \approx 0.05810 = \alpha_{\rm DSFN} + \mathcal O(\eps^2)$. 
Similarly, the leftmost boundary of the $(1^12^1,2^11^1)$ MMO isola occurs at $\alpha \approx 0.05505 = \alpha_{\rm FSN \, III} + \mathcal O(\eps^2)$. 
\end{remark}

The deconstruction of the quasi-periodic SSB MMO is similar to that of the periodic SSB MMO and of the periodic AP MMO. 
The main difference to the other MMO types is that the global return reinjects solutions into the funnel at points that are very close to, and straddle, a maximal canard $\gamma_n^{\eps}$. 
Consequently, the solution can exhibit MMOs with either $n$ or $n+1$ twists, depending on which side of $\gamma_n^{\eps}$ the global return projects the solution to.
An example is shown in Fig.~\ref{fig:deconapqp}(b). 
The green orbit segment corresponds to a portion of the quasi-periodic solution that lands close to, but just to the right of, the maximal canard $\gamma_3^{\eps}$ (not shown) and hence exhibits $3$ twists. 
Since the orbit approaches $\gamma_w^{\eps}$ from the upper half of $S_a^d$ (i.e., above $\mathcal L_s$) and the parity of the twist is odd, the subsequent fast jump is in the direction of $S_a^r$. 
The red orbit segment corresponds to a portion of the solution that lands close to, but just left of, the maximal canard $\gamma_3^{\eps}$ and hence exhibits $4$ twists. 
Since the orbit approaches $\gamma_w^{\eps}$ from the upper half of $S_a^d$ and the parity of the twist is even, the subsequent fast jump is in the direction of $S_a^{\ell}$.

Finally, we observe that the SNIC bifurcation at $\alpha_{\rm SNIC} \approx 0.06235$ where the MMO families terminate occurs at the same $\alpha$ value as a saddle-node bifurcation of ordinary singularities of the desingularized system \eqref{eq:modeldesingularized}.
That is, for $\alpha < \alpha_{\rm SNIC}$ the desingularized system has no ordinary singularities near the SFN (Fig.~\ref{fig:modelsnic}(a)).  
At $\alpha = \alpha_{\rm SNIC}$, the slow nullclines $\{ G_1=0 \}$ and $\{ G_2 = 0 \}$ touch at a single symmetric equilibrium point on the axis of symmetry within $S_a$.
Then, for $\alpha > \alpha_{\rm SNIC}$ the intersection $\{ G_1=0 \} \cap \{ G_2 = 0 \}$ perturbs to a pair of symmetric ordinary singularities on $S_a$ near the SFN (Fig.~\ref{fig:modelsnic}(b)). One of these is a saddle (red diamond) and the other is a stable node (blue square). 
The stable (red arrows) and unstable manifolds of the saddle are orthogonal and parallel to $\mathcal L_s$, respectively. 
The strong stable (blue, double arrows) and weak stable (blue, single arrows) manifolds of the stable node are orthogonal and parallel to $\mathcal L_s$, respectively. 

\begin{figure}[h!]
   \centering
   \includegraphics[width=5in]{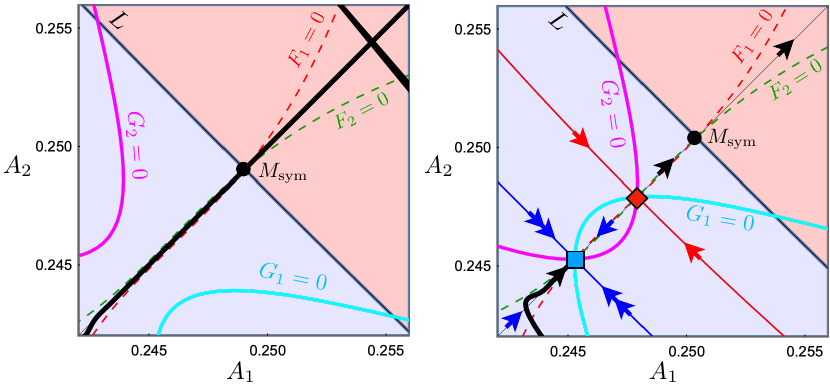}
   \put(-360,160){(a)}
   \put(-180,160){(b)}
   \caption{Dynamics near the SFN for $\alpha \approx \alpha_{\rm SNIC}$. (a) For $\alpha = 0.062 < \alpha_{\rm SNIC}$, the slow nullclines (cyan and magenta curves) have no intersections. As such, there is nothing to prevent solutions from tunnelling through the SFN along the weak canard $\gamma_w^{\eps} \subset \mathcal L_s$. The attractor (thick black curve) is a quasi-periodic SSB MMO with small-amplitude oscillations and interspike intervals generated by the canard dynamics around the weak SFN. (b) For $\alpha = 0.0625 > \alpha_{\rm SNIC}$, there is a symmetric saddle (red diamond) and symmetric stable node (blue square) on $S_a$. The funnel of the SFN is enclosed by the fold curve $L$ and by the (red) stable manifold of the saddle. Most solutions are projected into the basin of attraction of the stable node and hence the attractor is a symmetric equilibrium.}
   \label{fig:modelsnic}
\end{figure}

Thus, the mechanism by which the SNIC bifurcation occurs is as follows. 
Since there are no ordinary singularities near the SFN for $\alpha<\alpha_{\rm SNIC}$, solutions can tunnel through the weak SFN region and the geometric deconstructions discussed above continue to hold.
As such, the attractor of the coupled system for $\alpha<\alpha_{\rm SNIC}$ is a quasi-periodic SSB MMO. 
Then, for $\alpha>\alpha_{\rm SNIC}$, the stable manifold of the saddle equilibrium creates a boundary that severely limits the width of the funnel of the SFN.
Thus, any solution that lands on $S_a^d$ to the left of the stable manifold of the saddle lies in the basin of attraction of the stable node equilibrium. 
Due to this geometry, solutions may only transiently exhibit symmetry-breaking MMOs if they are initialized in the funnel of the SFN. 
Otherwise, the global return mechanism eventually projects solutions directly into the basin of attraction of the stable node and solutions converge to the symmetric stable node equilibrium.

\section{Discussion}  \label{sec:discussion}
In this article, we combined techniques from geometric singular perturbation theory and blow-up methods to analyse a novel symmetry-breaking mechanism --the $\mathbb Z_2$-symmetric folded node-- in coupled identical slow/fast oscillator networks with strong nonlinear mutually inhibitory coupling. 
We showed that the SFN is a folded node singularity that occurs at a cusp bifurcation of the layer problem, cf. cusped node singularities in  \cite{Kristiansen2023}. 
We established that at the SFN, the fold curve of the critical manifold is orthogonal to the axis of symmetry $\mathcal L_s$. 
We also showed that the SFN has two types: strong and weak. 
In both cases, the eigenvectors (of the desingularized system) at the SFN are parallel and orthogonal to $\mathcal L_s$. 
The SFN is strong if the strong eigendirection is the one aligned with $\mathcal L_s$ and it is weak if the weak eigendirection is the one aligned with $\mathcal L_s$. 
These distinctions are important since only the solution corresponding to the eigendirection aligned with $\mathcal L_s$ is a canard solution. 

For the strong SFN, the only singular canard is the strong canard $\gamma_0$.
The strong canard persists as a maximal canard $\gamma_0^{\eps}$ and, in contrast to classical folded nodes, it serves as the axis of rotation for the twisting of solutions. 
Moreover, we established that all solutions on the subset of the attracting sheet $S_a$ of the critical manifold bounded by the fold curve $L$ belong to  the funnel of the SFN. 
As such, every solution on $S_a$ bounded by $L$ exhibits at least one twist about $\gamma_0^{\eps}$.
We also demonstrated numerically that, in addition to the twisting about $\gamma_0^{\eps}$, there is a secondary twisting of solutions about an as yet undetermined secondary axis of rotation. 
We conjecture that this secondary axis of rotation is the solution that corresponds in the singular limit to the solution that is tangent at the SFN to the weak eigendirection (and hence is orthogonal to $\mathcal L_s$). 
We provided preliminary blow-up analysis that indicates that there is a distinguished curve of nilpotent fixed points in the blown-up vector field.  
We postulate that a cylindrical blow-up of this distinguished nilpotent curve will allow us to study the dynamics of the secondary rotations.

For the weak SFN, we showed that the only singular canard is the weak canard $\gamma_w$.
The weak canard persists as a maximal canard $\gamma_w^{\eps}$ and it serves as the axis of rotation for the twisting of solutions.  
We showed that all solutions on the subset of $S_a$ enclosed by $L$ are part of the funnel of the weak SFN. 
As such, every solution on $S_a$ bounded by $L$ exhibits at least one twist about $\gamma_w^{\eps}$.
In addition to the primary weak canard, there exist $n$ secondary maximal canards $\gamma_i^{\eps}$ that partition the slow manifolds into ribbons, where solutions in the ribbon enclosed by $\gamma_{i}^{\eps}$ and $\gamma_{i+1}^{\eps}$ exhibit $i$ twists about $\gamma_w^{\eps}$. 
The parity of the ribbon, i.e., the number of twists that it exhibits, determines the direction of the subsequent fast jump, and plays an important role in symmetry-breaking.

We then demonstrated our theoretical results in a model of the eukaryotic cell cycle \eqref{eq:coupled}. The model consists of two identical oscillators, each with slow/fast dynamics, coupled by strong nonlinear mutual inhibition. 
The attractors of \eqref{eq:coupled} may be symmetric equilibrium states, anti-phase limit cycles, or symmetry-breaking mixed-mode oscillations. 
By applying the theory of canard-induced mixed mode oscillations together with our theory for SFNs, we showed that the MMOs can be decomposed into a local mechanism and a global return mechanism. 
The local mechanism consists of the SFN on the fold curve $L^d$, which induces local twisting and ribboning of the slow manifolds. 
The parity of the twists/ribbons determines the direction of the subsequent fast jump and plays a key role in the symmetry-breaking. 
The global return mechanism consists of the fast jumps from the fold curves $L^{\ell}$ and $L^r$ that project orbits into the funnel of the SFN on $S_a^d$.
If the global return projects orbits to points in the funnel far from a maximal canard, then the attractor is a periodic AP MMO or a periodic SSB MMO. 
However, if the global return projects points into the funnel very close to a maximal canard, then the exponential sensitivity of maximal canards alters the landing point of the global return in each MMO event and the resulting attractor is quasi-periodic. 
Hence, the crossing of the solution of \eqref{eq:coupled} over a maximal canard is the mechanism by which the system may transition from a periodic AP MMO attractor to a periodic SSB MMO attractor, via a quasi-periodic MMO attractor. 

In the course of our analysis of the coupled cell model, we showed that the symmetric folded singularities may be symmetric folded nodes, symmetric folded saddles, or symmetric folded saddle-nodes, depending on parameters. 
In particular, the symmetric folded saddle-nodes were the $\mathbb Z_2$-symmetric analogues of the classical FSN I, FSN II, and FSN III bifurcations. 
Whilst the $\mathbb Z_2$-symmetric FSNs bear many similarities to their classical counterparts, they do possess distinct features worth examining in future work. 
For instance, in the $\mathbb Z_2$-symmetric FSN I bifurcation, an entire sheet of the critical manifold vanishes at the same time as the symmetric folded singularities. 
The $\mathbb Z_2$-symmetirc FSN II bifurcation was carefully analysed in \cite{Kristiansen2023} in the context of linearly coupled FitzHugh-Nagumo oscillators. 
Additionally, the class of FSN III bifurcations that we reported on here involve pitchfork bifurcations of the desingularized system in which pairs of asymmetric folded singularities bifurcated from a symmetric folded singularity. To the best of our knowledge, this class of FSN III bifurcations has not been reported before and their rigorous analysis has yet to be performed.

Symmetrically coupled identical oscillators were also used in \cite{gandhi2025conceptual} to investigate SSB dynamics in the context of eukaryotic cell cycle regulation. In that work, SSB was shown to be governed by a homoclinic bifurcation associated with a symmetric equilibrium. A key difference is that the uncoupled oscillator in \cite{gandhi2025conceptual} is intrinsically excitable, whereas the oscillators considered here are intrinsically oscillatory. Whether this distinction is responsible for the two mathematically different SSB mechanisms remains an interesting direction for future work. Whilst we have provided analysis of the symmetrically coupled model \eqref{eq:coupled}, many studies of the eukaryotic cell cycle induce symmetry-breaking by introducing asymmetry into the vector field itself. For instance, the bifurcation analyses in \cite{Dragoi2024} examine the states that may arise by varying a parameter in oscillator $\Omega_1$ but not $\Omega_2$, or by allowing the coupling strengths between the two oscillators to be different. How the geometric structures and mechanisms that we have identified here would persist and change with these symmetry-breaking changes to the vector field is left to future work.



\appendix
\section{Dynamics in the entry chart $\boldsymbol K_1$} \label{appendix:entrychart}
In this Appendix, we outline the proof of Proposition~\ref{prop:entrychart}. The system \eqref{eq:blownupK1} possesses invariant hyperplanes $\left\{ r_1 = 0 \right\}$ and $\left\{ \eps_1=0 \right\}$. To study the full dynamics in $K_1$, we examine the dynamics in these lower-dimensional invariant subspaces, and systematically build up to the full 4D phase space. 

In the invariant subspace $\left\{ r_1=0, \eps_1=0 \right\}$, the dynamics are given by
\begin{equation}  \label{eq:blownupK1r1eps1}
  \begin{split}
      \dot y_1 &= 0, \\
      \dot z_1 &= y_1 + \lambda_1 z_1.
  \end{split}
\end{equation}
This system possesses a line, $\ell_1 = \left\{ r_1 = 0, y_1 \in \mathbb R, z_1 = -\tfrac{y_1}{\lambda_1}, \eps_1 = 0 \right\}$, of equilibria. The eigenvalues at any point in $\ell_1$ are $\lambda_1$ and $0$. The corresponding stable and center subspaces are 
\[ \mathbb E^s(\ell_1) = {\rm span} \begin{bmatrix} 0 \\ 1 \end{bmatrix} \quad {\rm and } \quad \mathbb E^c(\ell_1) = {\rm span} \begin{bmatrix} -\lambda_1 \\ 1 \end{bmatrix}. \]
Thus, the system \eqref{eq:blownupK1r1eps1} has a 1D center manifold in the invariant hyperplane $\left\{ r_1=0, \eps_1=0 \right\}$, with transverse stable fibers.

In the invariant subspace $\left\{ r_1 = 0 \right\}$, the dynamics are given by 
\begin{equation}
  \begin{split}
  	\dot y_1 &= \eps_1 \left( 2y_1 + \left( \lambda_1+\lambda_2 \right) z_1 \right), \\ 
	\dot z_1 &= y_1 + \lambda_1 z_1 + \eps_1 z_1, \\ 
	\dot \eps_1 &= 2 \eps_1^2.
  \end{split}
\end{equation}
Once again, the line $\ell_1$ is a set of equilibria. The stable spectrum consists of the eigenvalue $\lambda_1$ and the center spectrum consists of two zero eigenvalues. The corresponding linear subspaces are 
\[ \mathbb E^s\left( \ell_1 \right) = {\rm span} \begin{bmatrix} 0 \\ 1 \\ 0 \end{bmatrix}, \quad \text{ and } \quad  \mathbb E^c \left( \ell_1 \right) = \operatorname{span} \left\{ \begin{bmatrix} -\lambda_1 \\ 1 \\ 0 \end{bmatrix}, \begin{bmatrix} 0 \\ -\lambda_2 y_1 \\ -\lambda_1^3 \end{bmatrix} \right\}, \]
where the second vector in $\mathbb E^c(\ell_1)$ is a generalized eigenvector. Thus, by the center manifold theorem, there exists a 2D center manifold in the invariant subspace $\left\{ r_1 = 0 \right\}$, with transverse stable fibers. 

Next, the dynamics in the invariant subspace $\{ \eps_1 = 0 \}$ are governed by
\begin{equation}
  \begin{split}
      \dot r_1 &= 0, \\
      \dot y_1 &= 0, \\ 
      \dot z_1 &= y_1 + \lambda_1 z_1 + r_1 H_1.
  \end{split}
\end{equation}
This system possesses a 2D surface of equilibria, $S_1 = \left\{ r_1 \geq 0, y_1 \in \mathbb R, z_1 = -\tfrac{y_1}{\lambda_1} + \mathcal O(r_1), \eps_1 = 0 \right\}$, which contains the line $\ell_1$ of equilibria. The stable spectrum of $S_1$ is $\sigma_s(S_1) = \{ \lambda_1+\mathcal O(r_1) \}$, and the center spectrum of $S_1$ is $\sigma_c(S_1) = \{ 0,0 \}$. The corresponding (generalized) eigenspaces are
\[ \mathbb E^s(S_1) = {\rm span} \begin{bmatrix} 0 \\ 0 \\ 1 \end{bmatrix} \quad {\rm and } \quad \mathbb E^c(S_1) = {\rm span} \left\{ \begin{bmatrix} -\lambda_1+\mathcal O(r_1) \\ 0 \\ -\tfrac{\gamma_1}{\lambda_1} y_1 + \mathcal O(r_1, y_1^2) \end{bmatrix}, \begin{bmatrix} -1+\mathcal O(r_1) \\ -\tfrac{\gamma_1}{\lambda_1} y_1 + \mathcal O(r_1, y_1^2) \\ 0 \end{bmatrix} \right\}. \]
Thus, there is a 2D center manifold of the surface of equilibria, which emanates from $\ell_1$.

For the full 4D system \eqref{eq:blownupK1}, the surface $S_1$ is a manifold of equilibria with negative eigenvalue $\lambda_1 + \mathcal O(r_1)$ and triple zero eigenvalue. The stable and center subspaces are given by  
\[ \mathbb E^s(S_1) = {\rm span}\begin{bmatrix} 0\\0\\1\\0 \end{bmatrix} \quad {\rm and } \quad \mathbb E^c (S_1) = {\rm span} \left\{ \begin{bmatrix} -\lambda_1+\mathcal O(r_1) \\ 0 \\ -\tfrac{\gamma_1}{\lambda_1} y_1 + \mathcal O(r_1, y_1^2) \\ 0 \end{bmatrix}, \begin{bmatrix} -1+\mathcal O(r_1) \\ -\tfrac{\gamma_1}{\lambda_1} y_1 + \mathcal O(r_1, y_1^2) \\ 0 \\ 0 \end{bmatrix}, \begin{bmatrix} 0 \\ \lambda_2 y_1 + \mathcal O(r_1) \\ 0 \\ \lambda_1^3 + \mathcal O(r_1) \end{bmatrix}\right\}. \]
%

\section{Dynamics in the exit chart $\boldsymbol K_3$}  \label{appendix:exitchart}
In this Appendix, we outline the proof of Proposition~\ref{prop:exitchart}.  As with the entry chart, we examine the dynamics of the blown-up system \eqref{eq:blownupK3} in a sequence of lower-dimensional invariant subspaces and systematically build up to the dynamics of the full 4D system. 

In the invariant subspace $\left\{ r_3 = 0, \eps_3 = 0 \right\}$, the dynamics of \eqref{eq:blownupK3} reduce to 
\begin{equation} 
  \begin{split}
    \dot y_3 &= 0, \\ 
    \dot z_3 &= y_3 - \lambda_1 z_3.  
  \end{split}
\end{equation}
This system possesses a line, $\ell_3 = \left\{ r_3=0, y_3 \in \mathbb R, z_3 = \tfrac{y_3}{\lambda_1}, \eps_3 = 0 \right\}$, of equilibria with eigenvalues given by $\lambda = -\lambda_1$ and $\lambda=0$. Recall that $\lambda_1<0$, so that $\ell_3$ has a 1D unstable subspace and 1D center subspace. These subspaces are given by
\[ \mathbb E^u \left( \ell_3 \right) = \operatorname{span} \begin{bmatrix} 0 \\ 1 \end{bmatrix}  \quad \text{ and } \quad 
\mathbb E^c \left( \ell_3 \right) = \operatorname{span} \begin{bmatrix} \lambda_1 \\ 1 \end{bmatrix}.
\]
Thus, there exists a 1D center manifold of the line of equilibria $\ell_3$ in the subspace $\left\{ r_3 = 0, \eps_3 = 0 \right\}$.

In the invariant subspace $\left\{ r_3 = 0 \right\}$, the dynamics of \eqref{eq:blownupK3} simplify to 
\begin{equation}
  \begin{split}
    \dot y_3 &= \eps_3 \left(- 2y_3 +  \left( \lambda_1+\lambda_2 \right) z_3  \right), \\ 
    \dot z_3 &= y_3 - \lambda_1 z_3 - \eps_3 z_3, \\ 
    \dot \eps_3 &= -2\eps_3^2.
  \end{split}
\end{equation}
The equilibria of this system are given by the line $\ell_3$. For a point in $\ell_3$, the eigenvalues are $-\lambda_1,0$, and $0$, and the corresponding subspaces are 
\[ \mathbb E^u\left( \ell_3 \right) = \operatorname{span} \begin{bmatrix} 0 \\ 1 \\ 0 \end{bmatrix} \quad \text{ and } \quad 
\mathbb E^c\left( \ell_3 \right) = \operatorname{span} \left\{ \begin{bmatrix} \lambda_1 \\ 1 \\ 0 \end{bmatrix}, \begin{bmatrix} 0 \\ \lambda_2 y_3 \\ -\lambda_1^3 \end{bmatrix}  \right\},
\]
where the second vector in $\mathbb E^c(\ell_3)$ is a generalized eigenvector. By center manifold theory, we have that there exists a 2D center manifold of the line $\ell_3$ in the hyperplane $\{ r_3=0 \}$.

Next, we examine the invariant subspace $\{ \eps_3 = 0 \}$. The dynamics are given by
\begin{equation}
  \begin{split}
    \dot r_3 &= 0, \\ 
    \dot y_3 &= 0, \\ 
    \dot z_3 &= y_3 - \lambda_1 z_3 + r_3 H_3.
  \end{split}
\end{equation}
This system possesses a surface, $S_3 = \left\{ r_3 \geq 0, y_3 \in \mathbb R, z_3 = \tfrac{y_3}{\lambda_1} + \mathcal O(r_3), \eps_3 = 0 \right\}$, of equilibria. It has a positive eigenvalue $-\lambda_1 + \mathcal O(r_3)$ and a double zero eigenvalue. The associated subspaces are 
\[ \mathbb E^u (S_3) = \operatorname{span} \begin{bmatrix} 0 \\ 0 \\1 \end{bmatrix} \quad \text{ and } \quad
\mathbb E^c (S_3) = \operatorname{span} \left\{ \begin{bmatrix} \lambda_1 + \mathcal O(r_3) \\ 0 \\ \tfrac{\gamma_1}{\lambda_1} y_3 + \mathcal O(r_3,y_3^2) \end{bmatrix}, \begin{bmatrix} -1+\mathcal O(r_3) \\ \tfrac{\gamma_1}{\lambda_1} y_3 + \mathcal O(r_3,y_3^2) \\ 0 \end{bmatrix} \right\}.
\]
Hence, for each point of $S_3$, there exists a 2D center manifold which emanates from the line $\ell_3$.

For the full 4D system \eqref{eq:blownupK3}, the surface $S_3$ is a manifold of equilibria with positive eigenvalue $-\lambda_1 + \mathcal O(r_1)$ and triple zero eigenvalue. The unstable and center subspaces are given by  
\[ \mathbb E^u(S_3) = {\rm span}\begin{bmatrix} 0\\0\\1\\0 \end{bmatrix} \quad {\rm and } \quad \mathbb E^c (S_3) = {\rm span} \left\{ \begin{bmatrix} \lambda_1+\mathcal O(r_3) \\ 0 \\ \tfrac{\gamma_1}{\lambda_1} y_3 + \mathcal O(r_3,y_3^2) \\ 0 \end{bmatrix}, \begin{bmatrix} -1+\mathcal O(r_3) \\ \tfrac{\gamma_1}{\lambda_1} y_3 + \mathcal O(r_3,y_3^2) \\ 0 \\ 0 \end{bmatrix}, \begin{bmatrix} 0 \\ \lambda_2 y_3 + \mathcal O(r_3) \\ 0 \\ \lambda_1^3 + \mathcal O(r_3) \end{bmatrix}\right\}. \]
%

\section{Numerical computation of maximal canards}  \label{appendix:bvp}
In this appendix, we briefly outline the set-up of the two-point boundary value problem used to compute maximal canard solutions of the normal form \eqref{eq:normalform}. The method here closely follows that of \cite{Hasan2017}. Note that the computation of the slow manifolds follows a very similar procedure.

As is usual in the numerical computation of slow manifolds and canard orbits, we first rescale time to the unit interval so that the total integration time appears as an explicit control parameter: 
\begin{equation} \label{eq:rescaledvf}
  \begin{split}
    \dot{\boldsymbol u} &= T \boldsymbol f (\boldsymbol u),
  \end{split}
\end{equation}
where $\boldsymbol u = (u_1,u_2,v_1,v_2)^T$, $\boldsymbol f = (f_1,f_2,\eps g_1, \eps g_2)^T$ is the right-hand-side of \eqref{eq:normalform}, $T$ is the total integration time, and the overdot denotes the derivative with respect to the rescaled time. 
We compute canard orbits as solutions of \eqref{eq:rescaledvf} that satisfy boundary conditions of the form
\begin{equation*}
  \begin{split}
    \boldsymbol u(0) \in \Sigma_0 \quad \text{ and } \quad \boldsymbol u(1) \in \Sigma_1,
  \end{split}
\end{equation*}
where $\Sigma_0$ is a 1D curve on $S_a$ and $\Sigma_1$ is a 1D curve chosen to be transverse to the unstable manifold of the saddle slow manifold $S_{s,\eps}$. For the normal form \eqref{eq:normalform}, we choose the left endpoint data to satisfy
\[ \Sigma_0 = S_a \cap \left\{ u_2-u_1 = 0.3 \right\}, \]
which is a 1D curve on $S_a$ parallel to $\mathcal L_s$ and sufficiently far from the SFN. For the right endpoint, we choose 
\[ \Sigma_1 = \left\{ f_1 = 0 \right\} \cap \left\{ f_2 = \mathcal O(\eps) \right\} \cap \left\{ u_2 = \sigma : \sigma = {\rm constant} \right\}. \]
The first two conditions in $\Sigma_1$ select a particular slow manifold $S_{s,\eps}$ from the exponentially close family of saddle slow manifolds, and the third condition is chosen with $\sigma$ sufficiently far from the SFN.

\section{Bifurcations of the desingularized system} \label{subsec:desingbifn}
As demonstrated in Section~\ref{subsec:strongweaksfns}, the desingularized system \eqref{eq:modeldesingularized} has a rich bifurcation structure. It consists of branches of symmetric ordinary singularities, $E_{\rm sym}$, asymmetric ordinary singularities, $E_{\rm asym}$, symmetric folded singularities, $M_{\rm sym}$, and asymmetric folded singularities, $M_{\rm asym}$. These branches connect, interact, and create/annihilate each other in myriad ways. In this appendix, we give a detailed account of the bifurcation structure that incorporates information about both the ordinary and folded singularities (Fig.~\ref{fig:desingbifn}).

\begin{figure}[h!]
    \centering
    \includegraphics[width=5in]{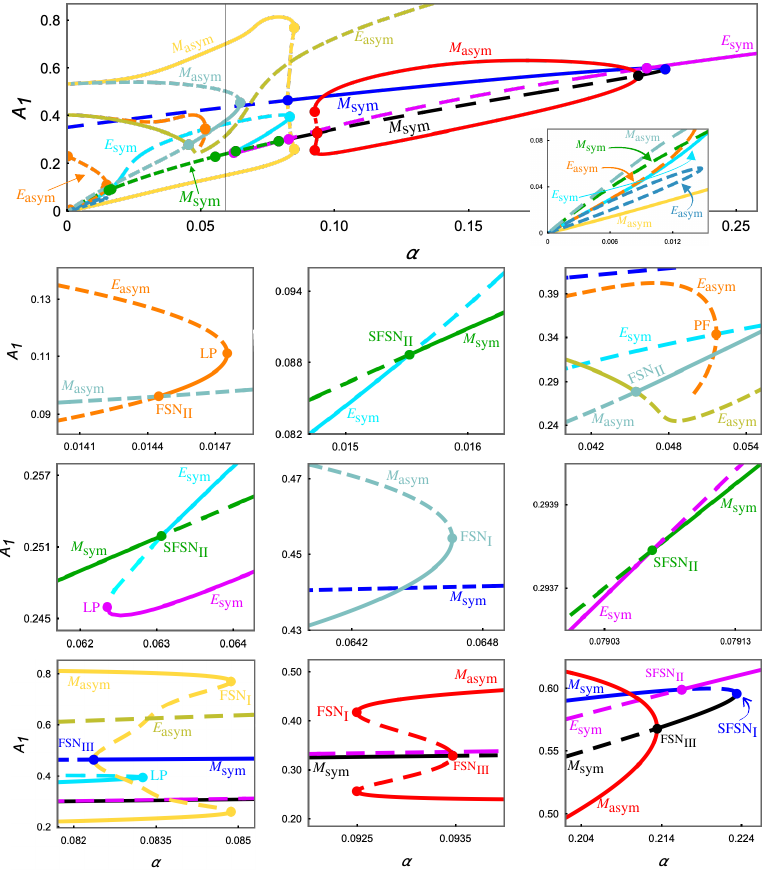}
    \put(-361,404){(a)}
    \put(-361,280){(b)}
    \put(-240,280){(c)}
    \put(-120,280){(d)}
    \put(-361,188){(e)}
    \put(-240,188){(f)}
    \put(-120,188){(g)}
    \put(-361,93){(h)}
    \put(-240,93){(i)}
    \put(-120,93){(j)}
    \caption{Bifurcations of the ordinary and folded singularities of the desingularized system \eqref{eq:modeldesingularized}. 
    The thin black vertical line along $\alpha = 0.0592$ indicates the $\alpha$ value used in Figs.~\ref{fig:modelcriticalmanifold}, \ref{fig:coupledbifnasymmetric}(b), \ref{fig:modelsectors}, and \ref{fig:deconperiodicmmos}.
    There are branches of symmetric and asymmetric ordinary singularities, $E_{\rm sym}$ and $E_{\rm asym}$, respectively, as well as branches of symmetric and asymmetric folded singularities, $M_{\rm sym}$ and $M_{\rm asym}$, respectively. These branches undergo saddle-node, transcritical, and pitchfork bifurcations (in the case of ordinary singularities), and FSN bifurcations of types I, II, and III (in the case of folded singularities). We denote the $\mathbb Z_2$-symmetric FSN bifurcations of types I and II by SFSN$_{\rm I}$ and SFSN$_{\rm II}$, respectively.}
    \label{fig:desingbifn}
\end{figure}
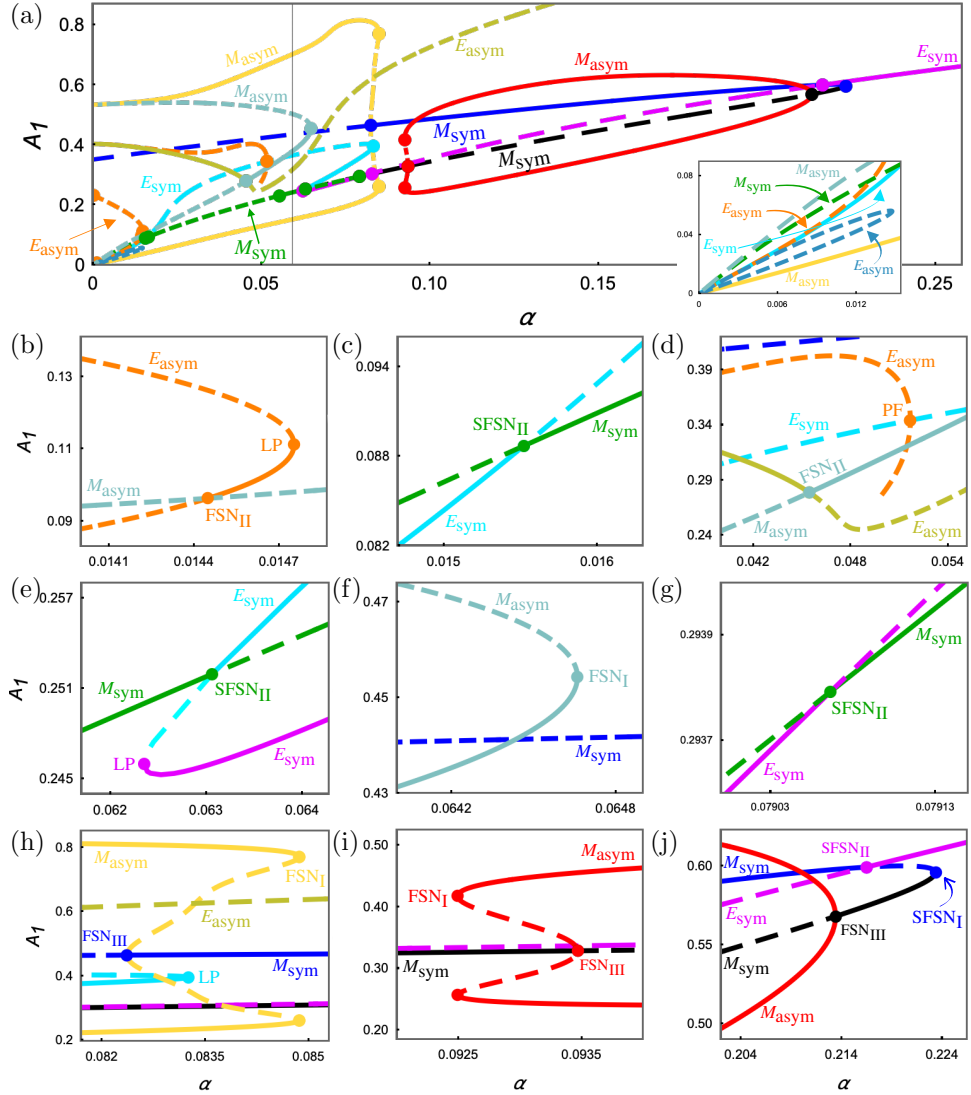

At $\alpha = 0$, there is a degenerate bifurcation in which numerous distinct branches of ordinary and folded singularities appear to emerge from the origin (Fig.~\ref{fig:desingbifn}(a), inset). There are also 6 non-trivial singularities. In order of increasing $A_1$, the non-trivial singularities consist of
(i) an asymmetric equilibrium $E_{\rm asym}$ (orange), 
(ii) a symmetric folded singularity $M_{\rm sym}$ (blue), 
(iii)-(iv) a pair of asymmetric equilibria (olive and orange) which almost coincide, and 
(v)-(vi) a pair of asymmetric folded singularities (steel and yellow) which also almost coincide. 

As the parameter $\alpha$ is increased, the steel blue curve of asymmetric folded singularities $M_{\rm asym}$ crosses the orange curve of asymmetric equilibria $E_{\rm asym}$ in a FSN II bifurcation (Fig.~\ref{fig:desingbifn}(b)). The steel branch transitions from folded saddles to faux folded nodes, and the orange branch transitions from saddles to stable nodes. Moreover, the orange branch of asymmetric equilibria eventually collides with an unstable branch of asymmetric equilibria at a saddle-node bifurcation (Fig.~\ref{fig:desingbifn}(b), LP), after which the orange branch (locally) ceases to exist. 

For larger values of $\alpha$, there are 9 branches of singularities (6 folded and 3 ordinary). Increases in $\alpha$ cause the green $M_{\rm sym}$ branch to pass through the cyan $E_{\rm sym}$ branch in a $\mathbb Z_2$-symmetric FSN II bifurcation (Fig.~\ref{fig:desingbifn}(c), SFSN$_{\rm II}$). The symmetric folded singularities transition from SFSs to SFNs, and the symmetric equilibria transition from stable nodes to saddles. 

Next, the steel $M_{\rm asym}$ and olive $E_{\rm asym}$ branches undergo a FSN II bifurcation (Fig.~\ref{fig:desingbifn}(d)). The asymmetric folded singularities are folded saddles/folded nodes to the left/right of the FSN II, respectively.
Similarly, the asymmetric equilibria are stable nodes/saddles to the left/right of the FSN II. 
Subsequently, with further increases in $\alpha$, the orange $E_{\rm asym}$ branches terminate on the cyan $E_{\rm sym}$ branch at a pitchfork bifurcation. (The lower portion of the orange $E_{\rm asym}$ branch may extend to smaller $\alpha$, however, our numerical continuation codes were unable to compute beyond the values shown.) 

For even larger $\alpha$, there are 8 singularities (6 folded and 2 ordinary). The green branch of symmetric folded singularities and cyan branch of symmetric ordinary singularities undergo a $\mathbb Z_2$-symmetric FSN II bifurcation (Fig.~\ref{fig:desingbifn}(e), SFSN$_{\rm II}$). There, $M_{\rm sym}$ changes from a SFN to a SFS as $\alpha$ is increased, whilst $E_{\rm sym}$ changes from a saddle to a stable node. There is also a nearby saddle-node (LP) bifurcation of the symmetric equilibria where an additional magenta branch is created. 

By increasing $\alpha$ further, we encounter a region where the asymmetric folded singularities undergo a FSN I bifurcation (Fig.~\ref{fig:desingbifn}(f)). More specifically, the upper steel $M_{\rm asym}$ branch of folded saddles and the lower steel $M_{\rm asym}$ branch of folded nodes collide and annihilate each other in a genuine saddle-node bifurcation. 

For larger $\alpha$ still, we find that the green branch of symmetric folded singularities undergoes yet another $\mathbb Z_2$-symmetric FSN II bifurcation (Fig.~\ref{fig:desingbifn}(g), SFSN$_{\rm II}$). This time, the green branch of SFSs coalesces with the magenta branch of stable node equilibria at the SFSN II point. Subsequently, the green branch consists of SFNs and the magenta branch consists of symmetric saddles. 

Further increases in $\alpha$ lead us to a region of the parameter space where several branches of singularities are created and annihilated (Fig.~\ref{fig:desingbifn}(h)). First, the blue $M_{\rm sym}$ branch undergoes a FSN III bifurcation in which the symmetric folded singularity changes type from SFS to SFN, and a pair of (yellow) asymmetric folded saddles bifurcate off the symmetric branch. 
As $\alpha$ is increased, the dashed yellow $M_{\rm asym}$ folded saddles collide with solid yellow $M_{\rm asym}$ folded nodes at a FSN I bifurcation. 
In this region, we also find that the cyan $E_{\rm sym}$ branch undergoes a saddle-node bifurcation. 

In the next $\alpha$ interval (Fig.~\ref{fig:desingbifn}(i)), increases in $\alpha$ result in a pair of FSN I bifurcations. Each FSN I gives rise to a branch of asymmetric folded nodes and asymmetric folded saddles (solid and dashed red branches). The (dashed red) asymmetric folded saddles terminate at a FSN III bifurcation on the black $M_{\rm sym}$ branch. 

The outer red $M_{\rm asym}$ branches of folded nodes terminate at a FSN III bifurcation (Fig.~\ref{fig:desingbifn}(j)), where they merge with the black $M_{\rm sym}$ branch. 
The black SFNs that exist to the right of the FSN III collide with the dashed blue SFSs at a $\mathbb Z_2$-symmetric FSN I bifurcation. 
These blue SFSs also undergo a $\mathbb Z_2$-symmetric FSN II bifurcation with the magenta $E_{\rm sym}$ branch. 
For $\alpha$ values to the left/right of the $\mathbb Z_2$-symmetric FSN II, $M_{\rm sym}$ is a SFN/SFS and $E_{\rm sym}$ is a saddle/stable node. 
Finally, for all $\alpha$ values beyond the FSN I, there is only a single stable symmetric equilibrium.

\section{Dynamics of the symmetrically coupled model} \label{subsec:modeldynamics}
In this appendix, we provide a more global overview of the bifurcation structure of the symmetrically coupled cell model \eqref{eq:coupled} with respect to $\alpha$ (Fig.~\ref{fig:coupledbifn}(a)). The system possesses four distinct types of states: symmetric equilibria, $E_{\rm sym}$, asymmetric equilibria, $E_{\rm asym}$, anti-phase (AP) limit cycles, and broken-symmetry rhythms. 

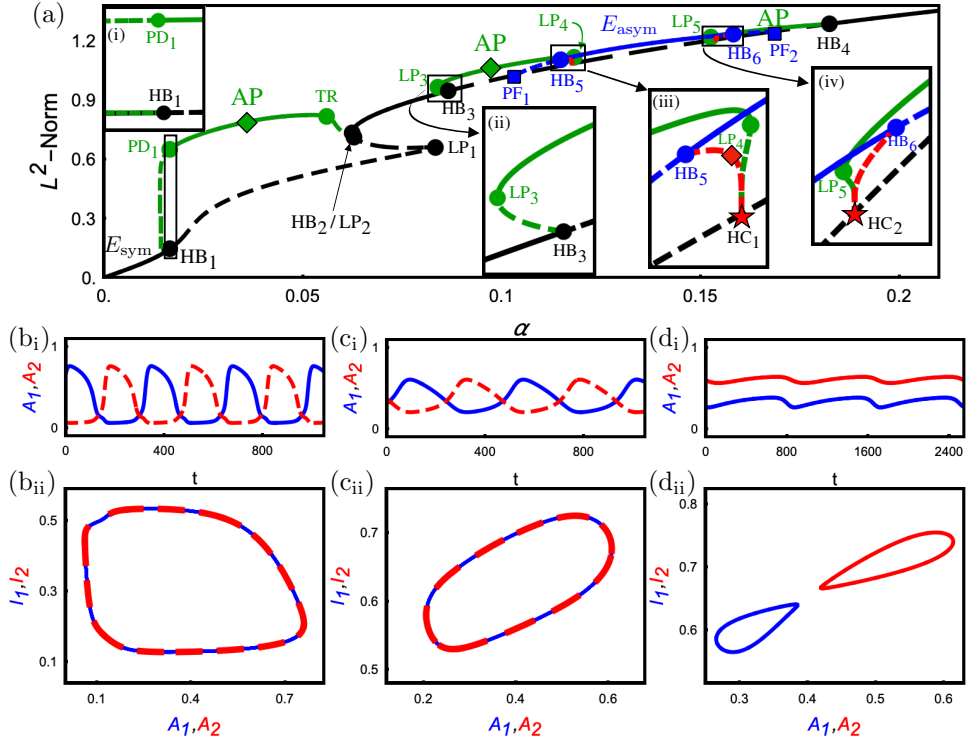
\begin{figure}[h!]
   \centering
   \includegraphics[width=5in]{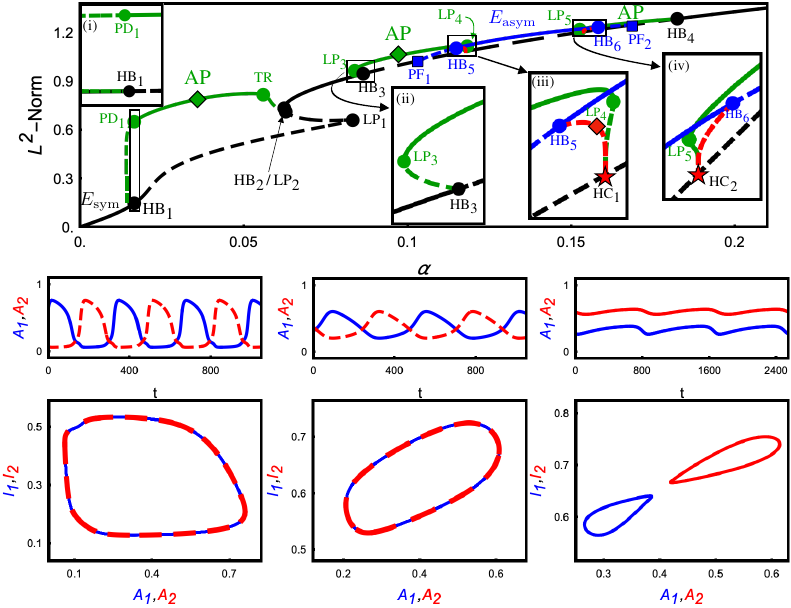}
   \put(-352,270){(a)}
   \put(-361,148){(b$_{\rm i}$)}
   \put(-240,148){(c$_{\rm i}$)}
   \put(-120,148){(d$_{\rm i}$)}
   \put(-361,94){(b$_{\rm ii}$)}
   \put(-240,94){(c$_{\rm ii}$)}
   \put(-120,94){(d$_{\rm ii}$)}
   \caption{Dynamics of the coupled cell model \eqref{eq:coupled} for the standard parameter set. (a) Partial bifurcation diagram showing only the biologically relevant states, i.e., those with non-negative values. Stability/instability is indicated by solid/dashed curves. The symmetric equilibria, $E_{\rm sym}$ (black curves), undergo Hopf bifurcations, which give rise to anti-phase (AP) limit cycles (green curve), and pitchfork bifurcations, which give rise to asymmetric equilibria, $E_{\rm asym}$ (blue curves). The asymmetric equilibria also undergo Hopf bifurcations, which give rise to weak symmetry-breaking rhythms (red curves). The strong symmetry-breaking MMOs that we study in Section~\ref{subsec:modelmmos} exist in the interval between the TR and LP$_2$ bifurcations.
   Remaining panels: dynamics of $\Omega_1$ (solid blue) and $\Omega_2$ (dashed red), corresponding to the diamond markers in (a). 
   (b$_{\rm i}$) and (b$_{\rm ii}$): time series and $(A,I)$ projection of the AP relaxation oscillation attractor for $\alpha = 0.035$. (c$_{\rm i}$) and (c$_{\rm ii}$): time series and $(A,I)$ projection of the AP limit cycle attractor for $\alpha = 0.095$. (d$_{\rm i}$) and (d$_{\rm ii}$): time series and $(A,I)$ projection of the weak symmetry-breaking rhythm for $\alpha = 0.117424$.}
   \label{fig:coupledbifn}
\end{figure}

Starting at $\alpha = 0$, the symmetric equilibria, $E_{\rm sym}$ (black curve), are stable. As $\alpha$ is increased, the $E_{\rm sym}$ branch undergoes a subcritical Hopf bifurcation (HB$_1$) at $\alpha \approx 0.016578$ and becomes unstable. The $E_{\rm sym}$ branch remains unstable until the saddle-node bifurcation (LP$_2$) at $\alpha \approx 0.062352$ is reached\footnote{This is actually the SNIC bifurcation identified in Fig.~\ref{fig:coupledbifnasymmetric}. The periodic branch that terminates on the SNIC exists over an extremely thin $\alpha$ interval and is not shown.}. Then, $E_{\rm sym}$ is stable until there is another subcritical Hopf bifurcation (HB$_3$) at $\alpha \approx 0.086496$, after which $E_{\rm sym}$ becomes unstable. As $\alpha$ is further increased, we find that a pair of asymmetric equilibria, $E_{\rm asym}$ (blue curve), emerge from $E_{\rm sym}$ in a pitchfork bifurcation (PF$_1$) at $\alpha \approx 0.103796$. The asymmetric equilibria cease to exist at another pitchfork bifurcation (PF$_2$) at $\alpha \approx 0.168628$. Next, the symmetric equilibrium regains stability at a supercritical Hopf bifurcation (HB$_4$) at $\alpha \approx 0.182313$, and remains stable for all $\alpha$ to the right of HB$_4$. 

At each of the Hopf bifurcations (black circles) along the $E_{\rm sym}$ branches, a family of AP limit cycles (green curves) emerges. The AP cycles that emerge from the subcritical Hopf bifurcation HB$_1$ are unstable. These become stable in a period-doubling bifurcation (PD$_1$) at $\alpha \approx 0.016554$. On this stable segment of the green AP branch, the AP limit cycles are relaxation oscillations with segments of slow variation interspersed by fast transitions (Fig.~\ref{fig:coupledbifn}(b$_{\rm i}$)). In the projection into the $(A,I)$ phase plane, the two oscillators trace out the same paths (Fig.~\ref{fig:coupledbifn}(b$_{\rm ii}$)). These AP relaxation oscillations are stable until they encounter the torus bifurcation (TR) at $\alpha \approx 0.05598$. For larger $\alpha$, the AP relaxation oscillations are unstable and terminate at the subcritical Hopf bifurcation (HB$_2$) on the branch of symmetric equilibria at $\alpha \approx 0.062821$. There is a very small $\alpha$-interval between PD$_1$ and HB$_1$, i.e., for $0.016554 \lesssim \alpha \lesssim 0.016578$, on which the AP limit cycles and the symmetric equilibria are bistable (Fig.~\ref{fig:coupledbifn}(a), inset (i)).

The AP limit cycles that emerge from the subcritical Hopf bifurcation HB$_3$ are unstable. This unstable branch collides with a stable branch of AP limit cycles at the saddle-node of periodics bifurcation (LP$_3$) at $\alpha \approx 0.083927$. These stable AP limit cycles are approximately sinusoidal (Fig.~\ref{fig:coupledbifn}(c$_{\rm i}$)) and are comparable to the out-of-phase oscillations of the S- and M-modules of the cell cycle \cite{Dragoi2024}. As before, the projection of the AP limit cycle attractor into the $(A,I)$ plane shows that the two oscillators trace out the same paths (Fig.~\ref{fig:coupledbifn}(c$_{\rm ii}$)). We note that there is an interval of bistability between the symmetric equilibria and the AP limit cycles for $\alpha$ values between LP$_3$ and HB$_3$ (Fig.~\ref{fig:coupledbifn}(a), inset (ii)).
As $\alpha$ is increased, the AP branch collides with another unstable AP branch at the saddle-node of periodics (LP$_4$) at $\alpha \approx 0.118126$. 
The unstable AP branch then terminates on the symmetric equilibrium branch in a homoclinic bifurcation (HC$_1$) at $\alpha \approx 0.11766$ (Fig.~\ref{fig:coupledbifn}(a), inset (iii)).
Consequently, there is a small window of $\alpha$ values where the asymmetric equilibrium state and an AP limit cycle are both stable (Fig.~\ref{fig:coupledbifn}(a), inset (iii)).  

The AP limit cycles that emanate from the supercritical Hopf bifurcation HB$_4$ are stable, and they are qualitatively similar to the AP limit cycles presented in Fig.~\ref{fig:coupledbifn}(c$_{\rm i}$) and (c$_{\rm ii}$). This stable AP branch merges with an unstable AP branch at a saddle-node of periodics bifurcation (LP$_5$) at $\alpha \approx 0.152537$. This unstable AP branch terminates at the homoclinic bifurcation (HC$_2$) at $\alpha \approx 0.15374$ (Fig.~\ref{fig:coupledbifn}(a), inset (iv)). We observe that there is a small $\alpha$-interval, enclosed by LP$_5$ and HB$_6$ for which the asymmetric equilibrium state (blue curve) is bistable with the almost sinusoidal AP limit cycles (Fig.~\ref{fig:coupledbifn}(a), inset (iv)).

The asymmetric equilibria (blue curve) that emerge from the pitchfork bifurcation PF$_1$ at $\alpha \approx 0.103796$ are unstable between PF$_1$ and the subcritical Hopf bifurcation (HB$_5$) at $\alpha \approx 0.114814$. The asymmetric equilibrium branch is then stable up to the subcritical Hopf bifurcation (HB$_6$) at $\alpha \approx 0.158204$.

The limit cycles (Fig.~\ref{fig:coupledbifn}(a); dashed, red curves) that emerge from the asymmetric subcritical Hopf bifurcation HB$_5$ are unstable weak symmetry-breaking rhythms. (We use weak symmetry-breaking here to refer to solutions that are small perturbations of a symmetric state.) These weak symmetry-breaking rhythms have scalloped time series profiles (Fig.~\ref{fig:coupledbifn}(d$_{\rm i}$)). The projection of the weak symmetry-breaking rhythm into the $(A,I)$ plane (Fig.~\ref{fig:coupledbifn}(d$_{\rm ii}$)) shows that it is close to a symmetric AP rhythm. The branch of weak symmetry-breaking rhythms terminates at the nearby homoclinic bifurcation HC$_1$ (Fig.~\ref{fig:coupledbifn}(a), inset (iii)). 
Similarly, the limit cycles that emerge from the asymmetric subcritical Hopf bifurcation HB$_6$ are weak symmetry-breaking rhythms that terminate at a nearby homoclinic bifurcation (HC$_2$) at $\alpha \approx 0.15374$ (Fig.~\ref{fig:coupledbifn}, inset (iv)). 

The strong symmetry-breaking MMOs studied in Section~\ref{subsec:modelmmos} exist in the $\alpha$ interval between the torus bifurcation and the SNIC bifurcation (labelled LP$_2$).



\end{document}